\documentclass[a4paper,12pt,reqno]{amsart}
\usepackage{amsfonts}
\usepackage{amsmath,array}
\usepackage{amssymb}
\usepackage{mathrsfs, mathtools}
\usepackage[colorlinks]{hyperref}
 \usepackage{graphicx}
\usepackage{enumerate}
\usepackage[usenames]{color}
\usepackage{comment}

\makeatletter
\newcommand{\setword}[2]{%
  \phantomsection
  #1\def\@currentlabel{\unexpanded{#1}}\label{#2}%
}
\makeatother
\newtheorem{thm}{Theorem}[section]

\newtheorem{lem}[thm]{Lemma}
\newtheorem{prop}[thm]{Proposition}

\newtheorem{remark}[thm]{Remark}
\numberwithin{equation}{section}
\usepackage[english]{babel} 
\usepackage{blindtext}

\theoremstyle{definition}
\newtheorem{definition}[thm]{Definition}

\usepackage[colorlinks]{hyperref}
\usepackage{cite}
\begin{document}

\allowdisplaybreaks 
 \title[Spectral analysis, maximum principles and shape optimization]{Spectral analysis, maximum principles and shape optimization for nonlinear superposition operators of mixed fractional order}

 \author[Yergen Aikyn, Sekhar Ghosh, Vishvesh Kumar, and Michael Ruzhansky]{Yergen Aikyn, Sekhar Ghosh, Vishvesh Kumar, and Michael Ruzhansky}

\address[Yergen Aikyn]{Department of Mathematics: Analysis, Logic and Discrete Mathematics, Ghent University, Ghent, Belgium}
\email{aikynyergen@gmail.com}

\address[ Sekhar Ghosh]{Department of Mathematics, National Institute of Technology Calicut, Kozhikode, Kerala, India - 673601}
\email{sekharghosh1234@gmail.com / sekharghosh@nitc.ac.in}
\address[Vishvesh Kumar]{Department of Mathematical Sciences,
		Indian Institute of Technology (BHU),
		Varanasi, Uttar Pradesh, 221005, India.}
\email{vishveshmishra@gmail.com }

\address[Michael Ruzhansky]{Department of Mathematics: Analysis, Logic and Discrete Mathematics, Ghent University, Ghent, Belgium\newline and \newline
School of Mathematical Sciences, Queen Marry University of London, United Kingdom}
\email{michael.ruzhansky@ugent.be}
\date{}

\begin{abstract}
 The main objective of this paper is to investigate the spectral properties, maximum principles, and shape optimization problems for a broad class of nonlinear ``superposition operators"  defined as continuous superpositions of operators of mixed fractional order, modulated by a signed finite Borel measure on the unit interval. This framework encompasses, as particular cases, mixed local and nonlocal operators such as $-\Delta_p+(-\Delta_p)^s$, finite (possibly infinite) sums of fractional $p$-Laplacians with different orders, as well as operators involving fractional Laplacians with ``wrong" signs.
 
 The main findings, obtained through variational techniques, concern the spectral analysis of the Dirichlet eigenvalue problem associated with general superposition operators with special emphasis on various properties of the first eigenvalue and its corresponding eigenfunction. 

We establish weak and strong maximum principles for positive superposition operators by introducing an appropriate notion of the {\it nonlocal tail} for this class of superposition operators and deriving a logarithmic estimate, both of which are of independent interest. Utilizing these newly developed tools, we further investigate the spectral properties of such superposition operators and prove that the first eigenvalue is isolated and simple. Moreover, we show that the eigenfunctions corresponding to positive eigenvalues are globally bounded and that they change sign when associated with higher eigenvalues. In addition, we demonstrate that the second eigenvalue is well-defined and provide the mountain pass characterization.

Finally, we address shape optimization problems, in particular, the Faber--Krahn inequality associated with the principal frequency associated with the superposition operators.

\end{abstract}


\maketitle

\tableofcontents
\section{Introduction and main results}

Following Lord Rayleigh's celebrated conjecture in The Theory of Sound \cite{R1945}, the investigation of isoperimetric inequalities for eigenvalues of elliptic operators has remained a foundational theme in spectral theory. The conjecture states ``\textit{among all planar domains with a fixed area, the disk uniquely minimizes the first Dirichlet eigenvalue of the Laplacian}". Specifically, suppose $\lambda_1(\Omega)$ is the first eigenvalue of the eigenvalue problem 
\begin{align}
    -\Delta u&=\lambda u \text { in } \Omega,\\ \nonumber
    u&=0 \text { on } \partial\Omega,
\end{align}
where $\Omega \subset\mathbb{R}^2$ is a bounded domain. If $B_r$ is a ball such that $|B_r|=|\Omega|$, then Rayleigh’s conjecture asserts that $\lambda_1(B_r)\leq\lambda_1(\Omega)$. This conjecture was independently proven by Faber \cite{F1923}, through methods of discretization and approximation, and by Krahn \cite{K1925} for $N=2$. Later, Krahn \cite{K1926} extended the result to higher dimensions using the classical isoperimetric inequality and the Coarea formula. This fundamental result is now widely recognized as the Rayleigh-Faber-Krahn inequality. In \cite{DK2008}, it is proved that the equality in the Faber-Krahn inequality holds for $N>2$, only if $\Omega$ is itself a ball, up to a set of measure zero. For general discussions of this topic, we refer to  \cite{PS51, RSS} and references therein.

The main goal of this paper is to study the spectral properties and related shape optimization problems of a nonlinear superposition operator of fractional orders with zero Dirichlet boundary data. Specifically, we will establish the existence and the properties of the principal eigenvalues and eigenfunctions by obtaining a strong minimum/maximum principle and a logarithmic estimate. We, in particular, provide a characterization of the second eigenvalue for the superposition operator. We also investigate the associated shape optimization problem, particularly, we establish the Faber-Krahn inequality for this operator. We are concerned with the following nonlinear superposition operator of fractional operators.
\begin{equation}\label{super operator}
\mathcal{L}_{\mu, p} u:=\int_{[0,1]}(-\Delta)_{p}^{s} u d \mu(s), 
\end{equation}
where the Borel measure $\mu$ is defined as 
\begin{equation}\label{measure}
\mu:=\mu^{+}-\mu^{-}
\end{equation}
with $\mu^{+}$and $\mu^{-}$ being two nonnegative finite Borel measures over $[0,1]$. We define 
\begin{equation}\label{positive part of L}
    \mathcal{L}_{\mu, p}^{+} u:=\int_{[0,1]}(-\Delta)_p^{s} u d \mu^{+}(s),
\end{equation}
and
\begin{equation*}
    \mathcal{L}_{\mu, p}^{-} u:=\int_{[0,1]}(-\Delta)_p^{s} u d \mu^{-}(s).
\end{equation*}
Therefore, from \eqref{measure}, we can decompose the operator $\mathcal{L}_{\mu,p}$ as
\begin{equation*}
    \mathcal{L}_{\mu, p} = \mathcal{L}_{\mu, p}^{+} - \mathcal{L}_{\mu, p}^{-}.
\end{equation*}
Moreover, we assume that there exist $\bar{s} \in(0,1]$ and $\gamma \geq 0$ such that
\begin{align}
    \mu^{+}([\bar{s}, 1])&>0,  \label{measure1}\\
    \mu^{-}([\bar{s}, 1])&=0 \label{measure2} \text{ and }\\
    \mu^{-}([0, \bar{s})) &\leq \gamma \mu^{+}([\bar{s}, 1]).\label{measure3}
\end{align}
Note that the assumption \eqref{measure1}, asserts that there exists $s_{\sharp} \in[\bar{s}, 1]$ such that 
\begin{equation}\label{measure4}
\mu^{+}\left(\left[s_{\sharp}, 1\right]\right)>0.
\end{equation}

We will see later that \( s_\sharp \) also plays the role of a critical exponent. It is important to emphasize that a certain degree of freedom is available in the selection of \( s_\sharp \) as introduced above. Nonetheless, the strength and sharpness of the ensuing results are significantly influenced by this choice: the conclusions become more robust when \( s_\sharp \) is chosen to be ``as large as possible'' under the constraint imposed by condition~\eqref{measure4}. In particular, one may take \( s_\sharp := \bar{s} \) without loss of generality; however, whenever it is feasible to select a larger admissible value of \( s_\sharp \), such a choice leads to both qualitative and quantitative refinements of the obtained results.

An intriguing aspect of the superposition operator introduced in \eqref{super operator} lies in its capacity to encompass a broad class of well-known operators. In particular, it includes the negative $p$-Laplacian $-\Delta_{p}$ when $\mu$ is the Dirac measure concentrated at $1$, the fractional $p$-Laplacian $(-\Delta)_{p}^{s}$ when $\mu$ is the Dirac measure concentrated at a fractional order $s \in (0,1)$, and the so-called mixed-order operator $-\Delta_{p} + \epsilon (-\Delta)_{p}^{s}$ with $\epsilon \in (0,1]$ when $\mu$ is given by the sum of two Dirac measures $\delta_{1} + \epsilon \delta_{s}$, $s \in (0,1)$.  A noteworthy feature  of the operators considered in this work is their capacity to simultaneously encompass nonlinear operators together with an infinite (possibly uncountable) family of fractional operators. Furthermore, certain components of these operators may possess the “wrong sign,” provided that a dominant contribution, typically associated with terms of higher fractional order, ensures overall control of the operator.

Beyond their theoretical importance, such superposition operators naturally arise in various applied contexts, including anomalous diffusion, population dynamics, and mathematical biology, particularly in models involving Gaussian processes and L\'evy flights. For further details and related discussions, we refer the reader to \cite{DV21, DPSV25}.

There has been substantial research on the nonlinear eigenvalue problems and related inequalities involving different local, nonlocal and mixed local-nonlocal elliptic operators such as the $p$-Laplacian, fractional $p$-Laplacian, mixed local and nonlocal $p$-Laplacian, etc. The nonlinear eigenvalue involving $p$-Laplacian emerged from the celebrated study due to Anane \cite {A87}, Bhattacharya \cite{B89}, and Lindqvist \cite{L90} (see also \cite{L17}). Anane \cite {A87} proved that the first eigenvalue of the following nonlinear Dirichlet eigenvalue problem is simple and isolated for $1<p<\infty$:
\begin{align}\label{anane p}
     -\text{div}(|\nabla u|^{p-2}\nabla u)&=\lambda m(x)|u|^{p-2}u \text { in } \Omega, \nonumber \\ 
u&=0 \text { in } \mathbb{R}^{N} \backslash \Omega,
\end{align}
where $\Omega\subset\mathbb{R}^N$ is a bounded domain with boundary $\partial\Omega$ of class $C^{2, \alpha}, 0<\alpha<1,$ and $m\in L^{\infty}(\Omega)$ with $\text{meas}\{x\in\Omega : m(x)>0\}\neq 0$. We mention that for $p=2$ with $m\in C(\bar{\Omega})$ such that $m(x_0)>0$, for some $x_0\in\Omega$, the existence of a simple and principal eigenvalue is guaranteed by Hess and Kato \cite{HK80}, whereas for $m\equiv 1$, it follows from the well-known Krein-Rutman theorem. The result of Anane \cite{A87} was improved by Bhattacharya \cite{B89} for $m(x)\equiv1$ in a bounded domain of class $C^{2}$ for the following eigenvalue problem,
\begin{align}\label{lind p}
     -\text{div}(|\nabla u|^{p-2}\nabla u)&=\lambda|u|^{p-2}u \text { in } \Omega, \nonumber \\ 
u&=0 \text { on } \partial\Omega.
\end{align}
Lindqvist \cite{L90, L92} proved that the first eigenvalue $\lambda_1>0$ to the problem \eqref{lind p} is simple for any bounded domain $\Omega\subset\mathbb{R}^N$ and $m(x)\equiv1$. Moreover, $\lambda_1$ is principal and it coincides with the minimum of the following Rayleigh quotient:
$$\mathcal{R}=\min\frac{\int_\Omega |\nabla u|^pdx}{\int_\Omega |u|^pdx}.$$
Garcia Azorero and Peral Alonso \cite{GP87} established the existence of a sequence of positive eigenvalues employing the Lusternick-Schnirelmann theory \cite{Szulkin}. L\^e \cite{L06} presented a detailed review on the existence and properties of eigenvalues and eigenfunctions of the $p$-Laplacian prescribed with different boundary conditions. For further details, we refer to \cite{AH95, DR99, BD06, Cuesta, MR02, DKN97, FL2010} and the references therein.

The following nonlocal extension of the problem \eqref{lind p} has also received substantial interest in characterizing eigenvalues and eigenfunctions.
\begin{align}\label{frac p}
     (-\Delta_p)^s u&=\lambda|u|^{p-2}u \text { in } \Omega, \nonumber \\ 
u&=0 \text { in } \mathbb{R}^{N} \backslash \Omega,
\end{align}
where $(-\Delta_p)^s u$ represents the standard fractional $p$-Laplacian. For $p=2$, the problem \eqref{frac p} was investigated by Servadei and Valdinoci \cite{SV13, SV15}, where the authors obtained the existence of a sequence of positive eigenvalues of Lusternik-Schnirelmann type. They proved that the first eigenvalue coincides with the Rayleigh quotient and is simple and isolated. Moreover, all eigenfunctions are bounded, and all eigenfunctions other than the principal eigenfunction must be sign-changing. Further, they obtained the continuity of eigenfunctions in the sense of viscosity solutions and a strong maximum principle to conclude that the first eigenfunction must be of constant sign.

In the nonlinear case, Lindgren and Lindqvist \cite{LL14} introduced the fractional Rayleigh quotient for a generalised nonlocal eigenvalue problem associated with fractional $p$-Laplacian of fraction order $\alpha-\frac{N}{p}$ such that $N<\alpha p<N+p$ and $p\geq 2$. They used the notion of viscosity solutions to prove that the first eigenfunction is strictly positive, which is a minimiser for the Rayleigh quotient. Moreover, the first eigenvalue is simple and isolated. They also characterized the limiting case as $p\rightarrow\infty$ and studied the corresponding fractional $\infty$-Laplacian eigenvalue problem extending the works of Juutinen et al. \cite{JLM1999} (see also \cite{JL2005}). Later, Franzina and Palatucci \cite{FP14} completemented the study of \cite{LL14} for the $(s,p)$-eigenvalues in a bounded domain for $s\in(0,1)$ and $p\in(1,\infty)$. They proved that any positive eigenvalue corresponding to positive eigenfunctions must be the Rayleigh quotient, and any $(s,p)$-eigenfunctions are globally bounded. The limiting case $p\rightarrow\infty$, known as the nonlocal Cheeger problem, was addressed by Brasco et al. \cite{BLP2014}. The existence and variational characterization of the second eigenvalue to the problem \eqref{frac p} is established in Brasco and Parini \cite{BP16}.

Recently, eigenvalue problems having both local and nonlocal effects have drawn significant attention. Del Pezzo et al. \cite{DFR19} studied similar properties of eigenvalues and eigenfunctions for the mixed local and nonlocal operator, $-\Delta_p -\Delta_{J,p}$, where 
$$\Delta_{J,p}u:=-2\int_{\mathbb{R}^N}J(x-y)|u(x)-u(y)|^{p-2}(u(x)-u(y))dy.$$
The kernel $J:\mathbb{R}^N\rightarrow \mathbb{R}$ is a nonnegative, symmetric, radial, continuous function such that $J$ has compact support, $J(0)>0$ and $\int_{\mathbb{R}^N}J(x)dx=1$. In \cite{DFR19}, the authors proved that there exists a sequence of eigenvalues $(\lambda_n)$ such that $\lambda_n\rightarrow\infty$ to the following problem,
\begin{align}\label{mixed p}
     -\Delta_p u -\Delta_{J,p} u&=\lambda|u|^{p-2}u \text { in } \Omega, \nonumber \\ 
u&=0 \text { in } \mathbb{R}^{N} \backslash \Omega.
\end{align}
Each eigenfunction $\phi\in C^{1,\alpha}(\bar{\Omega})$ for some $\alpha\in(0,1)$. In particular, the first eigenvalue is simple and isolated. They used the equivalence of eigenfunctions to the viscosity solutions to obtain the regularity of eigenfunctions and   the strict positivity of the first eigenfunction. Palatucci and Piccinini \cite{PP26} proved the existence of eigenfunction for $-\Delta_p +(-\Delta_p)^s$ operator for $p\in(1,\infty)$ and $s\in(0,1),$ and showed that any eigenfunction is globally bounded. In particular, the eigenvalue corresponding to a positive eigenfunction coincides with the mixed Rayleigh quotient. Goel and Sreenadh \cite{GS19} proved the existence and characterization of the second eigenvalue for the mixed local and nonlocal operator: $-\Delta_p +(-\Delta_p)^s$. It is noteworthy to mention the notion of `\textit{nonlocal tail}' \cite{DKP14, DKP16}, which plays a crucial role in obtaining the existence and regularity of solutions to problems involving nonlocal operators. On the other hand, for obtaining regularity of solutions and shape optimization, the symmetrization principle proved in Frank-Seiringer \cite{FS-2008}, which generalizes the classical Polya-Szeg\"o inequality \cite{PS51} to the nonlocal case, plays a significant role. For further details on the development related to eigenvalue problems and the regularity of solutions to problems involving nonlocal operators, we refer to \cite{BDD22, I23, BDVV, BM25, LG24, LGG25, GK22, SVWZ, DM2022, GK25, GKR24, FL10} and the references therein.

We now turn our attention to the Dirichlet eigenvalue problem involving the superposition operator $\mathcal{L}_{\mu, p}$:
\begin{equation}\label{Eig problem 00}
    \begin{cases}\mathcal{L}_{\mu, p} u=\lambda |u|^{p-2}u & \text { in } \Omega,  \\ 
u=0 & \text { in } \mathbb{R}^{N} \backslash \Omega,\end{cases}
\end{equation}
where $\Omega$ is an open, bounded subset of $\mathbb{R}^{N}$ $(N \geq 2)$ with Lipschitz boundary.

We recall that $\lambda \in \mathbb{R}$ is called $(s, \mu)$-eigenvalue of \eqref{Eig problem 00} if there exists a nontrivial (weak) solution $u \in X_{p}(\Omega)\setminus\{0\}$ of \eqref{Eig problem 00}. Correspondingly,  $u \in X_{p}(\Omega)\setminus\{0\}$ is called an $(s, \mu)$-eigenfunction associated with the eigenvalue $\lambda$. The spectrum of $\mathcal{L}_{ \mu, p}$ denoted by $\sigma(s, \mu)$ is set of all $(s, \mu)$-eigenvalues of \eqref{Eig problem 00}.

The linear case, $p=2$, for the problem \eqref{Eig problem 00} was investigated in Dipierro et al. \cite{DPSV25}. They proved under the assumptions \eqref{measure1}--\eqref{measure4} that there exists a positive eigenvalue $\lambda_{1, \mu}(\Omega)>0$ to the following problem: 
\begin{equation}\label{Eigen 2}
    \begin{cases}\mathcal{L}_{\mu, 2} u=\lambda u & \text { in } \Omega,  \\ 
u=0 & \text { in } \mathbb{R}^{N} \backslash \Omega.\end{cases}
\end{equation}
Moreover, $\lambda_{1, \mu}(\Omega)$ is attained in the minimization problem:
\begin{equation}\label{first eigenvalue 1}
\lambda_{1, \mu}(\Omega):=\min _{u \in X(\Omega) \backslash\{0\}} \frac{\int_{[0,1]}[u]_{s}^{2} d \mu^{+}(s)-\int_{[0, \bar{s})}[u]_{s}^{2} d \mu^{-}(s)}{\int_{\Omega}|u(x)|^{2} d x}, 
\end{equation}
where $X(\Omega):=X_2(\Omega)$ is fractional Sobolev space defined in Section \ref{sec2}. 
In particular, for a sufficiently small $\bar{\gamma}:=\bar{\gamma}(N,R,s)>0$ such that
\begin{equation}\label{measure5}
\mu^{-}((0, \bar{s})) \leq \bar{\gamma} \delta \mu^{+}([\bar{s}, 1-\delta]),
\end{equation}
for some $\delta \in(0,1-\bar{s}]$, every eigenfunction corresponding to the first eigenvalue $\lambda_{1, \mu}(\Omega)$ is of fixed sign. Recently, the spectral analysis has been studied by Dipierro et al. \cite{DPSV25-1} for the operator $\mathcal{L}_{\mu,2}$. Specifically, they obtained the existence of a sequence of eigenvalues $(\lambda_{k})$ to problem \eqref{Eigen 2}  such that 
$$0<\lambda_{1} \leq \lambda_{2} \leq \cdots \leq\lambda_{k} \leq \cdots~\text { and  }~\lim _{k \rightarrow+\infty} \lambda_{k}=+\infty.$$
Moreover, for each $k\in\mathbb{N}$, the eigenpair $(\lambda_k, e_k)$ can be explicitly identified as
\begin{equation}\label{seq of eigvalues}
\lambda_{k+1}=\min _{u \in E_{k+1} \backslash\{0\}} \frac{\int_{[0,1]}[u]_{s}^{2} d \mu^{+}(s)-\int_{[0, \bar{s})}[u]_{s}^{2} d \mu^{-}(s)}{\|u\|_{L^{2}(\Omega)}^{2}}, 
\end{equation}
and the eigenfunction $e_{k+1} \in E_{k+1}$ as a minimizer of \eqref{seq of eigvalues}, where $E_{1}:=X_{2}(\Omega)$ and, for all $k \geq 1$,
\begin{align*}
E_{k+1}:= & \{u \in X_{2}(\Omega) \text { s.t. for all } j=1, \ldots, k \\
& \left.\quad \int_{[0,1]} c_{N, s} \iint_{\mathbb{R}^{2 N}} \frac{(u(x)-u(y))\left(e_{j}(x)-e_{j}(y)\right)}{|x-y|^{N+2 s}} d x d y d \mu(s)=0\right\}.
\end{align*}
Further, for each $k\in\mathbb{N}$, the eigenvalue $\lambda_{k}$ has finite multiplicity and the sequence of eigenfunctions $(e_{k})$ provides an orthonormal basis for $L^{2}(\Omega)$ as well as $X_{2}(\Omega)$.

Among others, the study due to \cite{DPSV25, DPSV25-1}, stems the motivation for investigating the spectral properties of the nonlinear superposition operator $\mathcal{L}_{\mu, p}$, as in \eqref{super operator}. There are significant difficulties in studying the eigenvalue problems involving the operator $\mathcal{L}_{\mu, p}$ as it is nonlinear and sign-changing. We employ an appropriate variational technique, combined with the Lusternik-Schnirelmann theory, to establish the existence of a sequence of positive eigenvalues, with the first eigenvalue being principal. We now state our first main results exhibiting the existence and properties of eigenpairs of the problem \eqref{Eig problem 00}. We emphasize here that this result is new for signed measure $\mu$ as well as for the positive measure $\mu^+$ for the case when $p \neq 2.$
\begin{thm}\label{principal eigenvalue P 0}
   Let $\Omega \subset \mathbb{R}^{N}$ be a bounded domain with Lipschitz boundary, and let $\mu$ satisfy \eqref{measure1}--\eqref{measure3}. Let $s_{\sharp}$ be as in \eqref{measure4}, and assume $1 < p < N/s_{\sharp}$. Then there exists a constant $\gamma_{0} > 0$, depending only on $N$, $\Omega$, and $p$, such that, for  $\gamma \in [0, \gamma_{0}]$, the statements below concerning the eigenvalues and eigenfunctions of problem \eqref{Eig problem 00} associated with $\mathcal{L}_{\mu, p}$ hold.
\begin{itemize}
    \item[(i)] The first eigenvalue $\lambda_{1, \mu}(\Omega)$ is given by
\begin{equation}\label{first eigenvalue intro}
\lambda_{1, \mu}(\Omega):=\inf_{u \in X_{p}(\Omega) \backslash\{0\}} \frac{\int_{[0,1]}[u]_{s,p}^{p} d \mu^{+}(s)-\int_{[0, \bar{s})}[u]_{s,p}^{p} d \mu^{-}(s)}{\int_{\Omega}|u|^{p} d x} . 
\end{equation}
\item[(ii)] There exists a  function $e_{1, \mu} \in X_p(\Omega)$, an eigenfunction corresponding to the eigenvalue $\lambda_{1, \mu}(\Omega),$ which attains the minimum in \eqref{first eigenvalue intro}.

\item[(iii)] The set of eigenvalues of the problem \eqref{Eig problem 00} consists of a sequence $(\lambda_{n, \mu})$ with
  \begin{align}
      0<\lambda_{1, \mu}\leq \lambda_{2, \mu} \leq \ldots \leq \lambda_{n, \mu} \leq \lambda_{n+1, \mu} \leq \ldots~\text{and }&
      \lambda_{n, \mu} \rightarrow \infty\,\,\,\text{as}\,\,n \rightarrow \infty.
  \end{align}
  \item [(iv)] In addition, if  $\mu$ satisfies \eqref{measure5}, then every eigenfunction corresponding to the eigenvalue $\lambda_{1, \mu}(\Omega)$ in \eqref{first eigenvalue intro} does not change sign.
\end{itemize}
   
\end{thm} 
 We would like to point out that Theorem \ref{principal eigenvalue P 0} is completely new in several interesting and genuine examples of operator $\mathcal{L}_{\mu, p}$ by using appropriately measure $\mu^+$ and $\mu^-.$ We list some of them as follows:

\begin{itemize}
    \item[(i)] The mixed local-nonlocal operator with the wrong sign is given by 
    \begin{equation*}
    \mathcal{L}_{\mu, p}=-\Delta_{p}-\alpha(-\Delta)_{p}^{s}.
\end{equation*}
    This obtained by choosing $s \in(0,1)$ and $\mu:=\delta_{1}-\alpha \delta_{s}$, with a small constant $\alpha.$

    \item[(ii)] Let us take $\mu:= \delta_1+\delta_{s_1}-\alpha \delta_{s_2}$ for $1>s_1>s_2>0.$ Then the mixed local-nonlocal operator with the wrong sign given by 
$$\mathcal{L}_{\mu, p}=-\Delta_{p} + (-\Delta_{p})^{s_1}-\alpha (-\Delta_{p})^{s_2}$$
\end{itemize}

\begin{itemize}
    \item[(iii)] If $0\leq ...<s_2<s_1<s_0\leq 1$ and
$\mu:=\sum_{k=1}^{+\infty} c_{k} \delta_{s_{k}}$  with  $\sum_{k=0}^{+\infty} a_k \in (0, +\infty)$ and the conditions such that
\begin{align*}
& c_{k}>0 \text { for all } k \in\{0, \ldots, \bar{k}\} \text { and } \sum_{k=\bar{k}+1}^{+\infty} c_{k} \leq \gamma \sum_{k=0}^{\bar{k}} c_{k} \\
& \text { for some } \bar{k} \in \mathbb{N} \text { and } \gamma \geq 0. 
\end{align*}
In addition to this, we also assume that there exists $\bar{\gamma}> 0,$ $\delta \in (0, 1-s_0]$ and $\bar{k} \in \mathbb{N}$  such that 
    \begin{equation} 
        \sum_{k=\bar{k}+1}^{+\infty} a_k \leq \bar{\kappa} \delta \sum_{k=1}^{\bar{k}} a_k.
    \end{equation}
Then the operator in \eqref{super operator} takes the form
\begin{equation*}
  \mathcal{L}_{\mu, p}(u)=\sum_{k=1}^{+\infty} c_{k}(-\Delta)_p^{s_{k}} u. 
\end{equation*}
    \item[(iv)] Let $f \not\equiv 0 $ be a measurable function such that $f \geq 0$ in $ (s_\sharp, 1),$  $\int_{s_\sharp}^1 f(s)\, ds >0,$ and
    \begin{align*}
    \int_0^{s_\sharp} \max\{0, -f(s)\}\, ds \leq \gamma \int_{s_\sharp}^1 f(s)\, ds. \nonumber
        \end{align*}
        Additionally, we assume that there exist $\bar{\kappa}>0$ and $\delta \in (0, 1-s_\sharp]$ such that 
        \begin{equation*} 
            \int_0^{s_\sharp} \max\{0, -f(s)\}\, ds \leq \bar{\gamma} \delta \int_{s_\sharp}^{1-\delta} f(s)\, ds.
        \end{equation*}
      By defining $d\mu(s):=f(s) ds,$ the corresponding operator is given by
    $$\mathcal{L}_{p, \mu}= \int_0^1 f(s) (-\Delta_p)^s u\, ds.$$
    
\end{itemize}

One may attempt to establish further properties of the eigenvalues of the eigenvalue problem~\eqref{Eig problem 00}. 
However, the proofs of properties such as the isolatedness of the principal eigenvalue and the sign-changing behavior of eigenfunctions corresponding to eigenvalues other than the principal one rely essentially on the maximum principle for the nonlinear operator \( \mathcal{L}_{\mu,p} \). 
Unfortunately, the maximum principle for \( \mathcal{L}_{\mu,p} \) fails to hold even in the linear case \( p = 2 \); we refer to~\cite[Appendix A]{DPSV25-1} for a concrete counterexample. 
The failure of the maximum principle arises from the fact that the general measure \( \mu \) may change sign. 
Indeed, a classical result by Bony, Courr\'ege, and Priouret (see~\cite[Theorem~3, p.~391]{BCP68}) asserts that, for linear translation-invariant operators, the maximum principle holds if and only if the measure \( \mu \) appearing in~\eqref{super operator} has a constant sign, that is, it holds precisely when either \( \mu^{+} \equiv 0 \) or \( \mu^{-} \equiv 0 \) in~\eqref{measure}. Therefore, it is reasonable to study the maximum principle for the operator $\mathcal{L}_{\mu,p}^{+}$ defined as 
\begin{equation*}
\mathcal{L}_{\mu, p}^{+} u:=\int_{[0,1]}(-\Delta)_p^{s} u d \mu^{+}(s).
\end{equation*} 
In fact, the rest of the results discussed below will be devoted to the study of the superposition operator $\mathcal{L}_{\mu, p}^{+}.$ 
The next two maximum principles allow us to investigate further properties of the spectrum of the operator $\mathcal{L}_{\mu,p}^{+}$. The following weak maximum principle is the extension of the case $p=2$ discussed in \cite{DPSV25-1}, for general values of $p.$
\begin{thm} \label{wmpinto}
    Let $\Omega \subset \mathbb{R}^N$ be an open subset with Lipschitz boundary. We assume that $\mu^+$ satisfies \eqref{measure1} and $1<p<\frac{N}{s_\sharp}$, where $s_\sharp$ is defined by \eqref{measure4}. Let $u \in X_p(\Omega)$ be such that $\mathcal{L}_{\mu, p}^+ u \geq 0$ in $\Omega$ in the weak sense and $u \geq 0$ a.e. in $\mathbb{R}^N \backslash \Omega.$ Then, $u \geq 0$ a.e. in $\Omega.$
\end{thm}
Now, we state the following strong minimum principle for the superposition operators $\mathcal{L}_{\mu, p}^+.$
\begin{thm}\label{smpintro} 
Let $\Omega \subset {\mathbb{R}^N}$ be an open, connected, and bounded subset.
Let $\mu=\mu^{+}$ satisfy \eqref{measure1} and $s_{\sharp}$ be as in \eqref{measure4}.  Assume that $u \in X_{p}(\Omega)$ is a weak supersolution of \begin{align}
\mathcal{L}_{\mu,p}^{+}  u=0~\text{in}~\Omega,\nonumber\\
u=0~\text{in}~{{\mathbb{R}^N}}\setminus\Omega,
\end{align}  such that  $u \not\equiv 0$ in $\Omega.$ Then $u>0$ a.e. in $\Omega$.	
\end{thm}
\begin{remark} In the statements of Theorem~\ref{wmpinto} and Theorem~\ref{smpintro}, the condition \( u \geq 0 \) in \( \mathbb{R}^N \setminus \Omega \) should be interpreted as \( u \geq 0 \) on \( \partial \Omega \) whenever \( \mu^{+}((0,1)) = 0. \)
\end{remark}

Next, we investigate the spectral properties of the operator $\mathcal{L}_{\mu,p}^{+}$. More specifically, we study the following eigenvalue problem:
\begin{equation}\label{Eig problem 000}
    \begin{cases}\mathcal{L}_{\mu,p}^{+} u=\lambda |u|^{p-2}u & \text { in } \Omega,  \\ 
u=0 & \text { in } \mathbb{R}^{N} \backslash \Omega .\end{cases}
\end{equation}

 In addition to the properties established in Theorem \ref{principal eigenvalue P 0}, we prove the strict positivity of the first eigenfunction, as well as the simplicity and isolatedness of the first eigenvalue. Furthermore, we show that every higher eigenfunction necessarily changes sign. We now combine a particular case of Theorem \ref{principal eigenvalue P 0} corresponding to a positive measure $\mu^+$ with several new spectral properties of $\mathcal{L}_{\mu,p}^{+}$ established in this paper, to present the following collection of results concerning the spectral properties of the eigenvalue problem \eqref{Eig problem 000} in a unified form for the convenience of the reader.
\begin{thm}\label{first eigenvalue exist + int}
  Let $\Omega \subset \mathbb{R}^{N}$ be a bounded domain with Lipschitz boundary, and let $\mu=\mu^+$ satisfy \eqref{measure1}. Let $s_{\sharp}$ be as in \eqref{measure4}, and assume $1 < p < N/s_{\sharp}$. Then the statements below concerning the eigenvalues and eigenfunctions of problem \eqref{Eig problem 000} associated with $\mathcal{L}_{\mu, p}^+$ hold.
\begin{enumerate}[(i)]
    \item The first eigenvalue $\lambda_{1, \mu^+}(\Omega)$ is given by
\begin{equation}\label{first eigenvalue+ int}
\lambda_{1, \mu^+}(\Omega):=\inf_{u \in X_{p}(\Omega) \backslash\{0\}} \frac{\int_{[0,1]}[u]_{s,p}^{p} d \mu^{+}(s)}{\int_{\Omega}|u|^{p} d x} . 
\end{equation}
\item There exists a  function $e_{1, \mu^+} \in X_p(\Omega)$, an eigenfunction corresponding to the eigenvalue $\lambda_{1, \mu^+}(\Omega)$ which attains the minimum in \eqref{first eigenvalue+ int}.

\item The set of eigenvalues of the problem \eqref{Eig problem 000} consists of a sequence $(\lambda_{n, \mu^+})$ with
  \begin{align}\label{seq int}
      0<\lambda_{1, \mu^+}\leq \lambda_{2, \mu^+} \leq \ldots \leq \lambda_{n, \mu^+} \leq \lambda_{n+1, \mu^+} \leq \ldots~\text{and }&
      \lambda_{n, \mu^+} \rightarrow \infty\,\,\,\text{as}\,\,n \rightarrow \infty.
  \end{align}

  \item Every eigenfunction corresponding to the eigenvalue $\lambda_{1, \mu^+}(\Omega)$ in \eqref{first eigenvalue+ int} does not change sign and $\lambda_{1, \mu^+}(\Omega)$ is simple.
    \item The $\sigma(s, \mu^+)$-spectrum of $\mathcal{L}_{\mu, p}^+$, that is, the set of $(s, \mu^+)$-eigenvalues of \eqref{first eigenvalue+ int}, is a closed set.
\item   Let $u \geq 0$ in $\Omega$ be an eigenfunction of \eqref{Eig problem 000} associated with an eigenvalue $\lambda > 0$. Then $u > 0$ in $\Omega$.

\item Let $v$ be an eigenfunction of \eqref{Eig problem 000} associated to an eigenvalue $\lambda>\lambda_{1, \mu^+}(\Omega)$. Then $v$ must be sign-changing.

\item Let $v$ be an eigenfunction of \eqref{Eig problem 000} associated to an eigenvalue $\lambda\neq \lambda_{1, \mu^+}(\Omega)$. Then there is a positive constant $C$ independent of $v$ such that  
\begin{equation*}
    \lambda \geq C(N, s_{\sharp},p)\left|\Omega_{+}\right|^{-\frac{ps_{\sharp}}{N}} \text { and  } \lambda \geq C(N, s_{\sharp},p)\left|\Omega_{-}\right|^{-\frac{ps_{\sharp}}{N}} ,
\end{equation*}
where $\Omega_{+}:=\{x \in \Omega:\, v>0\}$ and $\Omega_{-}:=\{x \in \Omega:\,v<0\}$.

\item   The first eigenvalue $\lambda_{1, \mu^+}$ of the problem \eqref{Eig problem 000} is isolated.

\item All eigenfunctions for positive eigenvalues $u \in X_p(\Omega)$ of \eqref{Eig problem 000} are globally bounded, that is, $u \in L^\infty(\mathbb{R}^N).$
\end{enumerate}
   \end{thm}

 We briefly highlight the novelty of Theorem \ref{first eigenvalue exist + int} in comparison with the existing literature. To the best of our knowledge, the theorem is entirely new for $p\neq 2$ and for a general positive measure $\mu^+$. The properties $(i)$--$(iv)$ were previously studied in \cite{DPSV25, DPSV25-1} only in the Hilbertian case $p=2$, whereas the remaining properties $(v)$-$(x)$ appear to be new even in this case.

Next, we turn our attention to the following shape optimization problem:
\begin{equation} \label{optfk}
    \inf\{\lambda_{1, \mu^+}(\Omega) :\, |\Omega| = \rho \},
\end{equation}
where $\rho > 0$ is fixed. A solution to this problem, in the case where $\mathcal{L}_{\mu,p}^+ = -\Delta$, is provided by the classical Faber–Krahn inequality (see \cite{F1923, K1925, Polya55}), which asserts that among all domains of a given measure, the Euclidean ball minimizes the first Dirichlet eigenvalue of the Laplacian.  
This fundamental result has been extended to the $p$-Laplacian by several authors; we refer to \cite{B99, CMT2015, BDV2015, BH17, CV2019} and the references therein for detailed discussions. In the nonlocal setting, the Faber–Krahn inequality for the fractional $p$-Laplacian was established by Brasco, Lindgren, and Parini \cite{BLP2014}, they solved the optimization problem \eqref{optfk} for the principal frequency of the fractional Dirichlet $p$-Laplacian. Further refinements, including geometric analyses of specific domains such as triangles and quadrilaterals, were later presented in \cite{C2017}.  For the mixed local–nonlocal operator $\mathcal{L}_{\mu,p}^+ = -\Delta_p + (-\Delta_p)^s$, a complete solution to the optimization problem \eqref{optfk} was recently achieved in \cite{BDVV2, BDVV3}. Moreover, a version of the Faber–Krahn inequality was obtained in \cite{GS19} for mixed operators involving nonlocal terms associated with radially symmetric, nonnegative, and continuous kernels of compact support. We also refer to the recent contributions \cite{BM25, KB2025} for further developments along this line of research.
Finally, we state the  Faber-Krahn inequality for the first eigenvalue of \eqref{Eig problem 000}, which provide the solution to \eqref{optfk} in the sense that \begin{equation} 
  \lambda_{1, \mu^+}(B):=  \inf\{\lambda_{1, \mu^+}(\Omega) :\, |\Omega| = \rho\},
\end{equation}
where $B$ is the Euclidean ball with volume $\rho.$

\begin{thm} \label{FKintro} Let $\Omega \subset \mathbb{R}^N$ be a bounded open set with boundary $\partial \Omega$ of class $C^1$. Assume that $\mu=\mu^{+}$ satisfies \eqref{measure1}. Let $s_{\sharp}$ be as in \eqref{measure4} and $1<p<N/s_{\sharp}$.  Let $\rho:=|\Omega| \in(0, \infty)$, and let $B$ be any Euclidean ball with volume $\rho$. Then,
\begin{equation}\label{Faber-Krahn inequalityint}
    \lambda_{1, \mu^+}(\Omega) \geq \lambda_{1, \mu^+}\left(B\right).
\end{equation}
Moreover, if the equality holds in \eqref{Faber-Krahn inequalityint}, then $\Omega$ is a ball (upto a set of Lebesgue measure zero).
\end{thm}

The proof of Theorem~\ref{FKintro} relies on the variational characterization of the principal eigenvalue $\lambda_{1,\mu^+}$ of the operator $\mathcal{L}_{\mu,p}^+$, in conjunction with a suitable P\'olya–Szeg\"o type inequality. Our argument draws inspiration from the methods developed in \cite{BLP2014, BDVV2} for the nonlocal and mixed Faber–Krahn inequalities. In those works, the authors employed a nonlocal version of the P\'olya–Szeg\"o inequality, originally established by Almgren and Lieb \cite{AL89} and later extended by Frank and Seiringer \cite{FS-2008}, which asserts that the $W_p^s(\mathbb{R}^N)$-norm does not increase under symmetric rearrangement. In the present study, we further extend this result (see Lemma~\ref{nonloc rear}) to a more general setting, allowing its application to the proof of Theorem~\ref{FKintro}.

We now turn to the study of the second eigenvalue $\lambda_{2,\mu^+}(\Omega)$ of the operator $\mathcal{L}_{\mu,p}^+$ for $1 < p < \infty$ and $\mu = \mu^+$ satisfying \eqref{measure4}. When $\mu^+$ is the Dirac measure concentrated at $1$, the operator $\mathcal{L}_{\mu,p}^+$ reduces to the nonlinear $p$-Laplacian $-\Delta_p$, and the investigation of its second eigenvalue was carried out in the seminal work of Cuesta, De Figueiredo, and Gossez \cite{CFG99}. The nonlocal analogue, namely the study of the second eigenvalue for the fractional $p$-Laplacian $(-\Delta)_p^{s}$ (corresponding to $\mu^+$ concentrated at some fractional power $s \in (0,1)$), was later addressed in \cite{BP16}. At this point, it is worth emphasizing that, in the framework of the nonlinear eigenvalue problem \eqref{Eig problem 000}, the very notion of a second eigenvalue is not a priori well defined. Indeed, the spectrum $\sigma_{s,\mu^+}(\Omega)$ may, in principle, contain a sequence of eigenvalues accumulating at $\lambda_{1,\mu^+}(\Omega)$. 
Further progress in this direction was made in \cite{GS19}, where the authors examined the second eigenvalue of the mixed local–nonlocal operator $\mathcal{L}_{\mu,p}^+ := -\Delta_p + (-\Delta)_p^{s},$
in the particular case where $\mu$ is the sum of two Dirac measures, $\delta_1 + \delta_s$ with $s \in (0,1)$. Building upon and extending these developments, we establish a unified framework for analyzing the second eigenvalue of the general nonlocal (nonlinear) superposition operator $\mathcal{L}_{\mu,p}^+$. 
This general setting introduces new and significant analytical challenges, primarily due to the lack of scaling invariance of the Rayleigh quotient associated with the eigenvalue problem \eqref{Eig problem 000}. Despite these difficulties, we successfully adapt and refine the ideas of \cite{BP16}, which were themselves inspired by the variational minimax method originally developed by Drábek and Robinson in their seminal work \cite{DR99}.

\begin{thm} \label{thm1.11inro}
    Let $\Omega \subset \mathbb{R}^{N}$ be a bounded domain with Lipschitz boundary.  Let $\mu^+$ satisfy \eqref{measure1}. Let $s_{\sharp}$ be as in \eqref{measure4} and $1<p<N/s_{\sharp}$. Then, there exists a real positive number $\lambda_{2, \mu^+}(\Omega)$ with the following properties:
     
(i) $\lambda_{2, \mu^+}(\Omega)$ is an $(s, \mu^+)$-eigenvalue of the operator $\mathcal{L}_{\mu,p}^+$.

(ii) $\lambda_{2, \mu^+}(\Omega)>\lambda_{1, \mu^+}(\Omega)$.

(iii) If $\lambda>\lambda_{1, \mu^+}(\Omega)$ is an eigenvalue, then $\lambda \geq \lambda_{2, \mu^+}(\Omega)$.

(iv) Every eigenfunction $u \in  \mathcal{M}:=\left\{u \in X_{p}(\Omega): \,\, \|u\|_{L^{p}(\Omega)}=1\right\},$  
    associated to $\lambda_{2, \mu^+}(\Omega)$ has to change sign.
\end{thm}

The final result of this paper provides a mountain pass characterization of $\lambda_{2, \mu^+}(\Omega)$. Variational characterizations of the second eigenvalue for particular cases of the operator $\mathcal{L}_{\mu,p}^+$ have been previously established in \cite{CFG99, BP16, GS19}. The present result extends these frameworks to a more general setting, offering a unified approach to the study of second eigenvalues for nonlocal nonlinear superposition operators.

\begin{thm} \label{thm1.12intro}
    Let $\Omega \subset \mathbb{R}^{N}$ be a bounded domain with Lipschitz boundary. Let $\mu^+$ satisfy \eqref{measure1}. Let $s_{\sharp}$ be as in \eqref{measure4} and $1<p<\infty$. Then, the second eigenvalue $\lambda_{2, \mu^+}$ has the following variational characterization
    \begin{equation}
        \lambda_{2, \mu^+}= \inf_{\phi \in \Gamma_1 (e_{1, \mu^+}, -e_{1, \mu^+})} \max_{ u \in \operatorname{Im}(\phi)} \|u\|^p_{X_p(\Omega)},
    \end{equation}
    where $\Gamma_1 (e_{1, \mu^+}, -e_{1, \mu^+}):= \{\phi \in C([0, 1], \mathcal{M}): \phi(0)= -e_{1, \mu^+} \,\,\,\text{and}\,\, \phi(1)= e_{1, \mu^+}  \}$   with $e_{1, \mu^+}$ is the first eigenfuction of the operator $\mathcal{L}_{\mu,p}^+.$
    \end{thm}
    
A natural line of investigation arising after Theorems \ref{FKintro}, \ref{thm1.11inro} and \ref{thm1.12intro} concerns the so-called {\it Hong–Krahn–Szeg\"o inequality} for the second Dirichlet eigenvalue $\lambda_{2, \mu^+}(\Omega)$ of the operator $\mathcal{L}_{\mu,p}^+$. For the classical Dirichlet eigenvalue problem associated with the Laplacian $-\Delta$, this inequality was first established by Krahn \cite{K1925} and later rediscovered independently by Hong \cite{H54} and Szeg\"o \cite{Polya55}. Its extension to the nonlinear $p$-Laplacian $-\Delta_p$ was subsequently obtained by Brasco and Franzina \cite{BF13}. 
  We recall that, in the local case of $p$-Laplacian, the Hong-Krahn-Szeg\"o inequality says that: 
\begin{center}
    ``{\it In the class of  all domains of fixed volume, the disjoint union of two equal balls
has the smallest second eigenvalue.}"
\end{center}
In other words, it asserts that 
\begin{equation} \label{HKS}
    |\Omega|^{\frac{p}{N}}\lambda_2(\Omega) \geq 2^{\frac{p}{N}} |B|^{\frac{p}{N}} \lambda_1(B),
\end{equation}
where $B$ is any $N$-dimensional ball. We refer to \cite{RSS} for more details. Moreover, the equality in \eqref{HKS} holds if and only if $\Omega$ is a disjoint union of two equal balls. Brasco and Parini \cite{BP16} proved the following nonlocal Hong-Krahn-Szeg\"o inequality for the fractional $p$-Laplacian:  
\begin{equation} \label{nHKS}
    |\Omega|^{\frac{ps}{N}}\lambda_2(\Omega) \geq 2^{\frac{ps}{N}} |B|^{\frac{ps}{N}} \lambda_1(B),
\end{equation}
where $B$ is any $N$-dimensional ball. Equality is never achieved in \eqref{nHKS}.
This phenomenon marks a substantial departure from the behavior observed in the local setting. The underlying reason lies in the fact that, in general, nonlocal energy functionals are strongly affected by the relative positioning of the distinct connected components of the domain. A similar feature is exhibited by the Hong–Krahn–Szeg\"o inequality for mixed local–nonlocal operators, as discussed in \cite{BDVV3, GS19}. Consequently, the associated shape optimization problem
\[
    \inf\{\lambda_2(\Omega): |\Omega| = c\}
\]
admits no minimizer in either the purely nonlocal or the mixed local–nonlocal framework. The standard methodology for proving the Hong-Krahn-Szeg\"o inequality combines the Faber-Krahn inequality for $\lambda_{1}(\Omega)$, key structural properties of the second eigenvalue $\lambda_{2}(\Omega)$, and the identification of nodal domains (via the Nodal Lemma) for the corresponding eigenfunctions (see \cite{BDVV3}). A crucial analytical ingredient in this framework is the global $L^{\infty}(\mathbb{R}^N)$ boundedness of the second eigenfunction, coupled with its interior H\"older regularity. In the case of the operator $\mathcal{L}_{\mu,p}^+$, and in light of the results established in this paper, the development of a comprehensive interior H\"older regularity theory for its eigenfunctions remains an open and essential problem (see \cite{BDVV3}) for establishing the Hong–Krahn–Szeg\"o inequality for $\lambda_{2, \mu^+}.$ Since the regularity theory for the operator $\mathcal{L}_{\mu,p}^+$ requires a comprehensive and detailed analysis, its development, together with the investigation of the Hong-Krahn-Szeg\"o inequality, will be addressed in a forthcoming work.

Because of the generality of the measure $\mu^+$, the results obtained in this paper cover the following classes of interesting examples and seem to be new in these particular cases, except for the first example, to the best of our knowledge, and therefore may serve as a starting point for further investigations in these specific settings. We present several particular cases of the operator $\mathcal{L}_{\mu, p}^+$ introduced in \eqref{super operator} below:
\begin{itemize}
    \item[(i)] If $s\in[0,1)$ and $\mu^+:=\delta_{1}+ \delta_{s}$, with $\delta_{s}$ and $\delta_{s}$ being the Dirac's deltas centered at $1$ and $s$, respectively, the operator $\mathcal{L}_{\mu, p}^+$ boils down to the mixed operator
\begin{equation*}
    \mathcal{L}_{\mu, p}^+=- \Delta_p+(-\Delta)_p^{s}.
\end{equation*}
   
    \item[(ii)] If $0\leq ...<s_2<s_1<s_0\leq 1$ and
$\mu^+:=\sum_{k=1}^{+\infty} c_{k} \delta_{s_{k}}$ (provided that the series converges) with  $c_{k} \geq 0$ for any $k \in \mathbb{N} \backslash\{0\}.$
Then the operator in \eqref{super operator} takes the form
\begin{equation*}
  \mathcal{L}_{\mu, p}^+(u)=\sum_{k=1}^{+\infty} c_{k}(-\Delta)_p^{s_{k}} u. 
\end{equation*}

    \item[(iii)] Let $f \not\equiv 0 $ be a measurable function such that $f \geq 0$ in $ (0, 1)$ and  $\int_{\bar{s}}^1 f(s)\, ds >0.$ Choosing $d\mu^+(s):=f(s) ds,$ he corresponding operator turns out to be 
    $$\mathcal{L}_{p, \mu}= \int_0^1 f(s) (-\Delta_p)^s u\, ds.$$
    
\end{itemize}

We conclude this introduction with an overview of the structure of the paper. Section~\ref{sec2} sets the stage by recalling the functional-analytic framework that underpins our analysis. In Section~\ref{sec3}, we turn our attention to proving Theorem~\ref{principal eigenvalue P 0}, focusing on the Dirichlet eigenvalue problem associated with the general nonlinear superposition operator $\mathcal{L}_{\mu,p}$, and we develop several foundational results needed later on.  
Sections~\ref{sec4} and~\ref{sec5} are devoted to the weak and strong minimum/maximum principles for the operator $\mathcal{L}_{\mu,p}^+$. There, we introduce an appropriate notion of {\it the nonlocal tail} for this class of operators and derive a logarithmic estimate for supersolutions to prove Theorem \ref{wmpinto} and Theorem \ref{smpintro}. Building upon these results, Section~\ref{sec6} deepens the study of eigenvalues and eigenfunctions of $\mathcal{L}_{\mu,p}^+$, culminating in the proof of the global $L^{\infty}(\mathbb{R}^N)$ boundedness of eigenfunctions corresponding to positive eigenvalues. In particular, we prove Theorem \ref{first eigenvalue exist + int} in this section. In Section~\ref{sec7}, we explore a shape optimization problem for the first eigenvalue of $\mathcal{L}_{\mu,p}^+$, formulated through the Faber–Krahn inequality and prove Theorem~\ref{FKintro}. Section~\ref{sec8} then focuses on the analysis of the second eigenvalue of $\mathcal{L}_{\mu,p}^+$, proving Theorem \ref{thm1.11inro} and finally, Section~\ref{sec9} concludes the paper by proving Theorem \ref{thm1.12intro}, which presents a variational characterization of this second eigenvalue.

\section{Preliminaries: Solution space setup and their embeddings} \label{sec2}
The purpose of this section is to develop the functional analytic framework required for our study, with particular emphasis on the relevant notions of fractional Sobolev spaces and their fundamental properties. For a more comprehensive treatment of this material, we refer the reader to \cite{DPSV, DPSV25-1, AGKR, BGK25, BGK25ii}.

We begin this section by introducing the Gagliardo semi-norm for $s \in [0, 1],$  as
\begin{equation*}
    [u]_{s, p}:= \begin{cases}\|u\|_{L^{p}\left(\mathbb{R}^{N}\right)} & \text { if } s=0, \\ \left(c_{N, s, p} \iint_{\mathbb{R}^{2 N}} \frac{|u(x)-u(y)|^{p}}{|x-y|^{N+s p}} d x d y\right)^{1 / p} & \text { if } s \in(0,1), \\ \|\nabla u\|_{L^{p}\left(\mathbb{R}^{N}\right)} & \text { if } s=1,\end{cases}
\end{equation*}
where, $$c_{N,p,s}:=\frac{\frac{sp}{2}(1-s)2^{2s-1}}{\pi^{\frac{N-1}{2}}}\frac{\Gamma(\frac{N+ps}{2})}{\Gamma(\frac{p+1}{2})\Gamma(2-s)} $$ is the normalizing constant. 
Due to the normalization of the constant $C_{N, s, p},$ we have the following relations:
$$
\lim _{s \searrow 0}[u]_{s, p}=[u]_{0, p} \quad \text { and } \quad \lim _{s \nearrow 1}[u]_{s, p}=[u]_{1, p}.
$$

     Let $\Omega\subset\mathbb{R}^N$ be a bounded domain with Lipschitz boundary. We use the Sobolev space $X_{p}(\Omega)$, introduced in \cite{DPSV}, which consists of all measurable functions $u: \mathbb{R}^{N} \rightarrow \mathbb{R}$ such that $u=0$ in $\mathbb{R}^{N} \backslash \Omega$ and 
\begin{equation}\label{norm on Xp}
  \|u\|_{X_p(\Omega)} = \rho_{p}(u):=\left(\int_{[0,1]}[u]_{s, p}^{p} d \mu^{+}(s)\right)^{1 / p}< +\infty.
\end{equation}
We define the dual pairing between $X_{p}(\Omega)$ and its dual as
\begin{equation}\label{dual pairing +}
\langle u, v\rangle_{+}:=\int_{[0,1]} c_{N, s, p} \iint_{\mathbb{R}^{2 N}} \frac{|u(x)-u(y)|^{p-2}(u(x)-u(y))(v(x)-v(y))}{|x-y|^{N+sp}} d x d y d \mu^{+}(s) . 
\end{equation}
for any $u, v \in X_{p}(\Omega).$ Moreover, using hypothesis \eqref{measure2}, we also define 
\begin{equation}\label{dual pairing -}
\langle u, v\rangle_{-}:=\int_{[0, \bar{s})} c_{N, s, p} \iint_{\mathbb{R}^{2 N}} \frac{|u(x)-u(y)|^{p-2}(u(x)-u(y))(v(x)-v(y))}{|x-y|^{N+sp}} d x d y d \mu^{-}(s) . 
\end{equation}

We recall the essential results developed in \cite{DPSV, AGKR} for further development.
\begin{lem}\cite{AGKR, DPSV}\label{Uniform convexity}
    The Sobolev space $X_{p}(\Omega)$ is separable for $1 \leq p < \infty$ and is uniformly convex for $1<p < \infty$.
\end{lem}

\footnotetext{
${ }^{2}$ We point out that we write
\begin{equation}\label{footnote}
\int_{[0,1]} c_{N, s, p} \iint_{\mathbb{R}^{2 N}} \frac{|u(x)-u(y)|^{p-2}(u(x)-u(y))(v(x)-v(y))}{|x-y|^{N+ ps}} d x d y d \mu^{+}(s) 
\end{equation}
with an abuse of notation. Indeed, to be precise, one should write
$$
\begin{aligned}
& \int_{(0,1)} c_{N, s, p} \iint_{\mathbb{R}^{2 N}} \frac{|u(x)-u(y)|^{p-2}(u(x)-u(y))(v(x)-v(y))}{|x-y|^{N+p s}} d x d y d \mu^{+}(s) \\
& \quad+\mu^{+}(\{0\}) \int_{\Omega} |u(x)|^{p-2}u(x) v(x) d x+\mu^{+}(\{1\}) \int_{\Omega} |\nabla u(x)|^{p-2}\nabla u(x) \cdot \nabla v(x) d x.
\end{aligned}
$$
To ease notation, unless otherwise specified, we will always use the compact expression \eqref{footnote}.}

\begin{lem}\label{negative part reabsorb}\cite[Proposition 4.1]{DPSV} 
Let $p \in(1, N)$ and let assumptions \eqref{measure2} and \eqref{measure3} hold.
Then, there exists $c_{0}=c_{0}(N, \Omega, p)>0$ such that, for any $u \in X_{p}(\Omega)$, we have
$$
\int_{[0, \bar{s}]}[u]_{s, p}^{p} d \mu^{-}(s) \leq c_{0} \gamma \int_{[\bar{s}, 1]}[u]_{s, p}^{p} d \mu(s).
$$
\end{lem}

\begin{lem}\label{negative part convergence}\cite[Proposition 3.1]{AGKR} 
    Let $\mu$ satisfy \eqref{measure1} and \eqref{measure2} and, let $s_{\sharp}$ be as in \eqref{measure4}. Suppose $(u_{k})$ is a sequence in $X_{p}(\Omega)$ that converges weakly to some $u$ in $X_{p}(\Omega)$ as $k \rightarrow+\infty$.
Then
\begin{equation*}
\lim _{k \rightarrow+\infty} \int_{[0, \bar{s})}\left[u_{k}\right]_{s,p}^{p} d \mu^{-}(s)=\int_{[0, \bar{s})}[u]_{s,p}^{p} d \mu^{-}(s) .
\end{equation*}
\end{lem}

\begin{lem}\label{energycomparison}\cite[Lemma 2.10]{AGKR}
Let $\mu$ satisfy \eqref{measure1} and \eqref{measure2} for some $\overline{s}\in (0,1)$. Let $R>0$ be such that $\Omega\subset B_R$ and let $\delta\in (0, 1-\overline{s}]$. Assume that \eqref{measure5} holds. Then, for any $u\in X_p(\Omega)$, we have
\begin{equation}\label{energy comp equation}
\int_{[0,1]} [|u|]^p_{s,p} d\mu(s) \leq \int_{[0,1]} [u]^p_{s,p} d\mu(s).
\end{equation}
Furthermore, the inequality in \eqref{energy comp equation} is strict unless either $u \geq 0$ or $u \leq 0$ a.e. in $\mathbb{R}^{N}$.
\end{lem}

We have the following result about the continuous and compact embeddings of $X_p(\Omega)$.
\begin{thm}\label{embedding lem}
    Let $\Omega\subset\mathbb{R}^N$ be a bounded domain with Lipschitz boundary. Let $\mu$ satisfy \eqref{measure1}, \eqref{measure2} and \eqref{measure3} and $s_{\sharp}$ be as in \eqref{measure4}.
Then, the space $X_{p}(\Omega)$ is continuously embedded in $W^{s_{\sharp},p}(\Omega)$.
Furthermore, 
\begin{enumerate}[(i)]
    \item if $N>p s_{\sharp}$, then the embedding $X_p(\Omega) \hookrightarrow L^{r}(\Omega)$ is continuous for any $r \in[1,p_{s_{\sharp}}^{*}]$ and compact for any $r \in[1,p_{s_{\sharp}}^{*})$, where $p_{s_{\sharp}}^{*}=\frac{Np}{N-ps_{\sharp}}$.
    \item if $N = p s_{\sharp}$, then the embedding $X_p(\Omega) \hookrightarrow L^{r}(\Omega)$ is continuous and compact for any $r \in [1,+\infty)$.
    \item if $N < p s_{\sharp}$, then the embedding $X_p(\Omega) \hookrightarrow C^{0,s_{\sharp}-N/p}(\bar{\Omega})$ is continuous and $X_p(\Omega) \hookrightarrow C^{0,\alpha}(\bar{\Omega})$ is compact for any $0<\alpha<s_{\sharp}-N/p$.
\end{enumerate}
\end{thm}
\begin{proof} The proof of the continuous the embedding $X_p(\Omega) \hookrightarrow W^{s_{\sharp},p}(\Omega)$  follows from \cite[Proposition 2.5]{AGKR}. The parts $(i)$, $(ii)$ and $(iii)$ are now the consequence of standard embedding results discussed in \cite[Theorem 4.54]{DD2012}, (see also \cite{DPV12}).
\end{proof}

\begin{lem}\cite[Lemma 5.9]{DPSV}\label{Brezis-Lieb lemma}
    Let $u_{k}$ be a bounded sequence in $X_{p}(\Omega)$. Suppose that $u_{k}$ converges to some $u$ a.e. in $\mathbb{R}^{N}$ as $k \rightarrow+\infty$. Then,
\begin{equation*}
\int_{[0,1]}[u]_{s, p}^{p} d \mu^{ \pm}(s)=\lim _{k \rightarrow+\infty}\left(\int_{[0,1]}\left[u_{k}\right]_{s, p}^{p} d \mu^{ \pm}(s)-\int_{[0,1]}\left[u_{k}-u\right]_{s, p}^{p} d \mu^{ \pm}(s)\right).
\end{equation*}
\end{lem}

We employ the Liusternik-Schnirelman theory to show the existence of a diverging sequence of eigenvalues of the operator $\mathcal{L}_{\mu,p}$. For this, we recall the notion of `\textit{genus}' of a set. Consider the class
$$
\mathcal{A}=\{A \subset X_{p}(\Omega)\setminus\{0\}: A \text { is closed, } A=-A\}.
$$
For any $\emptyset\neq A \in \mathcal{A}$, the genus of $A$ is denoted as $\gamma^{*}(A)$ and is defined as
\begin{equation}\label{def of genus}
    \gamma^{*}(A)= \inf \left\{m \in \mathbb{N} \cup\{0\} ; \text { there exists } h \in C\left(A ; \mathbb{R}^{m} \backslash\{0\}\right), h(-u)=h(u)\right\}.
\end{equation}
For empty set, $\emptyset$, the genus is defined as zero, that is $\gamma^{*}(\emptyset)=0$.
\begin{thm}\label{tech them for seq gen}\cite[Theorem 5.7]{Struwe} Let $X$ be a Banach space and let $\mathcal{S} \subset X \backslash\{0\}$ be a complete symmetric $C^{1,1}$ - manifold. Assume that  $\mathcal{I} \in C^{1}(\mathcal{S})$ is an even functional on  $\mathcal{S} $. We also assume that  $\mathcal{I}$ satisfies the Palais-Smale (PS) condition and is bounded from below on $\mathcal{S}$. Let
$$
\hat{\gamma}(\mathcal{S})=\sup \left\{\gamma^{*}(K): K \subset \mathcal{S} \text { compact and symmetric }\right\},
$$
where $\gamma^{*}$ is defined in \eqref{def of genus}. Then the functional $\mathcal{I}$ admits at least $\hat{\gamma}(\mathcal{S}) \leq \infty$ pairs of critical points. 
In addition, if  $ \gamma^{*}(\mathcal{S}) \geq k$, then the values $\beta_{k}:=\inf _{K \subset \mathcal{F}_{k}} \sup _{u \in K} \mathcal{I}(u)$ (provided they are finite) are critical points of  $\mathcal{I},$ where
 $\mathcal{F}_{k}=\left\{K \in \mathcal{A}: K \subset \mathcal{S}, \gamma^{*}(K) \geq k\right\}$.
\end{thm}

We conclude this section with the following two useful inequalities.
\begin{lem}\cite[Formula 2.2]{Simon}
    For all $t_{1}, t_{2} \in \mathbb{R}^{N}$, there exists a constant $C>0$ such that the following holds
\begin{equation}\label{Simon inequality}
    \left\langle |t_{1}|^{p-2}t_{1}-|t_{2}|^{p-2}t_{2}, t_{1}-t_{2}\right\rangle \geq \begin{cases}C\left|t_{1}-t_{2}\right|^{p} & \text { if } p \geq 2,  \\ C \frac{\left|t_{1}-t_{2}\right|^{2}}{\left(\left|t_{1}\right|+\left|t_{2}\right|\right)^{2-p}} & \text { if } 1<p \leq 2.\end{cases}
\end{equation}
\end{lem}  

\begin{lem}\cite[Lemma B.1]{BP16}\label{tech lem foe lam 2}
    Let $1<p<\infty$ and $U, V \in \mathbb{R}$ such that $U V \leq 0$. We define the following function
\begin{equation*}
    g(t)=|U-t V|^{p}+|U-V|^{p-2}(U-V) V|t|^{p}, \quad t \in \mathbb{R}.
\end{equation*}
Then we have
\begin{equation*}
    g(t) \leq g(1)=|U-V|^{p-2}(U-V) U, \quad t \in \mathbb{R}.
\end{equation*}
\end{lem} 
\section{Eigenvalue problem for the operator $\mathcal{L}_{\mu,p}$} \label{sec3}

The aim of this section is to study the Dirichlet eigenvalue problem of the operator $\mathcal{L}_{\mu,p}$. That is, we let $\lambda \in \mathbb{R}$ and we take into account the problem
\begin{equation}\label{Eig problem P}
    \begin{cases}\mathcal{L}_{\mu,p} u=\lambda |u|^{p-2}u & \text { in } \Omega, \\ 
u=0 & \text { in } \mathbb{R}^{N} \backslash \Omega .\end{cases}
\end{equation}


Let $\mathfrak{I}_{p}:X_{p}(\Omega) \rightarrow \mathbb{R}$ be the functional defined as 
\begin{align}\label{energy functional P}
\begin{split}
    \mathfrak{I}_{p}(u) & :=\frac{1}{p} \int_{[0,1]}[u]_{s,p}^{p} d \mu^{+}(s)-\frac{1}{p} \int_{[0, \bar{s})}[u]_{s,p}^{p} d \mu^{-}(s) \\
& =\frac{1}{p}\|u\|_{X_p(\Omega)}^{p}-\frac{1}{p} \int_{[0, \bar{s})}[u]_{s,p}^{p} d \mu^{-}(s), 
\end{split}
\end{align}
where $\|\cdot\|_{X_p(\Omega)}$ is the norm given in \eqref{norm on Xp}. Note that $\mathfrak{I}_{p}$ is well-defined on $X_{p}(\Omega)$ by extending $u=0$ on $\mathbb{R}^N \backslash \Omega$. Moreover, $\mathfrak{I}_{p} \in C^1\left(X_{p}(\Omega), \mathbb{R}\right)$ with derivative given by
\begin{align}\label{functional Hp}
\begin{split}
    \left\langle \mathfrak{I}_{p}^{\prime}(u), v\right\rangle&:=  \int_{[0,1]} c_{N, s, p} \iint_{\mathbb{R}^{2 N}} \frac{|u(x)-u(y)|^{p-2}(u(x)-u(y))(v(x)-v(y))}{|x-y|^{N+sp}} d x d y d \mu^{+}(s) \\
& -\int_{[0, \bar{s})} c_{N, s, p} \iint_{\mathbb{R}^{2 N}} \frac{|u(x)-u(y)|^{p-2}(u(x)-u(y))(v(x)-v(y))}{|x-y|^{N+sp}} d x d y d \mu^{-}(s) 
\end{split}
\end{align}
for any $u,v \in X_{p}(\Omega)$.

\begin{definition}
      A function $u \in X_{p}(\Omega)$ is a (weak) solution of \eqref{Eig problem P} if $u$ satisfies
\begin{equation}\label{weak solution of eig problem P}
    \left\langle \mathfrak{I}_{p}^{\prime}(u), v\right\rangle=\lambda \int_{\Omega} |u|^{p-2} u v d x \text { for all } v \in X_{p}(\Omega),
\end{equation}
where $\left\langle \mathfrak{I}_{p}^{\prime}(u), v\right\rangle$ is defined in \eqref{functional Hp}.
\end{definition}

We recall that if there exists a nontrivial solution $u \in X_{p}(\Omega)\setminus\{0\}$ of \eqref{Eig problem P}, then $\lambda \in \mathbb{R}$ is called an eigenvalue of the operator $\mathcal{L}_{\mu,p}$ and  $u \in X_{p}(\Omega)\setminus\{0\}$ is called an eigenfunction associated with the eigenvalue $\lambda$.

The following lemmas will be essential in establishing the various properties of eigenvalues and eigenfunctions to problem \eqref{Eig problem P}. 

\begin{lem}\label{tech lem for principal eigenvalue}
Let $\Omega \subset \mathbb{R}^{N}$ be a bounded domain with Lipschitz boundary. Assume that $\mu$ satisfies \eqref{measure1}, \eqref{measure2} and \eqref{measure3}. Let $s_{\sharp}$ be as in \eqref{measure4} and $1<p<N/ s_{\sharp}$.
Let $X_{0}$ be a non-empty, weakly closed linear subspace of $X_{p}(\Omega)$ and
$$
\mathcal{N}:=\left\{u \in X_{0}:\,\,\, \|u\|_{L^{p}(\Omega)}=1\right\} .
$$

Then, there exists $\gamma_{0}>0$ in \eqref{measure3}, depending only on $N$, $\Omega$, and $p$, such that, for any $\gamma \in\left[0, \gamma_{0}\right],$ we have 
\begin{equation}\label{minimizer exist}
\inf _{u \in \mathcal{N}} \mathfrak{I}_{p}(u)=\mathfrak{I}_{p}\left(u_{0}\right)>0,
\end{equation}
for some  $u_{0} \in \mathcal{N}.$ Moreover, for any $v \in X_{0}$, we have
\begin{align}\label{relation P}
\begin{split}
    & \int_{[0,1]} c_{N, s, p} \iint_{\mathbb{R}^{2 N}} \frac{|u_{0}(x)-u_{0}(y)|^{p-2}\left(u_{0}(x)-u_{0}(y)\right)(v(x)-v(y))}{|x-y|^{N+sp}} d x d y d \mu^{+}(s) \\
& \quad-\int_{[0, \bar{s})} c_{N, s, p} \iint_{\mathbb{R}^{2 N}} \frac{|u_{0}(x)-u_{0}(y)|^{p-2}\left(u_{0}(x)-u_{0}(y)\right)(v(x)-v(y))}{|x-y|^{N+sp}} d x d y d \mu^{-}(s)\\
& \quad =p \mathfrak{I}_{p}\left(u_{0}\right) \int_{\Omega} |u_{0}(x)|^{p-2}u_{0}(x) v(x) d x . 
\end{split}
\end{align}
\end{lem}
\begin{proof}
Let $\{u_{k}\} \in \mathcal{N}$ be a minimizing sequence for the functional $\mathfrak{I}_{p}$, that is
\begin{equation}\label{minimizing sequence}
\lim _{k \rightarrow+\infty} \mathfrak{I}_{p}\left(u_{k}\right)=\inf _{u \in \mathcal{N}} \mathfrak{I}_{p}(u) . 
\end{equation}
Thus, we conclude that 
\begin{equation}\label{upper bound for Ip}
\mathfrak{I}_{p}\left(u_{k}\right) \leq C 
\end{equation}
for some $C>0$ independent of $k$.
On the other hand, by \eqref{energy functional P} and Lemma \ref{negative part reabsorb}, we have that
\begin{equation*}
    \mathfrak{I}_{p}(u) \geq \frac{1}{p}\left(1-c_{0} \gamma\right)\|u\|_{X_p(\Omega)}^{p}.
\end{equation*}
Consequently, if $\gamma$ is small enough, possibly depending on $c_{0}$ (and therefore on $N$, $\Omega$, and $p$), then for all $u \in X_{p}(\Omega),$
\begin{equation}\label{lower bound for Ip}
\mathfrak{I}_{p}(u) \geq c\|u\|_{X_p(\Omega)}^{p}, 
\end{equation}
for some constant $c>0$.
From \eqref{upper bound for Ip} and \eqref{lower bound for Ip}, we conclude that $u_{k}$ is bounded in $X_{p}(\Omega)$.
Therefore, up to subsequence, there exists $u_{0} \in X_{0}$ such that $u_{k}\rightharpoonup u_{0}$ weakly in $X_{p}(\Omega)$. By Theorem \ref{embedding lem}, it follows that $u_{k}\to u_{0}$ strongly in $L^{p}\left(\Omega\right)$. Hence, we have $\left\|u_{0}\right\|_{L^{p}(\Omega)}=1$ and $u_{0} \in \mathcal{N}$.

Now, observe that, by Fatou's lemma,
\begin{align*}
& \liminf _{k \rightarrow+\infty} \int_{[0,1]} c_{N, s, p} \iint_{\mathbb{R}^{2 N}} \frac{\left|u_{k}(x)-u_{k}(y)\right|^{p}}{|x-y|^{N+sp}} d x d y d \mu^{+}(s) \\
& \quad \geq \int_{[0,1]} c_{N, s, p} \iint_{\mathbb{R}^{2 N}} \frac{\left|u_{0}(x)-u_{0}(y)\right|^{p}}{|x-y|^{N+sp}} d x d y d \mu^{+}(s) . 
\end{align*}
This, together with Lemma \ref{negative part convergence}, implies that
\begin{align*}
\lim _{k \rightarrow+\infty} \mathfrak{I}_{p}\left(u_{k}\right) & =\frac{1}{p} \lim _{k \rightarrow+\infty}\left(\int_{[0,1]}\left[u_{k}\right]_{s,p}^{p} d \mu^{+}(s)-\int_{[0, \bar{s})}\left[u_{k}\right]_{s,p}^{p} d \mu^{-}(s)\right) \\
& \geq \frac{1}{p}\int_{[0,1]}\left[u_{0}\right]_{s,p}^{p} d \mu^{+}(s)-\frac{1}{p}\int_{[0, \bar{s})}\left[u_{0}\right]_{s,p}^{p} d \mu^{-}(s) \\
& =\mathfrak{I}_{p}\left(u_{0}\right) \geq \inf _{u \in \mathcal{N}} \mathfrak{I}_{p}(u) .
\end{align*}
Taken together with \eqref{minimizing sequence}, this argument ensures the existence of a minimizer $u_0$, thereby establishing \eqref{minimizer exist}.

Now, we claim that $\mathfrak{I}_{p}\left(u_{0}\right)>0$. Indeed, since $u_{0} \in \mathcal{N}$, we have that $u_{0} \not \equiv 0$. Hence, \eqref{energy functional P} and \eqref{lower bound for Ip} imply 
\begin{equation*}
    \mathfrak{I}_{p}\left(u_{0}\right) \geq c\left\|u_{0}\right\|_{X_p(\Omega)}^{p}>0.
\end{equation*}

Next, we proceed with the proof of \eqref{relation P}. To this end, let $\varepsilon \in(-1,1) \backslash\{0\}$ and $v \in X_{0}$ and set
\begin{equation*}
    u_{\varepsilon}(x):=\frac{u_{0}(x)+\varepsilon v(x)}{\left\|u_{0}+\varepsilon v\right\|_{L^{p}(\Omega)}} .
\end{equation*}
Obviously, we have $u_{\varepsilon} \in \mathcal{N}$. Also, from \eqref{dual pairing +} and \eqref{dual pairing -}, we have
\begin{equation*}
    \left\|u_{0}+\varepsilon v\right\|_{L^{p}(\Omega)}^{p}=1+p \varepsilon \int_{\Omega} |u_{0}(x)|^{p-2}u_{0}(x) v(x) d x+...+\varepsilon^{p}\|v\|_{L^{p}(\Omega)}^{p},
\end{equation*}
\begin{equation*}
    \left\|u_{0}+\varepsilon v\right\|_{X_p(\Omega)}^{p}=\left\|u_{0}\right\|_{X_p(\Omega)}^{p}+p \varepsilon\left\langle u_{0}, v\right\rangle_{+}+...+\varepsilon^{p}\|v\|_{X_p(\Omega)}^{p},
\end{equation*}
and
\begin{equation*}
     \int_{[0, \bar{s})}\left[u_{0}+\varepsilon v\right]_{s,p}^{p} d \mu^{-}(s)=\int_{[0, \bar{s})}\left[u_{0}\right]_{s,p}^{p} d \mu^{-}(s)+p \varepsilon\left\langle u_{0}, v\right\rangle_{-}+...+\varepsilon^{p} \int_{[0, \bar{s})}[v]_{s,p}^{p} d \mu^{-}(s) .
\end{equation*}
From this and \eqref{energy functional P}, we obtain
\begin{align}\label{relation P1}
\begin{split}
    p \mathfrak{I}_{p}\left(u_{\varepsilon}\right)&=  \frac{1}{\left\|u_{0}+\varepsilon v\right\|_{L^{p}(\Omega)}^{p}}\left(\left\|u_{0}+\varepsilon v\right\|_{X_p(\Omega)}^{p}-\int_{[0, \bar{s})}\left[u_{0}+\varepsilon v\right]_{s,p}^{p} d \mu^{-}(s)\right) \\
& =\frac{p \mathfrak{I}_{p}\left(u_{0}\right)+p \varepsilon\left(\left\langle u_{0}, v\right\rangle_{+}-\left\langle u_{0}, v\right\rangle_{-}\right)+...+\varepsilon^{p}\left(\|v\|_{X_p(\Omega)}^{p}-\int_{[0, \bar{s})}[v]_{s,p}^{p} d \mu^{-}(s)\right)}{1+p \varepsilon \int_{\Omega} |u_{0}(x)|^{p-2}u_{0}(x) v(x) d x+...+\varepsilon^{p}\|v\|_{L^{p}(\Omega)}^{p}} .
\end{split}
\end{align}
Now, observe that
\begin{align}\label{relation P2}
    &\frac{p \mathfrak{I}_{p}\left(u_{0}\right)}{\varepsilon(1+p \varepsilon \int_{\Omega} |u_{0}(x)|^{p-2}u_{0}(x) v(x) d x+...+\varepsilon^{p}\|v\|_{L^{p}(\Omega)}^{p})} \\\nonumber 
 & \quad \quad \quad \quad \quad \quad \quad \quad = \frac{p \mathfrak{I}_{p}\left(u_{0}\right)}{\varepsilon}-\frac{p^{2} \mathfrak{I}_{p}\left(u_{0}\right) \int_{\Omega}|u_{0}(x)|^{p-2}u_{0}(x) v(x)dx}{1+p \varepsilon \int_{\Omega} |u_{0}(x)|^{p-2}u_{0}(x) v(x) d x+...+\varepsilon^{p}\|v\|_{L^{p}(\Omega)}^{p}}\\
 & \quad \quad \quad \quad \quad \quad \quad \quad -...-\frac{p\varepsilon^{p-1}\mathfrak{I}_{p}(u_{0})\|v\|_{L^{p}(\Omega)}^{p}}{1+p \varepsilon \int_{\Omega} |u_{0}(x)|^{p-2}u_{0}(x) v(x) d x+...+\varepsilon^{p}\|v\|_{L^{p}(\Omega)}^{p}}.\nonumber
\end{align}
Then, dividing both sides of \eqref{relation P1} by $\varepsilon$ and using \eqref{relation P2}, we arrive at 
\begin{align*}
    \frac{p \mathfrak{I}_{p}\left(u_{\varepsilon}\right)-p \mathfrak{I}_{p}(u_{0})}{\varepsilon}&=\frac{p\left(\langle u_{0}, v\rangle_{+}-\langle u_{0}, v\rangle_{-}-p \mathfrak{I}_{p}(u_{0}) \int_{\Omega} |u_{0}(x)|^{p-2}u_{0}(x) v(x) d x\right)}{1+p \varepsilon \int_{\Omega} |u_{0}(x)|^{p-2}u_{0}(x) v(x) d x+\varepsilon^{p}\|v\|_{L^{p}(\Omega)}^{p}} \\
&+...+\frac{\varepsilon^{p-1}\left(\|v\|_{X_p(\Omega)}^{p}-\int_{[0, \bar{s})}[v]_{s,p}^{p} d \mu^{-}(s)-p \mathfrak{I}_{p}\left(u_{0}\right)\|v\|_{L^{p}(\Omega)}^{p}\right)}{1+p \varepsilon \int_{\Omega} |u_{0}(x)|^{p-2}u_{0}(x) v(x) d x+\varepsilon^{p}\|v\|_{L^{p}(\Omega)}^{p}} .
\end{align*}
Since $u_{0}$ is a minimizer of $\mathfrak{I}_{p}$ over $\mathcal{N}$, by
passing to the limit as $\varepsilon \rightarrow 0,$ we derive the desired identity in \eqref{relation P}, which completes the proof.
\end{proof}

The following lemma provides the pointwise convergence almost everywhere for $\{u_k\}$. The idea of the proof is taken from \cite{SFV} (see also \cite{MGS25}).
\begin{lem}\label{tech lem for PS gen} Let $\Omega \subset \mathbb{R}^{N}$ be a bounded domain with Lipschitz boundary. Assume that $\mu$ satisfies \eqref{measure1}, \eqref{measure2} and \eqref{measure3}. We further assume that $\mu^+\{1\}>0.$ Let $s_{\sharp}$ be as in \eqref{measure4} and $1<p<N/ s_{\sharp}$.
Let $\mathfrak{I}_{p}$ be the functional defined as in \eqref{energy functional P} and $\tilde{\mathfrak{I}}_{p}:=\left.\mathfrak{I}_{p}\right|_{\mathcal{M}},$ where \begin{equation}\label{the set M P}
    \mathcal{M}:=\left\{u \in X_{p}(\Omega): \quad \|u\|_{L^{p}(\Omega)}=1\right\}. 
\end{equation} Let $\left\{u_{k}\right\}$ be a sequence in $X_{p}(\Omega)$ such that $\mathfrak{I}_{p}\left(u_{k}\right) \rightarrow c$ for some $c \in \mathbb{R}$ and $\|\tilde{\mathfrak{I}}_{p}^{\prime}\left(u_{k}\right)\|_{*} \rightarrow 0$, where 
\begin{equation*}
\|\tilde{\mathfrak{I}}_{p}^{\prime}(u_k)\|_{*}=\inf \left\{\left\|\mathfrak{I}^{\prime}_{p}(u_k)-a \mathcal{P}^{\prime}(u_k)\right\|_{X^{*}}: a \in \mathbb{R}\right\}
\end{equation*}
with $
 \mathcal{P}(u):=\frac{1}{p}\int_{\Omega} |u|^{p} d x.
$
Then, up to a subsequence, we have $\nabla u_{k}(x) \rightarrow \nabla u(x)$ a.e. in $\Omega$ as $k \rightarrow \infty$, provided that $\gamma$ in \eqref{measure3} is sufficiently small.
\end{lem}
\begin{proof} From the condition $\mathfrak{I}_{p}(u_{k}) \to c$ for some $c \in \mathbb{R}$, it follows that the sequence $\{\mathfrak{I}_{p}(u_{k})\}$ is bounded. Thus, there exists $M>0$ such that 
\[
|\mathfrak{I}_{p}(u_{k})| \leq M \quad \text{for all } k \in \mathbb{N}.
\]
 Therefore, Lemma \ref{negative part reabsorb} yields that
$$  (1-c_0\gamma)\|u_k\|^p_{X_p(\Omega)}\leq \|u_k\|^{p}_{X_p(\Omega)}-\int_{[0,\bar{s})}[u_k]_{s,p}^p d\mu^{-}(s) = p\mathfrak{I}_p\left(u_k\right) \leq pM\text{ for every } k \in \mathbb{N}.$$
Consequently, the sequence $\{u_{k}\}$ is bounded in $X_{p}(\Omega)$, provided that $\gamma$ is chosen sufficiently small.
Therefore, by Theorem \ref{embedding lem} along with the fact that $X_p(\Omega)$ is a separable space,  up to a subsequence (still denoted by $\{u_{k}\}$), we obtain as $k \to \infty$ that
\begin{eqnarray}\label{all convergences}
    &u_{k} \rightharpoonup u \text { weakly in } X_{p}(\Omega), \quad & \nabla u_{k} \rightharpoonup \nabla u \text { weakly in }\left(L^{p}(\Omega)\right)^{N}, \nonumber\\
&u_{k}(x) \rightarrow u(x) \text { pointwise a.e. in } \Omega, \quad & \left|u_{k}(x)\right| \leq h(x) \text { a.e. in } \Omega, \\
&u_{k} \rightarrow u \text { strongly in } L^{r}(\Omega), &\nonumber
\end{eqnarray}
where $r \in[1, p_{s_\sharp}^{*})$ and $h \in L^{p}(\Omega)$.

Further, the condition that $\|\tilde{\mathfrak{I}}_{p}^{\prime}\left(u_{k}\right)\|_{*} \rightarrow 0$ implies that for each $k\in\mathbb{N}$, there exists a sequence $\left\{a_{k}\right\}$ such that $\|{\mathfrak{I}}_{p}^{\prime}(u_{k})-a_k\mathcal{P}'(u_k)\|_{X_p(\Omega)*}\rightarrow 0$. In particular, for all $v \in X_{p}(\Omega)$, we have 
\begin{equation}\label{gradient conv 1}
\left|\left\langle\mathfrak{I}_{p}^{\prime}\left(u_{k}\right), v\right\rangle-a_{k} \int_{\Omega}|u_{k}|^{p-2} u_{k} vdx \right| \leq \varepsilon_{k}\|v\|_{X_p(\Omega)}~\text{ as } \epsilon_k\rightarrow0,
\end{equation}
 which by taking $v=u_{k}$ in \eqref{gradient conv 1} further indicate that
\begin{align*}
    \left|a_{k}\right| &\leq \varepsilon_{k}\left\|u_{k}\right\|_{X_p(\Omega)}+\left\|u_{k}\right\|_{X_p(\Omega)}^{p}+\int_{[0,\bar{s})} [u_{k}]_{s,p}^{p}d\mu^{-}(s) \\&\leq \varepsilon_{k}\left\|u_{k}\right\|_{X_p(\Omega)} +(1+c_0\gamma)\left\|u_{k}\right\|_{X_p(\Omega)}^{p}\leq \varepsilon_{k}\left\|u_{k}\right\|_{X_p(\Omega)} +2\left\|u_{k}\right\|_{X_p(\Omega)}^{p} .
\end{align*}
Thus, the sequence $\{a_{k}\}$ is bounded. Now, for any $j \in \mathbb{N}$, consider the truncation functions $ T_j: \mathbb{R} \rightarrow \mathbb{R}$ as
\begin{equation*}
T_j(t)= \begin{cases}t & \text { if }|t| \leq j, \\ j \frac{t}{|t|} & \text { if }|t|>j.\end{cases}
\end{equation*}
Since $T_j$ is bounded,  H\"older's inequality and \eqref{all convergences} give that
\begin{align}
    &\lim _{k \rightarrow \infty} \int_{\Omega}|\nabla u|^{p-2} \nabla u \nabla\left(T_j\left(u_{k}-u\right)\right) d x=0,\label{gradient conv 2}\\
    &\lim _{k \rightarrow \infty}\int_{(0,1)} c_{N,s,p}\iint_{\mathbb{R}^{2N}}A_{s, u}(x,y)\left(T_j\left(u_{k}-u\right)(x)-T_j\left(u_{k}-u\right)(y)\right) d x d yd\mu^+(s)=0,\label{gradient conv 3}\\
    &\lim _{k \rightarrow \infty}\int_{(0,1)} c_{N,s,p}\iint_{\mathbb{R}^{2N}}A_{s, u}(x,y)\left(T_j\left(u_{k}-u\right)(x)-T_j\left(u_{k}-u\right)(y)\right) d x d yd\mu^-(s)=0,\label{gradient conv 3.1}\\
    &\lim _{k \rightarrow \infty} \int_{\Omega}|u|^{p-2}u T_j\left(u_{k}-u\right) d x=0,\label{gradient conv 4}
\end{align}
where \begin{equation}\label{Asu}
    A_{s, u}(x,y)=\frac{|u(x)-u(y)|^{p-2}(u(x)-u(y))}{|x-y|^{N+p s}}.
\end{equation}
Using  \eqref{gradient conv 2}, \eqref{gradient conv 3}, \eqref{gradient conv 4}, and \eqref{gradient conv 3.1} we obtain $\left\langle\mathfrak{I}_{p}^{\prime}(u), T_j\left(u_{k}-u\right)\right\rangle=o_{k}(1)$. Indeed, we have 
\begin{align*}
   &\lim_{k\to +\infty} \left\langle\mathfrak{I}_{p}^{\prime}(u), T_j\left(u_{k}-u\right)\right\rangle \\
   &=\lim_{k\to +\infty} \int_{[0,1]}c_{N,s,p} \iint_{\mathbb{R}^{2N}} A_{s, u}(x,y) \left(T_j\left(u_{k}-u\right)(x)-T_j\left(u_{k}-u\right)(y)\right)d x d yd\mu^{+}(s)\\
   &-\lim_{k\to +\infty} \int_{[0,\bar{s})}c_{N,s,p} \iint_{\mathbb{R}^{2N}} A_{s, u}(x,y)\left(T_j\left(u_{k}-u\right)(x)-T_j\left(u_{k}-u\right)(y)\right)d x d yd\mu^{-}(s)= 0.
\end{align*}
Next, choosing the test function $v = T_j(u_{k} - u)$ in \eqref{gradient conv 1}, we obtain
\begin{align}\label{gradient conv 5}
\begin{split}
    \left|\left\langle\mathfrak{I}_{p}^{\prime}\left(u_{k}\right)-\mathfrak{I}_{p}^{\prime}(u), T_j\left(u_{k}-u\right)\right\rangle\right| &\leq a_{k}\left|\int_{\Omega}(\left|u_{k}\right|^{p-2} u_{k}-|u|^{p-2} u)\left(T_j\left(u_{k}-u\right)\right)\right|\\
    &+\varepsilon_{k}\left\|T_j\left(u_{k}-u\right)\right\|_{X_p(\Omega)}+o_{k}(1),
\end{split}
\end{align}
which, since $\mu^-\{1\}=0$, further yields that
\begin{align}\label{gradient conv F:1}
&\mu^{+}\{1\}\int_{\Omega}(\left|\nabla u_{k}\right|^{p-2} \nabla u_{k}-|\nabla u|^{p-2} \nabla u) \nabla\left(T_j\left(u_{k}-u\right)\right)dx  \nonumber \\
&+\int_{(0,1)}c_{N,s,p}\iint_{\mathbb{R}^{2N}}[A_{s,u_k}(x,y)-A_{s,u}(x,y)] \left(T_j\left(u_{k}-u\right)(x)-T_j\left(u_{k}-u\right)(y)\right)dxdy d\mu^{+}(s) \nonumber\\
&-\int_{(0,1)}c_{N,s,p}\iint_{\mathbb{R}^{2N}}[A_{s,u_k}(x,y)-A_{s,u}(x,y)]  \left(T_j\left(u_{k}-u\right)(x)-T_j\left(u_{k}-u\right)(y)\right)dxdy d\mu^{-}(s) \nonumber\\
& \leq \mu^{-}\{0\}\int_{\Omega}(\left| u_{k}\right|^{p-2}  u_{k}-|u|^{p-2}  u)\left(T_j\left(u_{k}-u\right)\right)dx \nonumber\\
& \quad \quad-\mu^{+}\{0\}\int_{\Omega}(\left| u_{k}\right|^{p-2}  u_{k}-|u|^{p-2}  u)\left(T_j\left(u_{k}-u\right)\right)dx \nonumber \\& \quad \quad+ a_{k}\left|\int_{\Omega}(\left|u_{k}\right|^{p-2} u_{k}-|u|^{p-2} u)\left(T_j\left(u_{k}-u\right)\right)\right|+\varepsilon_{k}\| T_j\left(u_{k}-u\right)\|_{X_p(\Omega)}+o_{k}(1).
\end{align}
We recall the following pointwise inequality from   \cite[Inequality (2.8)]{SFV}
\begin{align*}
    &\left[|u_{k}(x)-u_{k}(y)|^{p-2}(u_{k}(x)-u_{k}(y))-|u(x)-u(y)|^{p-2}(u(x)-u(y)\right]  \\
    &\quad\quad\times \left(T_j\left(u_{k}-u\right)(x)-T_j\left(u_{k}-u\right)(y)\right)\geq 0.
\end{align*}
Also, from \cite[Lemma 2.10]{AGKR}, we have
\begin{equation*}
    \int_{(0,1)}\frac{c_{N,s,p}}{|x-y|^{N+sp}}d\mu^{+}(s)>\int_{(0,\bar{s})}\frac{c_{N,s,p}}{|x-y|^{N+sp}}d\mu^{-}(s).
\end{equation*}
Now, using \eqref{Asu}, we observe that
\begin{align*}
&\int_{(0,1)}c_{N,s,p}\iint_{\mathbb{R}^{2N}}[A_{s,u_k}(x,y)-A_{s,u}(x,y)]  \left(T_j\left(u_{k}-u\right)(x)-T_j\left(u_{k}-u\right)(y)\right)dxdy d\mu^{+}(s)\\
&-\int_{(0,\bar{s})}c_{N,s,p}\iint_{\mathbb{R}^{2N}}[A_{s,u_k}(x,y)-A_{s,u}(x,y)] \left(T_j\left(u_{k}-u\right)(x)-T_j\left(u_{k}-u\right)(y)\right) dxdy d\mu^{-}(s) \\
&=\iint_{\mathbb{R}^{2N}}\left[|u_{k}(x)-u_{k}(y)|^{p-2}(u_{k}(x)-u_{k}(y))-|u(x)-u(y)|^{p-2}(u(x)-u(y)\right]\\
& \quad\quad \times\left(T_j\left(u_{k}-u\right)(x)-T_j\left(u_{k}-u\right)(y)\right)\\
&\quad\quad\times\left[\int_{(0,1)}\frac{c_{N,s,p}}{|x-y|^{N+p s}}d\mu^{+}(s)-\int_{(0,\bar{s})}\frac{c_{N,s,p}}{|x-y|^{N+p s}}d\mu^{-}(s)\right] dxdy \\
& \geq 0.
\end{align*}
Then, using the above observation and the strong convergence in \eqref{all convergences}, and passing to the limit in \eqref{gradient conv F:1}, noting that $\mu^+\{1\}>0$, we obtain
\begin{equation}
    \limsup_{k\to \infty}\int_{\Omega}(\left|\nabla u_{k}\right|^{p-2} \nabla u_{k}-|\nabla u|^{p-2} \nabla u) \nabla\left(T_j\left(u_{k}-u\right)\right)dx\leq 0.
\end{equation}

Now, define
\[
r_{k}(x) = \Big( |\nabla u_{k}(x)|^{p-2} \nabla u_{k}(x) - |\nabla u(x)|^{p-2} \nabla u(x) \Big) \cdot \nabla \big(u_{k}(x) - u(x)\big).
\]
By inequality \eqref{Simon inequality}, it follows that $r_{k}(x) \geq 0$. Consider the subsets of $\Omega$ as
\begin{equation*}
    S_k^j=\left\{x \in \Omega:\left|u_k(x)-u(x)\right| \leq j\right\}, \quad G_k^j=\left\{x \in \Omega:\left|u_k(x)-u(x)\right|>j\right\} .
\end{equation*}
Then for $\delta \in(0,1)$, we have
\begin{align*}
\int_{\Omega} r_k^\delta & =\int_{S_k^j} r_k^\delta+\int_{G_k^j} r_k^\delta \\
& \leq\left(\int_{S_k^j} r_k\right)^\delta\left|S_k^j\right|^{1-\delta}+\left(\int_{G_k^j} r_k\right)^\delta\left|G_k^j\right|^{1-\delta} \\
& \leq{{(\bar{C})^\delta\left|G_k^j\right|^{1-\delta} }}.
\end{align*}
Since $\left|G_{k}^{j}\right| \rightarrow 0$ as $k \rightarrow \infty$, we obtain
$$
0 \leq \limsup _{k \rightarrow \infty} \int_{\Omega} r_{k}^{\delta} d x \leq0.
$$
Thus, we get $r_k^\delta \to 0$ as $k \to \infty$ in $L^1(\Omega)$.
Subsequently, $r_{k}(x) \rightarrow 0,$ \,a.e. in $\Omega$ as $k \rightarrow \infty$. Therefore, as a consequence of \eqref{Simon inequality}, we deduce that 
$\nabla u_{k}(x) \to \nabla u(x) \text{ a.e. in } \Omega \text{ as } k \to \infty,$
which completes the proof of our result.
\end{proof}

Recall that  the functional $\mathfrak{I}_{p}$ given by \eqref{energy functional P} satisfies the Palais--Smale (PS) condition (at the level $c\in\mathbb{R}$) if every sequence $\{u_k\}_{k\in\mathbb{N}}\subset X_p(\Omega)$ such that
\begin{equation}\label{PS-seq}
\mathfrak{I}_{p}(u_k)\to c
\quad\text{and}\quad
\mathfrak{I}_{p}'(u_k) \to 0 \,\, \text{in}\,\,X_p(\Omega)^*
\quad\text{as } k\to+\infty,
\end{equation}
admits a subsequence which converges strongly in $X_p(\Omega)$. 

Now, we state  the following important result.

\begin{lem}\label{PS cond gen} Let $\Omega \subset \mathbb{R}^{N}$ be a bounded domain with Lipschitz boundary. Assume that $\mu$ satisfies \eqref{measure1}, \eqref{measure2} and \eqref{measure3}. Let $s_{\sharp}$ be as in \eqref{measure4} and $1<p<N/ s_{\sharp}$.
Let $\mathfrak{I}_{p}$ be the functional defined as in \eqref{energy functional P} and $\tilde{\mathfrak{I}}_{p}:=\left.\mathfrak{I}_{p}\right|_{\mathcal{M}},$ where \begin{equation}
    \mathcal{M}:=\left\{u \in X_{p}(\Omega): \quad \|u\|_{L^{p}(\Omega)}=1\right\}. 
\end{equation}
  Then, the functional $\tilde{\mathfrak{I}}_{p}$ satisfies the $(PS)$ condition on $\mathcal{M}$ provided that $\gamma$ in \eqref{measure3} is sufficiently small.
\end{lem}

\begin{proof}
 Let $\{u_{k}\} \in \mathcal{M}$ be a Palais-Smale sequence for $\tilde{\mathfrak{I}}_{p}$. Then there exists a constant $M>0$ and a sequence of real numbers $\{a_{k}\}$ such that
\begin{equation}\label{PS cond gen 1}
\mathfrak{I}_{p}(u_{k}) \leq M,
\end{equation}
and

\begin{align}\label{PS cond gen 2}
\begin{split}
     &\Bigg|\int_{[0,1]} c_{N, s, p} \iint_{\mathbb{R}^{2 N}} \frac{|u_{k}(x)-u_{k}(y)|^{p-2}(u_{k}(x)-u_{k}(y))(v(x)-v(y))}{|x-y|^{N+sp}} d x d y d \mu^{+}(s) \\
&\quad -\int_{[0,\bar{s})} c_{N, s, p} \iint_{\mathbb{R}^{2 N}} \frac{|u_{k}(x)-u_{k}(y)|^{p-2}(u_{k}(x)-u_{k}(y))(v(x)-v(y))}{|x-y|^{N+sp}} d x d y d \mu^{-}(s)\\
&\quad -a_{k}\int_{\Omega}|u_{k}|^{p-2}u_{k}vdx\Bigg| \leq \varepsilon_{k}\|v\|_{X_p(\Omega)}
\end{split}
\end{align}
for all $v \in \mathcal{M}$, and for some $\varepsilon_{k}>0$ such that $\varepsilon_{k} \rightarrow 0$ as $k \rightarrow \infty$. Using Lemma \ref{negative part reabsorb}, we deduce from \eqref{PS cond gen 1} that $\left\{u_{k}\right\}$ is bounded in $X_{p}(\Omega)$, provided that $\gamma$ is sufficiently small. Indeed, we have
$$  (1-c_0\gamma)\|u_k\|^p_{X_p(\Omega)}\leq \|u_k\|^{p}_{X_p(\Omega)}-\int_{[0,\bar{s})}[u_k]_{s,p}^p d\mu^{-}(s) = p\mathfrak{I}_p\left(u_k\right) \leq pM\text{ for every } k \in \mathbb{N}.$$
Hence, there exists a subsequence, still denoted by $\{u_{k}\}$, and $u \in X_{p}(\Omega)$ such that 
\begin{equation*}
    u_{k} \rightharpoonup u \quad \text{weakly in } X_{p}(\Omega), \quad \text{and} \quad u_{k} \to u \quad \text{strongly in } L^{r}(\Omega) \text{ for all } 1 \leq r < p_{s_\sharp}^{*}.
\end{equation*}
By taking $v = u_{k}$ as a test function in \eqref{PS cond gen 2}, we obtain
\begin{align*}
  |a_{k}| & \leq \left|\|u_{k}\|_{X_p(\Omega)}^{p}-\int_{[0,\bar{s})} [u_{k}]_{s,p}^{p}d\mu^{-}(s)+\varepsilon_{k}\left\|u_{k}\right\|_{X_p(\Omega)}\right|  \\ &\leq  (1+c_0 \gamma) \left\|u_{k}\right\|^p_{X_p(\Omega)}+\varepsilon_k\left\|u_{k}\right\|_{X_p(\Omega)} \leq C,
\end{align*}
which implies the boundedness of the sequence $\left\{a_{k}\right\}$. 

Next, we aim to prove that $u_{k} \to u$ strongly in $X_{p}(\Omega)$. To this end, we choose $v = u_{k} - u$
as a test function in \eqref{PS cond gen 2}, which yields
\begin{align}\label{PS cond gen 3.1}
\begin{split} 
&\left|\langle\mathfrak{I}_{p}^{\prime}\left(u_{k}\right), u_{k}-u\rangle\right| \\
&= \Bigg|\int_{[0,1]} c_{N, s, p} \iint_{\mathbb{R}^{2 N}} \frac{|u_{k}(x)-u_{k}(y)|^{p-2}(u_{k}(x)-u_{k}(y))}{|x-y|^{N+sp}} ((u_k-u)(x)-(u_k-u)(y)) d x d y d \mu^{+}(s) \\
& -\int_{[0,\bar{s})} c_{N, s, p} \iint_{\mathbb{R}^{2 N}} \frac{|u_{k}(x)-u_{k}(y)|^{p-2}(u_{k}(x)-u_{k}(y))}{|x-y|^{N+sp}}\\
& \times ((u_k-u)(x)-(u_k-u)(y))d x d y d \mu^{-}(s)\Bigg|\\
&\leq O\left(\varepsilon_{k}\right)+\left|a_{k}\right|\left\|u_{k}\right\|_{L^{p}(\Omega)}^{p-1}\left\|u_{k}-u\right\|_{L^{p}(\Omega)} \rightarrow 0, \text{ as } k\to \infty.  
\end{split}
\end{align}

By the Brezis-Lieb Lemma \ref{Brezis-Lieb lemma}, we have
\begin{align}\label{PS cond gen 4}
\begin{split}
\int_{[0,1]}\left[u_{k}-u\right]_{s,p}^{p}d\mu^{+}(s)  & =\int_{[0,1]}\left[u_{k}\right]_{s,p}^{p}d\mu^{+}(s)-\int_{[0,1]}[u]_{s,p}^{p}d\mu^{+}(s)+o_{k}(1),  \\
\int_{[0,\bar{s})}\left[u_{k}-u\right]_{s,p}^{p}d\mu^{-}(s)  & =\int_{[0,\bar{s})}\left[u_{k}\right]_{s,p}^{p}d\mu^{-}(s)-\int_{[0,\bar{s})}[u]_{s,p}^{p}d\mu^{-}(s)+o_{k}(1). 
\end{split}
\end{align}
From Lemma \ref{tech lem for PS gen}, (if $\mu^+\{1\}>0$), and $u_{k}(x) \rightarrow u(x)$ pointwise a.e. in $\Omega$ as $k \rightarrow \infty$, we deduce that
\begin{align*}
\left|\nabla u_{k}(x)\right|^{p-2} \nabla u_{k}(x) & \rightarrow|\nabla u(x)|^{p-2} \nabla u(x)  ~\text { pointwise a.e. in } \Omega, \\
\left|u_{k}(x)\right|^{p-2} u_{k}(x) & \rightarrow |u(x)|^{p-2} u(x) ~\text { pointwise a.e. in } \Omega, ~\text{ and}
\end{align*}
\begin{align*}
     &c_{N,s,p}^{\frac{1}{p^{\prime}}}\frac{|u_{k}(x)-u_{k}(y)|^{p-2}\left(u_{k}(x)-u_{k}(y)\right)}{|x-y|^{\frac{N+s p}{p^{\prime}}}}  \longrightarrow c_{N,s,p}^{\frac{1}{p^{\prime}}}\frac{|u(x)-u(y)|^{p-2}(u(x)-u(y))}{|x-y|^{\frac{N+s p}{p^{\prime}}}}  
\end{align*}
 pointwise a.e. in $\mathbb{R}^{2 N}\times(0,1)$ as $k \rightarrow \infty$,  where $p^{\prime}=\frac{p}{p-1}$ is the Lebesgue conjugate of $p.$

Moreover, $\{\left|\nabla u_{k}\right|^{p-2} \nabla u_{k}\}$ and $\{\left|u_{k}\right|^{p-2} u_{k}\}$ are bounded in $L^{p^{\prime}}(\Omega),$ and the sequence $$\left\{c_{N,s,p}^{\frac{1}{p^{\prime}}}\frac{|u_{k}(x)-u_{k}(y)|^{p-2}\left(u_{k}(x)-u_{k}(y)\right)}{|x-y|^{\frac{N+s p}{p^{\prime}}}}\right\}$$  is bounded in both spaces $L^{p^{\prime}}\left(\mathbb{R}^{2 N} \times(0,1), d x d y d \mu^{+}(s)\right)$ and $L^{p^{\prime}}\left(\mathbb{R}^{2 N} \times(0,1), d x d y d \mu^{-}(s)\right)$. Therefore, all these aforementioned sequences will converge weakly to some limits. Since the weak and pointwise limits coincide, we have that
\begin{equation*}
    \int_{\Omega}\left|\nabla u_{k}\right|^{p-2} \nabla u_{k}\cdot \nabla u dx \rightarrow \int_{\Omega}|\nabla u|^{p}dx,
\end{equation*}
\begin{equation*}
    \int_{\Omega} \left|u_{k}\right|^{p-2} u_{k} u d x  \rightarrow \int_{\Omega} |u|^{p} d x,
\end{equation*}
\begin{align*}
&\int_{(0,1)}c_{N,s,p}\iint_{\mathbb{R}^{2N}} \frac{|u_{k}(x)-u_{k}(y)|^{p-2}\left(u_{k}(x)-u_{k}(y)\right)(u(x)-u(y))}{|x-y|^{N+s p}}dxdyd\mu^{+}(s)   \\& \quad\longrightarrow
\int_{(0,1)}c_{N,s,p}\iint_{\mathbb{R}^{2N}}  \frac{|u(x)-u(y)|^{p}}{|x-y|^{N+s p}}dxdyd\mu^{+}(s),
\end{align*}
and 
\begin{align*}
&\int_{(0,1)}c_{N,s,p}\iint_{\mathbb{R}^{2N}} \frac{|u_{k}(x)-u_{k}(y)|^{p-2}\left(u_{k}(x)-u_{k}(y)\right)(u(x)-u(y))}{|x-y|^{N+s p}}dxdyd\mu^{-}(s)  \\& \quad \longrightarrow\int_{(0,1)}c_{N,s,p}\iint_{\mathbb{R}^{2N}}  \frac{|u(x)-u(y)|^{p}}{|x-y|^{N+s p}}dxdyd\mu^{-}(s),
\end{align*}
as $k \rightarrow \infty$. From this, we get 
\begin{align}\label{PS cond gen 5}
&\int_{[0,1]}c_{N,s,p}\iint_{\mathbb{R}^{2N}} \frac{|u_{k}(x)-u_{k}(y)|^{p-2}\left(u_{k}(x)-u_{k}(y)\right)(u(x)-u(y))}{|x-y|^{N+s p}}dxdyd\mu^{+}(s)  \\ 
    & \quad \longrightarrow \int_{[0,1]}c_{N,s,p}\iint_{\mathbb{R}^{2N}}  \frac{|u(x)-u(y)|^{p}}{|x-y|^{N+s p}}dxdyd\mu^{+}(s), \nonumber
\end{align}
\begin{align}\label{PS cond gen 5.1}
&\int_{[0,\bar{s})}c_{N,s,p}\iint_{\mathbb{R}^{2N}} \frac{|u_{k}(x)-u_{k}(y)|^{p-2}\left(u_{k}(x)-u_{k}(y)\right)(u(x)-u(y))}{|x-y|^{N+s p}}dxdyd\mu^{-}(s)  \\
& \quad \longrightarrow \int_{[0,\bar{s})}c_{N,s,p} \iint_{\mathbb{R}^{2N}}  \frac{|u(x)-u(y)|^{p}}{|x-y|^{N+s p}}dxdyd\mu^{-}(s),\nonumber
\end{align}
as $k \rightarrow \infty$. Using \eqref{PS cond gen 3.1}, \eqref{PS cond gen 4}, \eqref{PS cond gen 5}, \eqref{PS cond gen 5.1}, and Lemma \ref{negative part reabsorb} we obtain
\begin{align*}
 0&=\lim_{k \rightarrow \infty} \left\langle\mathfrak{I}_{p}^{\prime}\left(u_{k}\right), u_{k}-u\right\rangle =\lim_{k \rightarrow \infty} \left\langle\mathfrak{I}_{p}^{\prime}\left(u_{k}\right), u_{k}\right\rangle- \lim_{k \rightarrow \infty}\left\langle\mathfrak{I}_{p}^{\prime}\left(u_{k}\right), u\right\rangle\\
& =\lim_{k \rightarrow \infty} \|u_{k}\|^{p}_{X_p(\Omega)}- \lim_{k \rightarrow \infty}\int_{[0,\bar{s})}[u_{k}]^{p}_{s,p}d\mu^{-}(s)\\
&-\lim_{k \rightarrow \infty}\int_{[0,1]}c_{N,s,p}\iint_{\mathbb{R}^{2N}} \frac{|u_{k}(x)-u_{k}(y)|^{p-2}\left(u_{k}(x)-u_{k}(y)\right)(u(x)-u(y))}{|x-y|^{N+s p}}dxdyd\mu^{+}(s)\\
& + \lim_{k \rightarrow \infty}\int_{[0,\bar{s})}c_{N,s,p}\iint_{\mathbb{R}^{2N}} \frac{|u_{k}(x)-u_{k}(y)|^{p-2}\left(u_{k}(x)-u_{k}(y)\right)(u(x)-u(y))}{|x-y|^{N+s p}}dxdyd\mu^{-}(s)\\
&=\lim_{k \rightarrow \infty}\|u_{k}\|^{p}_{X_p(\Omega)}-\lim_{k \rightarrow \infty}\int_{[0,\bar{s})}[u_{k}]^{p}_{s,p}d\mu^{-}(s)-\|u\|^{p}_{X_p(\Omega)}+\int_{[0,\bar{s})}[u]^{p}_{s,p}d\mu^{-}(s)\\
& =\lim_{k \rightarrow \infty} \|u_{k}-u\|^{p}_{X_p(\Omega)}- \lim_{k \rightarrow \infty}\int_{[0,\bar{s})}[u_{k}-u]^{p}_{s,p}d\mu^{-}(s)\geq (1-c_{0}\gamma) \lim_{k \rightarrow \infty}\|u_{k}-u\|^{p}_{X_p(\Omega)}.
\end{align*}
Consequently,  we conclude that $u_k\rightarrow u$ converges strongly in $X_{p}(\Omega)$ as $k \rightarrow \infty$, provided that $\gamma$ is sufficiently small.
\end{proof}

Now, we state the main result of this section.

\begin{thm}\label{principal eigenvalue P}
   Let $\Omega \subset \mathbb{R}^{N}$ be a bounded domain with Lipschitz boundary, and let $\mu$ satisfy \eqref{measure1}--\eqref{measure3}. Let $s_{\sharp}$ be as in \eqref{measure4}, and assume $1 < p < N/s_{\sharp}$. Then there exists a constant $\gamma_{0} > 0$, depending only on $N$, $\Omega$, and $p$, such that, for all  $\gamma \in [0, \gamma_{0}]$ in \eqref{measure3}, the statements below concerning the eigenvalues and eigenfunctions of problem \eqref{Eig problem P} associated with $\mathcal{L}_{\mu, p}$ hold.
\begin{itemize}
    \item[(i)] The first eigenvalue $\lambda_{1, \mu}(\Omega)$ is given by
\begin{equation}\label{first eigenvalue}
\lambda_{1, \mu}(\Omega):=\inf_{u \in X_{p}(\Omega) \backslash\{0\}} \frac{\int_{[0,1]}[u]_{s,p}^{p} d \mu^{+}(s)-\int_{[0, \bar{s})}[u]_{s,p}^{p} d \mu^{-}(s)}{\int_{\Omega}|u|^{p} d x} . 
\end{equation}
\item[(ii)] There exists a  function $e_{1, \mu} \in X_p(\Omega)$, an eigenfunction corresponding to the eigenvalue $\lambda_{1, \mu}(\Omega) $ which attains the minimum in \eqref{first eigenvalue}.

\item[(iii)] The set of eigenvalues of the problem \eqref{Eig problem P} consists of a sequence $(\lambda_{n, \mu})$ with
  \begin{align}
      0<\lambda_{1, \mu} \leq \lambda_{2, \mu} \leq \ldots \leq \lambda_{n, \mu} \leq \lambda_{n+1, \mu} \leq \ldots~\text{and }&
      \lambda_{n, \mu} \rightarrow \infty\,\,\,\text{as}\,\,n \rightarrow \infty.
  \end{align}
  
  \item [(iv)] In addition, if  $\mu$ satisfies \eqref{measure5}, then every eigenfunction corresponding to the eigenvalue $\lambda_{1, \mu}(\Omega)$ in \eqref{first eigenvalue} does not change sign.
\end{itemize}

\end{thm}

\begin{proof} \begin{enumerate}[(i)]
    \item This is an immediate consequence of Lemma \ref{tech lem for principal eigenvalue}. In particular,  \eqref{minimizer exist} with $X_{0}:=X_{p}(\Omega)$ asserts that 
the expression for $\lambda_{1, \mu}(\Omega)$ introduced in \eqref{first eigenvalue} exists. Moreover, using \eqref{relation P}, we conclude that $\lambda_{1, \mu}(\Omega)$ is an eigenvalue.

\item Let $e_{1, \mu} \in X_p(\Omega)$ be an eigenfunction corresponding to $\lambda_{1, \mu}(\Omega)$. Then, $e_{1, \mu}$ is a weak solution to \eqref{Eig problem P}. Thus, we have 
\begin{equation}\label{nn11}
    \left\langle \mathfrak{I}_{p}^{\prime}(e_{1, \mu}), v\right\rangle=\lambda \int_{\Omega} |e_{1, \mu}|^{p-2} e_{1, \mu} v d x
\end{equation}
for every $v \in X_p(\Omega).$ Taking $v=e_{1, \mu} \in X_p(\Omega)$ as a test function in \eqref{nn11} and recalling \eqref{functional Hp}, we obtain
$$
\int_{[0,1]}\left[e_{1, \mu}\right]_{s,p}^{p} d \mu^{+}(s)-\int_{[0, \bar{s})}\left[e_{1, \mu}\right]_{s,p}^{p} d \mu^{-}(s)=\lambda_{1, \mu}(\Omega) \int_{\Omega} |e_{1, \mu}(x)|^{p} d x.
$$
This implies that $e_{1, \mu}$ is a minimizer of the first eigenvalue given by the expression in \eqref{first eigenvalue}.

\item  Note that $\mathfrak{I}_{p}$ is even, $\mathfrak{I}_{p}(0)=0$, while $\mathcal{M}$ is a complete, symmetric, and $C^{1,1}$-manifold in $X_{p}(\Omega)$. Moreover, from Lemma \ref{PS cond gen}, we get that $\mathfrak{I}_{p}$ is bounded from below on $\mathcal{M}$ and satisfies the $(PS)$ condition on $\mathcal{M}$. Therefore, from the application of the {\it Lusternik-Schnirelmann theory} (see Theorem \ref{tech them for seq gen}), there exists an unbounded sequence of eigenvalues $(\lambda_{k, \mu})$ for the problem \eqref{Eig problem P} such that
\begin{align}
      0<\lambda_{1, \mu} \leq \lambda_{2, \mu} \leq \ldots \leq \lambda_{n, \mu} \leq \lambda_{n+1, \mu} \leq \ldots.
  \end{align}

It remains to show that $\lambda_{k, \mu} \rightarrow \infty$ as $k \rightarrow \infty$. We prove it by contradiction. Suppose there exists $L>0$ such that $0<\lambda_{k, \mu} \leq L$ for all $k \in \mathbb{N}$. Since, $X_{p}(\Omega)$ is separable and reflexive (as a consequence of being uniformly convex space), $X_{p}(\Omega)$ admits a biorthogonal system $\left\{w_{k}, w_{k}^{*}\right\}$ with the following properties: $X_{p}(\Omega)=\overline{\operatorname{span}\left\{w_k: k \in \mathbb{N}\right\}}$ such that for all $w_k^* \in\left(X_{p}(\Omega)\right)^*$ we have $\left\langle w_i^*, w_j\right\rangle=\delta_{i, j}$. Moreover, $\left\langle w_k^*, v\right\rangle=0,$ $\forall k \in \mathbb{N}$ implies that $v=0$, (see \cite{Pelczyński}).
Let
\begin{equation*}
    X_k=\overline{\operatorname{span}\left\{w_k, w_{k+1}, \cdots\right\}} \text { and } \quad a_k=\inf _{A \in \Sigma_k} \sup _{u \in A \cap X_k} p \mathfrak{I}_{p}(u).
\end{equation*}
Note that the co-dimension of $X_k$ is $k-1$. Recall the following property of genus (See, \cite[Proposition 2.3]{Szulkin}): Let $Z$ be a subspace of $X_{p}(\Omega)$ with codimension $k$ and $\gamma(A)>k$, then $A \cap Z \neq \emptyset$. Using this property, we have $A \cap X_k \neq \emptyset$ for all $A \in \Sigma_k$. This shows that $\sup_{u \in A \cap X_k} p \mathfrak{I}_{p}(u)>0$. Furthermore, $a_k \leq \lambda_{k, \mu} \leq L,$ $\forall k \in \mathbb{N}$ by the definition of $a_k$ and characterisation of $\lambda_{k, \mu}$. Now, for each $k \in \mathbb{N}$, choose $v_k \in A \cap X_k$ such that
\begin{equation*}
    \int_{\Omega} \left|v_k\right|^p d x=1 \text { and } 0 \leq a_k \leq p \mathfrak{I}_{p}\left(v_k\right) \leq L+1,\\\ k \in \mathbb{N}.
\end{equation*}
This implies that $\{v_k\}$ is a bounded sequence in $X_{p}(\Omega)$. Therefore, proceeding as in the proof of Lemma \ref{tech lem for principal eigenvalue}, we ensure the existence of an element $v \in X_{p}(\Omega)$ such that $v_k \rightharpoonup v$ in $X_{p}(\Omega)$ up to a subsequence with $\int_{\Omega} |v|^p d x=1$. Therefore $v \not \equiv 0$ in $\Omega$. However, by the choice of the biorthogonal system, we have
\begin{equation*}
    \left\langle w_m^*, v\right\rangle=\lim _{k \rightarrow \infty}\left\langle w_m^*, v_k\right\rangle=0, \text { for every } m \in \mathbb{N},
\end{equation*}
which implies $v=0$. This gives a contradiction. Hence, $\lambda_{k, \mu} \rightarrow \infty$ as $k \rightarrow \infty$. This completes the proof.

\item We now prove that every eigenfunction $e_{1, \mu}$ corresponding to $\lambda_{1, \mu}(\Omega)$ does not change sign if the measure $\mu$ satisfies condition \eqref{measure5}.

Let us assume that $e_{1, \mu}$ changes sign in $\Omega$. Observe that, we have $$\left\|\left|e_{1, \mu}\right|\right\|_{L^{p}(\Omega)}=\left\|e_{1, \mu}\right\|_{L^{p}(\Omega)}.$$ Moreover, for any $s \in[0,1]$ and any $p\geq 1$ it holds that $\left[\left|e_{1, \mu}\right|\right]_{s,p} \leq \left[e_{1, \mu}\right]_{s,p}$. Hence, we have
$$
\int_{[0,1]}\left[\left|e_{1, \mu}\right|\right]_{s,p}^{p} d \mu^{+}(s) \leq  \int_{[0,1]}\left[e_{1, \mu}\right]_{s,p}^{p} d \mu^{+}(s),
$$
which implies that $\left|e_{1, \mu}\right| \in X_{p}(\Omega)$. In addition, applying Lemma \ref{energycomparison} (under the condition \eqref{measure5}), and using the fact that $e_{1, \mu}$ changes sign in $\Omega$, we obtain the following inequality:
\begin{align*}
    \int_{[0,1]}\left[\left|e_{1, \mu}\right|\right]_{s,p}^{p} d \mu^{+}(s)&-\int_{[0,1]}\left[\left|e_{1, \mu}\right|\right]_{s,p}^{p} d \mu^{-}(s)\\&\quad\quad\quad \quad < \int_{[0,1]}\left[e_{1, \mu}\right]_{s,p}^{p} d \mu^{+}(s)-\int_{[0,1]}\left[e_{1, \mu}\right]_{s,p}^{p} d \mu^{-}(s).
\end{align*}
But this contradicts the fact that $e_{1, \mu}$ is a minimizer as in part (ii). Therefore, $e_{1, \mu}$ cannot change sign in $\Omega$. This completes the proof.
\qedhere
\end{enumerate}  \end{proof}


\section{Weak maximum principles for nonlinear superposition operators} \label{sec4}
In this section, we will prove the weak maximum principle. 
 Consider the problem
\begin{align}\label{weakmax}
\mathcal{L}_{\mu,p}^{+}  u=0~\text{in}~\Omega,\nonumber\\
u=0~\text{in}~{{\mathbb{R}^N}}\setminus\Omega.
\end{align}

We say that a function $u\in X_p(\Omega)$ is a weak solution  of \eqref{weakmax}, if for every $\phi \in X_p(\Omega)$, we have
\begin{equation}\label{weaksolu}
\int_{[0,1]} c_{N, p, s}\int_{{\mathbb{R}^N}}\int_{{\mathbb{R}^N}}\frac{|u(x)-u(y)|^{p-2}(u(x)-u(y))(\phi(x)-\phi(y))}{|x-y|^{N+ps}}dxdy d\mu^+(s) =0.
\end{equation}
Moreover, if $u \in X_p(\Omega)$ satisfies the following inequality
\begin{equation}\label{weaksolu1}
\int_{[0,1]} c_{N, p, s}\int_{{\mathbb{R}^N}}\int_{{\mathbb{R}^N}}\frac{|u(x)-u(y)|^{p-2}(u(x)-u(y))(\phi(x)-\phi(y))}{|x-y|^{N+ps}}dxdy d\mu^+(s) \geq  0
\end{equation} for any nonnegative $\phi \in X_p(\Omega),$ then we say that $u \in X_p(\Omega)$ satisfies $\mathcal{L}_{\mu, p}^+u \geq 0$ in $\Omega$ in the weak sense.

\begin{thm}
    Let $\Omega \subset \mathbb{R}^N$ be an open subset with Lipschitz boundary. We assume that $\mu^+$ satisfies \eqref{measure1} and $1<p<\frac{N}{s_\sharp}$, where $s_\sharp$ is defined by \eqref{measure4}. Let $u \in X_p(\Omega)$ be such that $\mathcal{L}_{\mu, p}^+ u \geq 0$ in $\Omega$ in the weak sense and $u \geq 0$ a.e. in $\mathbb{R}^N \backslash \Omega.$ Then, $u \geq 0$ a.e. in $\Omega.$
\end{thm}
\begin{proof} We will prove this result by contradiction. For this purpose, let us assume that there exists a set $E \subset \Omega$ with $|E|>0$ such that $u <0$ a.e. in $E.$

Note that $u \geq 0$ a.e. in $\mathbb{R}^N \backslash \Omega,$ therefore, it yields that $u^{-} = 0$ a.e. in $\mathbb{R}^N \backslash \Omega.$ It is also easy to see that $u^{-} \in X_p(\Omega).$  Therefore, using $u^{-}$ as a test function in \eqref{weaksolu1} we get 
\begin{align}
  0 \leq &  \int_{[0,1]} c_{N, p, s}\int_{{\mathbb{R}^N}}\int_{{\mathbb{R}^N}}\frac{|u(x)-u(y)|^{p-2}(u(x)-u(y))(u^{-}(x)-u^{-}(y))}{|x-y|^{N+ps}}dxdy d\mu^+(s) \nonumber\\& = \int_{[0,1)} c_{N, p, s}\int_{{\mathbb{R}^N}}\int_{{\mathbb{R}^N}}\frac{|u(x)-u(y)|^{p-2}(u(x)-u(y))(u^{-}(x)-u^{-}(y))}{|x-y|^{N+ps}}dxdy d\mu^+(s) \nonumber \\&+ \mu^+(\{1\}) \int_{\Omega} |\nabla u (x)|^{p-2} \nabla u(x) \cdot \nabla u^{-}(x) \,dx.
\end{align}
Now, we use the following pointwise inequality (see \cite[Lemma A.2]{BP16} or \cite[Lemma C.2]{BLP2014})  $$|u(x)-u(y)|^{p-2}(u(x)-u(y))(u^{-}(x)-u^{-}(y)) \leq - |u^{-}(x)-u^{-}(y)|^p $$ to get 
\begin{align} \label{equ6.5}
    0 & \leq \int_{[0,1)} c_{N, p, s}\int_{{\mathbb{R}^N}}\int_{{\mathbb{R}^N}}\frac{|u(x)-u(y)|^{p-2}(u(x)-u(y))(u^{-}(x)-u^{-}(y))}{|x-y|^{N+ps}}dxdy d\mu^+(s) \nonumber \\&+ \mu^+(\{1\}) \int_{\Omega} |\nabla u (x)|^{p-2} \nabla u(x) \cdot \nabla u^{-}(x) \,dx\nonumber \\ \leq& - \int_{[0,1)} c_{N, p, s}\int_{{\mathbb{R}^N}}\int_{{\mathbb{R}^N}}\frac{|(u^{-}(x)-u^{-}(y))|^p}{|x-y|^{N+ps}}dxdy d\mu^+(s) - \mu^+(\{1\}) \int_{\Omega}  |\nabla u^{-}(x)|^p \,dx.
    \end{align}
    Now, if $\mu^+(\{1\}) >0,$ we infer from \eqref{equ6.5} that 
    $$\int_{\Omega}  |\nabla u^{-}(x)|^p \,dx \leq 0$$
    and thus, $u^{-}$ is constant and therefore, equal to zero  as $u^{-}$ has zero trace values along $\partial \Omega.$ This is a contradiction to the existence of a set $E$ as above. 

Next, let us suppose that $\mu^{+}(\{1\})=0.$ Thus, condition \eqref{measure1} implies that $\mu^+(\bar{s}, 1)>0$ and so from \eqref{equ6.5} it follows that 
\begin{align}
    0 \geq& \int_{(\bar{s},1)} c_{N, p, s}\int_{{\mathbb{R}^N}}\int_{{\mathbb{R}^N}}\frac{|(u^{-}(x)-u^{-}(y))|^p}{|x-y|^{N+ps}}dxdy d\mu^+(s) \nonumber \\& \geq \int_{(\bar{s},1)} c_{N, p, s}\int_{{E}}\int_{{\mathbb{R}^N \backslash \Omega}}\frac{|(u^{-}(x)|^p}{|x-y|^{N+ps}}dxdy d\mu^+(s) >0,
\end{align}
which is again a contradiction. Therefore, our assumption of the existence of the set $E$ is not correct. Therefore, $u \geq 0$ a.e. in $\Omega,$ completing the proof of the theorem. 
\end{proof}

\section{Strong minimum/maximum principles for nonlinear superposition operators} \label{sec5}
The main purpose of this section is to prove the strong minimum principle for the operator $\mathcal{L}_{\mu,p}^{+}.$
 Consider the problem
\begin{align}\label{p-kuusi}
\mathcal{L}_{\mu,p}^{+}  u=0~\text{in}~\Omega,\nonumber\\
u=0~\text{in}~{{\mathbb{R}^N}}\setminus\Omega.
\end{align} 

We say that a function $u\in X_p(\Omega)$ is a weak subsolution (or supersolution) of \eqref{p-kuusi}, if for every nonnegative $\phi \in X_p(\Omega)$, we have
\begin{equation}\label{k-subsup}
\int_{[0,1]} c_{N, p, s}\int_{{\mathbb{R}^N}}\int_{{\mathbb{R}^N}}\frac{|u(x)-u(y)|^{p-2}(u(x)-u(y))(\phi(x)-\phi(y))}{|x-y|^{N+ps}}dxdy d\mu^+(s) \leq (or \geq) 0.
\end{equation}

A function $u\in X_p(\Omega)$ is a weak solution of \eqref{p-kuusi}, if it is a weak subsolution as well as a weak supersolution of \eqref{k-subsup}. In particular, for every $\phi \in X_p(\Omega)$, $u$ satisfies
\begin{equation}\label{k-sol}
\int_{[0,1]} c_{N, p, s}\int_{{\mathbb{R}^N}}\int_{{\mathbb{R}^N}}\frac{|u(x)-u(y)|^{p-2}(u(x)-u(y))(\phi(x)-\phi(y))}{|x-y|^{N+ps}}dxdy d\mu^+(s)= 0.
\end{equation}

We define the nonlocal tail of a function $v\in X_p({\Omega})$ in a ball $B_R(x_0)\subset{{\mathbb{R}^N}}$ given by 
\begin{equation}
Tail(v,x_0,R)=\left[ \int_{(0, 1)} R^{sp} \left( \int_{{{\mathbb{R}^N}}\setminus{B_R(x_0)}}\frac{|v(x)|^{p-1}}{|x-x_0|^{N+ps}}dx \right)  d\mu^+(s) \right]^{\frac{1}{p-1}}.
\end{equation}
Clearly, for any $v\in L^r({{\mathbb{R}^N}}), r\geq p-1$ and $R>0$, we have $Tail(v,x_0,R)$ is finite, by using the H\"{o}lder inequality. It is important to mention here that the notion of nonlocal tail was first introduced by DiCastro et al. \cite{DKP16}.

The next aim is to establish a minimum principle for the problem \eqref{p-kuusi}. Prior to that, we will prove the following logarithmic estimate, which will be used to prove the minimum principle.
 \begin{lem}\label{k-log-lemma} 
 Let $\Omega \subset \mathbb{R}^{N}$ be a bounded domain with Lipschitz boundary, and assume that $\mu = \mu^{+}$ satisfies \eqref{measure1}, with $s_{\sharp}$ defined as in \eqref{measure4}. Let $1 < p < \infty$, and suppose that $u \in X_p(\Omega)$ is a weak supersolution of \eqref{p-kuusi} such that $u \geq 0$ in $B_R := B_R(x_0) \subset \Omega$.
  Then for any $B_r:=B_r(x_0)\subset B_{\frac{R}{2}}(x_0)$ and for any $d>0$, the following estimate holds:
 	\begin{align}\label{k-log-ineq}
 	\mu^+(\{1\})&\int_{B_{r}} |\nabla \log (u+d)|^p  dx+ \int_{(0, 1)} c_{N, p, s}\int_{B_r}\int_{B_r}\left|\log\frac{u(x)+d}{u(y)+d}\right|^p\frac{dxdy}{|x-y|^{N+ps}} d\mu^+(s) \nonumber\\&\leq Cr^{N} \sup_{s \in \Sigma} r^{-sp}\left(d^{1-p} \sup_{s \in \Sigma}  \left(\frac{r}{R}\right)^{sp}[Tail(u_{-},x_0,R)]^{p-1}+1\right) \nonumber \\&+   C\mu^+(\{1\}) r^{N-p}+C\mu^+(\{0\})r^N,
 	\end{align}
 	where $C=C(N,p,\Sigma, \mu^+)$, $\Sigma:=\text{supp}(\mu^+)\cap(0,1),$ $u_{-}$ is the negative part of $u$, that is, $u_{-}:=\max\{-u, 0\}$.
 \end{lem}

 \begin{proof}
  Let us first recall the following important inequality (see \cite[Lemma 3.1]{DKP16}): 
 	for $p\geq1$ and $\epsilon\in(0,1]$, we have that
 	\begin{align}\label{k-ineq}
 	|a|^p\leq  |b|^p +c_p\epsilon|b|^p+c_p(1+c_p\epsilon)\epsilon^{1-p}|a-b|^p,
 	\end{align} for all $a,b\in\mathbb{R},$
 	where $c_p:=(p-1)\Gamma(\max\{1,p-2\}).$ 
	
 We will now proceed to prove the main estimate of this lemma. Let $d>0$ and $\eta \in C_c^{\infty}({\mathbb{R}^N})$ be such that \begin{equation}
 0\leq\eta\leq1,~~~\eta\equiv1~\text{in}~B_r,~~~\eta\equiv0~\text{in}~{\mathbb{R}^N}\setminus B_{3r/2}~~~~\text{and}~|\nabla \eta|<Cr^{-1}.
 \end{equation}	
 	Since $u(x)\geq0$ for all $x\in supp(\eta)$,  $\psi=(u+d)^{1-p}\eta^p$ is a well-defined test function for \eqref{k-subsup}. Thus, we get
	
 \begin{align}\label{k-3.5}
 &\mu^+(\{0\}) \int_{\Omega} |u(x)|^{p-2}u(x) \frac{\eta^p(x)}{(u(x)+d)^{p-1}} dx \nonumber\\
 &+ \mu^+(\{1\}) \int_{\Omega} |\nabla u|^{p-2} \Big\langle \nabla u, \nabla \left( \frac{\eta^p(x)}{(u(x)+d)^{p-1}} \right) \Big\rangle dx  \nonumber \\&+ \int_{(0,1)} c_{N, p, s}\int_{B_{2r}} \int_{B_{2r}}\frac{|u(x)-u(y)|^{p-2}(u(x)-u(y))}{|x-y|^{N+ps}}\\&\quad\quad\quad \quad \times \left[\frac{\eta^{p}(x)}{(u(x)+d)^{p-1}}-\frac{\eta^{p}(y)}{(u(y)+d)^{p-1}}\right]dxdy\,d\mu^+(s)\nonumber\\
 &+2\int_{(0,1)} c_{N, p, s}\int_{{\mathbb{R}^N}\setminus B_{2r}} \int_{B_{2r}}\frac{|u(x)-u(y)|^{p-2}(u(x)-u(y))}{|x-y|^{N+ps}}\frac{\eta^{p}(x)}{(u(x)+d)^{p-1}}dxdyd\mu^+(s) \geq 0.\nonumber
 \end{align}

 	We will estimate each term individually. Set
 	\begin{align} \label{I3}
 	    I_1:= \mu^+(\{0\}) \int_{\Omega} |u(x)|^{p-2}u(x) \frac{\eta^p(x)}{(u(x)+d)^{p-1}} dx
 	\end{align}
     \begin{align} \label{I4}
         I_2:= \mu^+(\{1\}) \int_{\Omega} |\nabla u|^{p-2} \Big\langle \nabla u, \nabla \left( \frac{\eta^p(x)}{(u(x)+d)^{p-1}} \right) \Big\rangle dx 
     \end{align}
 	\begin{align}
 	I_3=\int_{(0,1)} c_{N, p, s} \int_{B_{2r}} \int_{B_{2r}}&\frac{|u(x)-u(y)|^{p-2}(u(x)-u(y))}{|x-y|^{N+ps}} \nonumber \\& \times \left[\frac{\eta^{p}(x)}{(u(x)+d)^{p-1}}-\frac{\eta^{p}(y)}{(u(y)+d)^{p-1}}\right]dxdy d\mu^+(s)\label{k-I1}
     \end{align}
     \begin{align*}
 	I_4&=2\int_{(0,1)} c_{N, p, s}\int_{{\mathbb{R}^N} \setminus B_{2r}} \int_{B_{2r}}\frac{|u(x)-u(y)|^{p-2}(u(x)-u(y))}{|x-y|^{N+ps}}\frac{\eta^{p}(x)}{(u(x)+d)^{p-1}}dxdy d\mu^+(s).
 	\end{align*}
 Let us first estimate $I_1$ and $I_2.$
 It is easy to see that 
 \begin{align}\label{I_1est}
      I_1:= \mu^+(\{0\}) \int_{\Omega} |u(x)|^{p-2}u(x) \frac{\eta^p(x)}{(u(x)+d)^{p-1}} dx \leq C\mu^+(\{0\}) r^{N}.
 \end{align}
 Next, let us estimate $I_2.$ For this, let us observe using the weighted Young inequality that 
 \begin{align} \label{I_2est}
      I_2&:= \mu^+(\{1\}) \int_{\Omega} |\nabla u|^{p-2} \Big\langle \nabla u, \nabla \left( \frac{\eta^p(x)}{(u(x)+d)^{p-1}} \right) \Big\rangle dx \nonumber \\& \leq p \mu^+(\{1\})   \int_{\Omega} |\nabla u|^{p-1} \frac{\eta(x)^{p-1} |\nabla \eta|  }{(u(x)+d)^{p-1}} dx - (p-1) \mu^+(\{1\})   \int_{\Omega} \frac{|\nabla u|^{p}  \eta^{p}(x)}{(u(x)+d)^{p}} dx \nonumber \\& \leq \frac{p-1}{2}  \mu^+(\{1\})\int_{\Omega} \frac{|\nabla u|^p}{(u+d)^p} \eta^p dx +2^{p-1}   \mu^+(\{1\}) \int_{\Omega} |\nabla \eta|^p dx \nonumber \\& \quad\quad\quad \quad\quad\quad\quad- (p-1) \mu^+(\{1\})   \int_{\Omega} \frac{|\nabla u|^{p}  \eta^{p}(x)}{(u(x)+d)^{p}} dx \nonumber \\& \leq -\frac{p-1}{2}  \mu^+(\{1\})\int_{B_{r}} \frac{|\nabla u|^p}{(u+d)^p}  dx+ c' 2^{p-1}   \mu^+(\{1\}) r^{N-p} \nonumber \\& \leq -\frac{p-1}{2}  \mu^+(\{1\})\int_{B_{r}} |\nabla \log (u+d)|^p  dx+ c' 2^{p-1}   \mu^+(\{1\}) r^{N-p},
 \end{align}
 for some position constant $c'>0$.

 	Now, we will estimate $I_3$ and $I_4$. Let us assume that $u(x)>u(y)$. Observe that $u(y)\geq0$ for all $y\in B_{2r}\subset B_R$ using the support of $\eta$. Then,  choosing 
 	\begin{align} \label{corona}
 	    a=\eta(x), b=\eta(y)~\text{and}~\epsilon=l\frac{u(x)-u(y)}{u(x)+d}\in(0,1)~\text{with}~l\in(0,1)
 	\end{align}
 	in the inequality \eqref{k-ineq}, it can be estimated that

 \begin{align}\label{k-3.6}
 	&\frac{|u(x)-u(y)|^{p-2}(u(x)-u(y))}{|x-y|^{N+ps}}\left[\frac{\eta^{p}(x)}{(u(x)+d)^{p-1}}-\frac{\eta^{p}(y)}{(u(y)+d)^{p-1}}\right]\nonumber\\
 	&
 		\leq \frac{(u(x)-u(y))^{p-1}}{(u(x)+d)^{p-1}}\frac{ \eta^{p}(y)}{|x-y|^{N+ps}}\left[1+c_pl \frac{u(x)-u(y)}{u(x)+d}-\left(\frac{u(x)+d}{u(y)+d}\right)^{p-1}\right]\nonumber\\
 		&+c_pl^{1-p} \frac{|\eta(x)-\eta(y)|^{p}}{|x-y|^{N+ps}}\nonumber\\
 		&= \left(\frac{u(x)-u(y)}{u(x)+d}\right)^{p} \frac{ \eta^{p}(y)}{|x-y|^{N+ps}}\left[\frac{1-\left(\frac{u(y)+d}{u(x)+d}\right)^{1-p}}{1-\frac{u(y)+d}{u(x)+d}}+c_pl\right] +c_pl^{1-p} \frac{|\eta(x)-\eta(y)|^{p}}{|x-y|^{N+ps}}\nonumber\\
 		&:=J_1 +c_pl^{1-p} \frac{|\eta(x)-\eta(y)|^{p}}{|x-y|^{N+ps}}.
 \end{align}
 We now aim to estimate $J_1$. Consider the following function
 \begin{equation*}
 	h(t):=\frac{1-t^{1-p}}{1-t}=-\frac{p-1}{1-t} \int_{t}^{1} \tau^{-p} d\tau, \quad \forall t \in(0,1).
 \end{equation*}

 Since, the function $h_1(t)=\frac{1}{1-t} \int_{t}^{1} \tau^{-p} d\tau$ is decreasing in $t\in(0,1)$, we have $h$ is increasing in $t\in(0,1)$. Thus, we have
 \begin{equation*}
 	h(t) \leq-(p-1),~\forall\, t\in(0,1).
 \end{equation*}
 {\bf{Case-1:}} $0<t\leq\frac{1}{2}$.\\
 In this case, 
 \begin{equation*}
 	h(t) \leq-\frac{p-1}{2^{p}} \frac{t^{1-p}}{1-t}.
 \end{equation*}
 For $t=\frac{u(y)+d}{u(x)+d} \in(0,1 / 2]$, i.e. for $u(y)+d \leq \frac{u(x)+d}{2}$, we get 
 \begin{equation}\label{k-3.7}
 	J_{1} \leq \left(c_pl-\frac{p-1}{2^{p}}\right)\left[\frac{u(x)-u(y)}{u(y)+d}\right]^{p-1}\frac{ \eta^{p}(y)}{|x-y|^{N+ps}},
 \end{equation}
 since $$(u(x)-u(y))\left(\frac{(u(y)+d)^{p-1}}{(u(x)+d)^{p}} \right)=\left(\frac{u(y)+d}{u(x)+d}\right)^{p-1} - \left(\frac{u(y)+d}{u(x)+d}\right)^{p}\leq 1.$$
 By choosing $l$ as
 \begin{equation}\label{k-3.8}
 	l=\frac{p-1}{2^{p+1} c_p} \left( =\frac{1}{2^{p+1} \Gamma (\text{max} \{1, p-2\})}<1\right),
 \end{equation}
 we obtain
 \begin{equation*}
 	J_{1} \leq-\frac{p-1}{2^{p+1}} \left[\frac{u(x)-u(y)}{u(y)+d}\right]^{p-1}\frac{ \eta^{p}(y)}{|x-y|^{N+ps}}.
 \end{equation*}

 {\bf{Case-2:}} $\frac{1}{2}<t<1$.\\
 Again choosing, $t=\frac{u(y)+d}{u(x)+d} \in(1 / 2,1)$, i.e. $u(y)+d>\frac{u(x)+d}{2}$, we obtain

 \begin{align}\label{k-3.9}
 	J_{1} &\leq [c_pl-(p-1)]\left[\frac{u(x)-u(y)}{u(x)+d}\right]^{p} \frac{ \eta^{p}(y)}{|x-y|^{N+ps}}\nonumber\\
 	&- \frac{\left(2^{p+1}-1\right)(p-1)}{2^{p+1}}\left[\frac{u(x)-u(y)}{u(x)+d}\right]^{p} \frac{ \eta^{p}(y)}{|x-y|^{N+ps}}
 \end{align}
 for the choice of $l$ as in \eqref{k-3.8}.

 We note that, for $2(u(y)+d)<u(x)+d$, we have 
 \begin{equation}\label{k-3.10}
 	\left[\log \left(\frac{u(x)+d}{u(y)+d}\right)\right]^{p} \leq c_p\left[\frac{u(x)-u(y)}{u(y)+d}\right]^{p-1},
 \end{equation}
 and, for $2(u(y)+d) \geq u(x)+d,$ we derive
 \begin{equation}\label{k-3.11}
 	\left[\log \left(\frac{u(x)+d}{u(y)+d}\right)\right]^{p}=\left[\log \left(1+\frac{u(x)-u(y)}{u(y)+d}\right)\right]^{p} \leq 2^{p}\left(\frac{u(x)-u(y)}{u(x)+d}\right)^{p},
 \end{equation}
 by using $u(x)>u(y)$ and $\log (1+x)\leq x, ~ \forall x\geq 0$.

 Thus, from the estimates \eqref{k-3.6}, \eqref{k-3.7}, \eqref{k-3.9}, \eqref{k-3.10} and \eqref{k-3.11}, we obtain
 	\begin{align*}
 		&\frac{|u(x)-u(y)|^{p-2}(u(x)-u(y))}{|x-y|^{N+ps}}\left[\frac{\eta^{p}(x)}{(u(x)+d)^{p-1}}-\frac{\eta^{p}(y)}{(u(y)+d)^{p-1}}\right]\\
 		& \leq-\frac{1}{c_p}\left[\log \left(\frac{u(x)+d}{u(y)+d}\right)\right]^{p} \frac{ \eta^{p}(y)}{|x-y|^{N+ps}}+c_pl^{1-p} \frac{|\eta(x)-\eta(y)|^{p}}{|x-y|^{N+ps}}.
 	\end{align*}
 This is also true for  $u(y)>u(x)$ by exchanging $x$ and $y$. The case $u(x)=u(y)$ holds trivially. Thus, we can estimate $I_3$ in \eqref{k-I1} as
 	\begin{align}\label{k-3.12}
 		I_{3} \leq &-\frac{1}{c(p)} \int_{(0,1)} c_{N, p, s} \int_{B_{2 r}} \int_{B_{2 r}} \left|\log \left(\frac{u(x)+d}{u(y)+d}\right)\right|^{p} \frac{ \eta^{p}(y)}{|x-y|^{N+ps}}dxdy\, d\mu^+(s)\nonumber \\
 		&+c(p) \int_{(0,1)} c_{N, p, s} \int_{B_{2 r}} \int_{B_{2 r}} \frac{|\eta(x)-\eta(y)|^{p}}{|x-y|^{N+ps}} dxdy\,d\mu^+(s),
 	\end{align}
 	for some constant $c(p)$ depending on the choice of $l$.
	
 We will now estimate the term $I_{4}$. We first observe the following facts: $u(y)\geq 0$ and $u(x)-u(y)\leq u(x)$ for all  $y \in B_{R}\setminus B_{2r}$. Thus, using $(u(x)-u(y))_{+}\leq u(x)$, we get
 \begin{equation}\label{4.40}
 	\frac{(u(x)-u(y))_{+}^{p-1}}{(d+u(x))^{p-1}} \leq 1,~\forall\, x \in B_{2 r}, \,y \in B_{R}.
 \end{equation}
 On the other hand for $y \in \mathbb{R}^N\setminus B_{R}$, we have
 \begin{equation}\label{4.41}
 	(u(x)-u(y))_{+}^{p-1} \leq 2^{p-1}\left[u^{p-1}(x)+(u(y))_{-}^{p-1}\right],~\forall\, x\in B_{2r}.
 \end{equation}
 Then using the inequalities \eqref{4.40}, \eqref{4.41} and $supp(\eta)\subset B_{3r/2}$, we obtain
 	\begin{align}
 		I_{4} \leq & 2 \int_{(0,1)} c_{N, p, s} \int_{B_{R} \setminus B_{2 r}} \int_{B_{2 r}} (u(x)-u(y))_{+}^{p-1}(d+u(x))^{1-p} \frac{ \eta^{p}(x)}{|x-y|^{N+ps}}dxdy d\mu^+(s)\nonumber \\
 		&+2\int_{(0,1)} c_{N, p, s} \int_{{\mathbb{R}^N} \setminus B_{R}} \int_{B_{2 r}} (u(x)-u(y))_{+}^{p-1}(d+u(x))^{1-p} \frac{ \eta^{p}(x)}{|x-y|^{N+ps}}dxdy d\mu^+(s)\nonumber \\
        \leq&2 \int_{(0,1)} c_{N, p, s} \int_{B_{3r/2}}\int_{B_{R} \setminus B_{2 r}}  (u(x)-u(y))_{+}^{p-1}(d+u(x))^{1-p} \frac{ \eta^{p}(x)}{|x-y|^{N+ps}}dydx d\mu^+(s)\nonumber \\
 		&+2\int_{(0,1)} c_{N, p, s} \int_{B_{3r/2}} \int_{{\mathbb{R}^N} \setminus B_{R}}  (u(x)-u(y))_{+}^{p-1}(d+u(x))^{1-p} \frac{ \eta^{p}(x)}{|x-y|^{N+ps}}dydx d\mu^+(s)\nonumber \\
 		\leq & C\int_{(0,1)} c_{N, p, s} \int_{B_{3r/2}}\int_{{\mathbb{R}^N}\setminus B_{2 r}}  \frac{\eta^{p}(x)}{|x-y|^{N+ps}}dydx d\mu^+(s)\nonumber\\
        &+C'd^{1-p} \int_{(0,1)} c_{N, p, s} \int_{B_{3r/2}}\int_{{\mathbb{R}^N} \setminus B_{R}}  \frac{(u(y))_{-}^{p-1}}{|x-y|^{N+ps}} dydx d\mu^+(s)\label{eq new}
\end{align}
Now, to estimate the first integral of \eqref{eq new}, we use $$|x-y|>2r-3r/2=r/2~\text{ for all }x\in B_{3r/2} \text { and } y\in \mathbb{R}^N \setminus B_{2r}.$$
In addition, for the second integral in \eqref{eq new}, we use $2r<R$ and
$$\frac{|y-x_0|}{|y-x|}\leq 1 +\frac{|x-x_0|}{|y-x|}\leq 1 +\frac{3r/2}{r/2}=4~\text{ for all }x\in B_{3r/2} \text { and } y\in \mathbb{R}^N \setminus B_{2r}.$$
Therefore, we have
        \begin{align}\label{k-3.13}
 		I_{4} \leq & C\int_{(0,1)} c_{N, p, s} \int_{B_{3r/2}}\int_{{\mathbb{R}^N}\setminus B_{2 r}}  \frac{\eta^{p}(x)}{|x-y|^{N+ps}}dydx d\mu^+(s)\nonumber\\
        &+C'd^{1-p} \int_{(0,1)} c_{N, p, s} \int_{B_{3r/2}}\int_{{\mathbb{R}^N} \setminus B_{2r}}  \frac{(u(y))_{-}^{p-1}}{|x-y|^{N+ps}} dydx d\mu^+(s)\nonumber\\
        \leq &C \sup _{x \in B_{3r/2}} r^{N} \int_{(0,1)} c_{N, p, s} 2^{N+ps}\int_{{\mathbb{R}^N}\setminus B_{2 r}} \frac{dy}{|2(x-y)|^{N+ps}} d\mu^+(s)\nonumber\\
        &+C'd^{1-p}\left|B_{r}\right| \int_{(0,1)} c_{N, p, s} 4^{N+ps}\int_{{\mathbb{R}^N} \setminus B_{R}} \frac{(u(y))_{-}^{p-1}}{\left|x_{0}-y\right|^{N+ps}}dy d\mu^+(s)\nonumber\\
 		\leq & C \sup_{s \in \Sigma}r^{N-ps} + C'd^{1-p} r^{N} \sup_{s \in \Sigma} r^{-sp}\sup_{s \in \Sigma} \bigg(\frac{r}{R}\bigg)^{sp}\left[\operatorname{Tail}\left(u_{-} ; x_{0}, R\right)\right]^{p-1}
 	\end{align}
 	for some constants $C:=C(N,p,\Sigma, \mu^+), C':=C'(N,p,\Sigma,\mu^+)$. Therefore, by using \eqref{I_1est}, \eqref{I_2est}, \eqref{k-3.12} and \eqref{k-3.13} in \eqref{k-3.5}, we get
 	\begin{align}\label{k-3.16}
 	  \mu^+(\{1\})&\int_{B_{r}} |\nabla \log (u+d)|^p  dx\nonumber \\&+ \int_{(0,1)} c_{N, p, s} \int_{B_{2 r}} \int_{B_{2 r}} \left|\log \left(\frac{u(x)+d}{u(y)+d}\right)\right|^{p} \frac{\eta^{p}(y)}{|x-y|^{N+ps}} dxdy\,d\mu^+(s)\nonumber \\
 	&\leq   \int_{(0,1)} c_{N, p, s}  \int_{B_{2 r}} \int_{B_{2 r}} \frac{|\eta(x)-\eta(y)|^{p}}{|x-y|^{N+ps}} dxdy \,d\mu^+(s) \nonumber \\
    &+C' d^{1-p} r^{N} \sup_{s \in \Sigma} r^{-sp}\sup_{s \in \Sigma} \bigg(\frac{r}{R}\bigg)^{sp}   \left[\operatorname{Tail}\left(u_{-} ; x_{0}, R\right)\right]^{p-1}\nonumber \\
 	&+C r^{N}\sup_{s \in \Sigma} r^{-sp} +c' 2^{p-1}   \mu^+(\{1\}) r^{N-p}+ C \mu^+(\{0\}) r^{N}.
 \end{align}
 Again, by using {$|\nabla\eta|\leq Cr^{-1}$}, we have
 \begin{align}\label{k-3.17}
 	\int_{(0,1)} c_{N,p,s}  &\int_{B_{2 r}} \int_{B_{2 r}} \frac{|\eta(x)-\eta(y)|^{p}}{|x-y|^{N+ps}} dxdy \,d\mu^+(s) \nonumber \\
    & \leq Cr^{-p} \int_{(0,1)}c_{N,p,s}   \int_{B_{2 r}} \int_{B_{2 r}}|x-y|^{-N+p(1-s)}dxdyd\mu^+  \nonumber \\
    & \leq C\left|B_{2 r}\right| \int_{(0,1)}  \frac{c_{N,p,s}}{p(1-s)}  r^{-s p}\,d\mu^+(s) \leq C(N,p,\Sigma,\mu^+) r^{N} \sup_{s \in \Sigma} r^{-sp}. 
 \end{align}
 Therefore, the logarithmic estimate \eqref{k-log-ineq} follows from \eqref{k-3.16} and \eqref{k-3.17}.
 \end{proof}

We now have all the ingredients to state the following strong minimum principle.
\begin{thm}[Strong minimum principle]\label{min p} 
Let $\Omega \subset \mathbb{R}^{N}$ be a bounded domain with Lipschitz boundary. 
Let $\mu=\mu^{+}$ satisfy \eqref{measure1} and $s_{\sharp}$ be as in \eqref{measure4}.  Assume that $u \in X_{p}(\Omega)$ is a weak supersolution of \eqref{p-kuusi} such that $u \not\equiv 0$ in $\Omega.$ Then $u>0$ a.e. in $\Omega$.	
\end{thm}
\begin{proof} If $\mu^+((0, 1))=0$ then the problem turns out to be the classical strong maximum principles for the $p$-Laplace operator, therefore, it is reasonable to assume that $\mu^{+}((0, 1))>0.$
 Suppose for a moment that $u>0$ a.e. in $K$ for every connected and compact subset of $\Omega$. Since $\Omega$ is connected and $u\not\equiv0$ in $\Omega$, there exists a sequence of compact and connected sets $K_j\subset\Omega$ such that	
	\begin{equation*}
		\left|\Omega \backslash K_{j}\right|<\frac{1}{j}~\text { and }~ u \not \equiv 0 ~\text { in }~ K_{j}.
	\end{equation*}
Thus $u>0$ a.e. in $K_j$ for all $j$. Now passing to the limit as $j\rightarrow\infty$, we get that $u>0$ a.e. $\Omega$. Thus, it is enough to prove the result stated in the lemma for compact and connected subsets of $\Omega$. Since $K \subset \Omega$ is compact and connected, then there exists $r>0$ such that $K \subset\{x \in \Omega: \operatorname{dist}(x, \partial \Omega)>2 r\}$. Again, using the compactness, there exist $x_i\in K$, $i=1,2,...,k,$ such that the balls $B_{r / 2}\left(x_{1}\right), \ldots B_{r / 2}\left(x_{k}\right)$ cover $K$ and	
	\begin{equation}\label{bf-A4}
		\left|B_{r / 2}\left(x_{i}\right) \cap B_{r / 2}\left(x_{i+1}\right)\right|>0, \quad i=1, \ldots, k-1.
	\end{equation}
	Suppose that $u$ vanishes on a subset of $K$ with positive measure. Then with the help of \eqref{bf-A4}, we conclude that there exists $i \in\{1, \ldots, k-1\}$ such that
	\begin{equation*}
		|Z|:=|\left\{x \in B_{r / 2}\left(x_{i}\right): u(x)=0\right\}|>0.
	\end{equation*}
 For $d>0$ and $x \in B_{r / 2}\left(x_{i}\right)$, define	
	\begin{equation*}
		F_d(x)=\log \left(1+\frac{u(x)}{d}\right).
	\end{equation*}
	Observe that for every $x\in Z$ we have 	
	\begin{equation*}
		F_d(x)=0.
	\end{equation*}
	Thus for every $x \in B_{r/2}\left(x_{i}\right)$ and $y\in Z$ with $x\neq y$ we get	
	\begin{equation*}
		\left|F_d(x)\right|^{p}=\frac{\left|F_d(x)-F_d(y)\right|^{p}}{|x-y|^{N+ps}}|x-y|^{N+ps} .
	\end{equation*}
	Integrating with respect to $y \in Z$, we get	
		\begin{equation*}
		|Z|\left|F_d(x)\right|^{p} \leq\left(\max _{x, y \in B_{r / 2}\left(x_{i}\right)}|x-y|^{N+ps}\right) \int_{B_{r / 2}\left(x_{i}\right)} \frac{\left|F_d(x)-F_d(y)\right|^{p}}{|x-y|^{N+ps}} d y.
	\end{equation*}
	 Again integrating with respect to $x \in B_{r / 2}\left(x_{i}\right)$ we deduce the following inequality:	
		\begin{equation}\label{bf-A5}
		\int_{B_{r/2}\left(x_{i}\right)}\left|F_{d}\right|^{p}dx\leq \frac{r^{N+ps}}{|Z|} \int_{B_{r/2}\left(x_{i}\right)} \int_{B_{r/2}\left(x_{i}\right)} \frac{\left|F_{d}(x)-F_{d}(y)\right|^{p}}{|x-y|^{N+ps}} dx dy.
	\end{equation}
    Again, integrating both sides with respect to $s \in (0, 1)$ with the measure $\mu^+,$ we obtain
   \begin{align} \label{bf-A5'}
       \int_{(0, 1)}c_{N, p, s}	&\int_{B_{r/2}\left(x_{i}\right)}\left|F_{d}\right|^{p}dx d\mu^+(s)\nonumber \\& \leq \int_{(0, 1)} c_{N, p, s}	 \frac{r^{N+ps}}{|Z|} \int_{B_{r/2}\left(x_{i}\right)} \int_{B_{r/2}\left(x_{i}\right)} \frac{\left|F_{d}(x)-F_{d}(y)\right|^{p}}{|x-y|^{N+ps}} dx dy\,\, d\mu^+(s) \nonumber  \\& \leq   \int_{(0, 1)} c_{N, p, s}  \frac{r^{N+sp}}{|Z|} \int_{B_{r/2}\left(x_{i}\right)} \int_{B_{r/2}\left(x_{i}\right)} \frac{\left|F_{d}(x)-F_{d}(y)\right|^{p}}{|x-y|^{N+ps}} dx dy\,\, d\mu^+(s) .
   \end{align}
Observe that
\begin{equation*}
	\left|\log \left(\frac{d+u(x)}{d+u(y)}\right)\right|^{p}=\left|F_d(x)-F_d(y)\right|^{p}.
\end{equation*}
Plugging the logarithmic estimate \eqref{k-log-ineq} into the above  inequality \eqref{bf-A5'} by using the fact that $u_{-}=0$ (hence $Tail(u_{-},x_i,R)=0$), we get 
	\begin{align} \label{bf-A6}
		& C'({N, p, s})\int_{B_{r/2}\left(x_{i}\right)}\left|\log \left(1+\frac{u(x)}{d}\right)\right|^{p} dx \nonumber \\& \quad\quad\quad\quad \leq C(N,p,s) \sup_{s \in \Sigma} \frac{r^{N+ps}}{|Z|} ( \sup_{s \in \Sigma} r^{N-ps}+ \mu^+(\{1\}) r^{N-p}+\mu^+(\{0\}) r^N),
	\end{align}
    where $C>0$ is independent of $d.$
	Now taking limit $d\rightarrow0$ in \eqref{bf-A6}, we obtain that $u=0$ a.e. in $B_{r/2}\left(x_{i}\right).$ Thanks to \eqref{bf-A4}, by repeating this arguments in the quasi-balls $B_{r / 2}\left(x_{i-1}\right)$ and $B_{r / 2}\left(x_{i+1}\right)$ and so on we obtain that $u\equiv0$ a.e. on $K$. This is a contradiction and hence $u>0$ a. e. in $K$. This completes the proof of the result.
\end{proof}

\section{Eigenvalue problem  for nonlinear superposition operators $\mathcal{L}_{\mu,p}^{+}$} \label{sec6}

In this section, we study the Dirichlet eigenvalue problem associated with the operator \( \mathcal{L}_{\mu,p}^{+} \) defined in~\eqref{positive part of L}. 
Precisely, for \( \lambda \in \mathbb{R} \), we consider the problem
\begin{equation}\label{Eig problem}
    \begin{cases}
        \mathcal{L}_{\mu,p}^{+} u = \lambda |u|^{p-2}u & \text{in } \Omega, \\[0.3em]
        u = 0 & \text{in } \mathbb{R}^{N} \setminus \Omega.
    \end{cases}
\end{equation}
Note that all the results established in Section~\ref{sec3} for the operator \( \mathcal{L}_{\mu,p} \) remain valid for the operator \( \mathcal{L}_{\mu,p}^{+} \). 
The main advantage of \( \mathcal{L}_{\mu,p}^{+} \) lies in the fact that it satisfies the strong maximum principle, which allows for a deeper analysis of its spectral properties. 
Accordingly, in this section, we investigate several spectral properties of this operator that fundamentally rely on the strong maximum principle.

The weak formulation of the eigenvalue problem \eqref{Eig problem} is given by
\begin{align*}
& \int_{[0,1]} c_{N, s, p} \iint_{\mathbb{R}^{2 N}} \frac{|u(x)-u(y)|^{p-2}(u(x)-u(y))(v(x)-v(y))}{|x-y|^{N+sp}} d x d y d \mu^{+}(s)\\
& \quad \quad \quad \quad \quad \quad \quad \quad \quad \quad \quad \quad \quad \quad \quad \quad \quad =\lambda \int_{\Omega} |u(x)|^{p-2}u(x) v(x) d x 
\end{align*}
for any $v \in X_{p}(\Omega)$.

Recall that if there exists a nontrivial function $u \in X_{p}(\Omega)$ solving \eqref{Eig problem}, then $\lambda \in \mathbb{R}$ is referred to as an eigenvalue of the operator $\mathcal{L}^{+}_{\mu,p}$. Any such function $u \in X_{p}(\Omega)$ is called an eigenfunction corresponding to the eigenvalue $\lambda$.

We define the energy functional $\mathfrak{L}_{p}:X_{p}(\Omega) \rightarrow \mathbb{R}$ as
\begin{align}\label{energy functional}
\begin{split}
    \mathfrak{L}_{p}(u) & :=\frac{1}{p} \int_{[0,1]}[u]_{s,p}^{p} d \mu^{+}(s)=\frac{1}{p}\|u\|_{X_p(\Omega)}^{p}. 
\end{split}
\end{align}

Note that  a direct computation shows that $\mathfrak{L}_{p} \in C^1\left(X_{p}(\Omega), \mathbb{R}\right)$ with
\begin{align}\label{derivative of L}
\begin{split}
    \left\langle \mathfrak{L}_{p}^{\prime}(u), v\right\rangle&:=\left\langle \mathfrak{L}_{p}^{\prime}(u), v\right\rangle_{+}\\
    &=\int_{[0,1]} c_{N, s, p} \iint_{\mathbb{R}^{2 N}} \frac{|u(x)-u(y)|^{p-2}(u(x)-u(y))(v(x)-v(y))}{|x-y|^{N+sp}} d x d y d \mu^{+}(s)
\end{split}
\end{align}
for any $v \in X_{p}(\Omega)$. 

\begin{definition}\label{weak solution def +}
      A function $u \in X_{p}(\Omega)$ is a solution of \eqref{Eig problem} if $u$ satisfies the equation
\begin{equation}\label{weak solution of eig problem}
    \left\langle \mathfrak{L}_{p}^{\prime}(u), v\right\rangle=\lambda \int_{\Omega} |u(x)|^{p-2}u(x) v(x) d x \text { for all } v \in X_{p}(\Omega),
\end{equation}
where $\left\langle \mathfrak{L}_{p}^{\prime}(u), v\right\rangle$ is defined in \eqref{derivative of L}.
\end{definition}

Note that Theorem \ref{principal eigenvalue P} remains true for the eigenvalue problem \eqref{Eig problem}. Taking advantage of the fact $\mu^-\equiv0$, we prove the following additional properties of the eigenvalues and eigenfunctions of \eqref{Eig problem}. We begin with the following result.

\begin{thm}\label{strong maximum for eigenvalue}
    Let $\Omega \subset \mathbb{R}^{N}$ be a bounded domain with Lipschitz boundary, and let $\mu=\mu^+$ satisfy \eqref{measure1}. Let $s_{\sharp}$ be as in \eqref{measure4}, and $1 < p < N/s_{\sharp}$. Let $u \geq 0$ in $\Omega$ and $u = 0$ in $\mathbb{R}^{N} \backslash \Omega$ be an eigenfunction of \eqref{Eig problem} associated with an eigenvalue $\lambda > 0$. Then $u > 0$ in $\Omega$.
\end{thm}
\begin{proof}
    Since $u \geq 0$ is an eigenfunction corresponding to $\lambda > 0$, we have
    \begin{equation*}
        \left\langle \mathfrak{L}_{p}^{\prime}(u), v\right\rangle=\lambda\int_{\Omega}|u|^{p-2}uvdx\geq 0,
    \end{equation*}
for any $v\in X_{p}(\Omega)$ with $v\geq 0$. This shows that the eigenfunction $u \geq 0$ is a weak
supersolution to the problem \eqref{p-kuusi}. The conclusion then follows from Theorem \ref{min p}.
\end{proof}

\begin{thm}\label{strict positive +}
    Let $\Omega$ be a bounded domain of $\mathbb{R}^{N}$ with Lipschitz boundary. Let $\mu^+$ satisfy \eqref{measure1}. Let $s_{\sharp}$ be as in \eqref{measure4} and $1<p<N/s_{\sharp}$. Let $u \in X_{p}(\Omega)$ be an eigenfunction of \eqref{Eig problem} associated with the eigenvalue $\lambda_{1, \mu^+}$.
Then either $u > 0$ or $u < 0$ in $\Omega$.
\end{thm}
\begin{proof}
    By Theorem \ref{principal eigenvalue P}-(iv) for $\mathcal{L}_{\mu, p}=\mathcal{L}^+_{\mu, p}$, we have that either $u \geq 0$ or $u \leq 0$ in $\Omega$.  If $u \geq 0$ in $\Omega$, by Theorem \ref{strong maximum for eigenvalue}, we have $u > 0$ in $\Omega$. Similarly, the case $u < 0$ can be proved by replacing $u$ with $(-u)$.
\end{proof}

\begin{thm}\label{global boundedness result} Let $\Omega$ be a bounded domain of $\mathbb{R}^{N}$ with Lipschitz boundary. Let $\mu^+$ satisfy \eqref{measure1} with $s_{\sharp}$ be as in \eqref{measure4} and $1 < p < N/s_{\sharp}$. Then, all eigenfunctions $u \in X_p(\Omega)$ to \eqref{Eig problem} for positive eigenvalues are bounded, that is, $u \in L^\infty(\mathbb{R}^N).$
\end{thm}

\begin{proof}
 To establish the theorem, it is sufficient to obtain an upper bound for the positive part $u_{+}$ of $u \in X_p(\Omega)$. Indeed, since $-u$ is also an eigenfunction associated with $\lambda$, an analogous argument yields the corresponding estimate for the negative part. Consequently, it suffices to show that  
\begin{equation}\label{boud of eigenfunction 1}
    \|u_{+}\|_{L^{\infty}} \leq 1 \quad \text{whenever} \quad \|u_{+}\|_{L^p} \leq \delta,
\end{equation}  
for some $\delta>0$ to be determined. This restriction is not essential, as the general case, that is boundedness for by any constant $C>0$ instead of $1,$ can be recovered by a scaling argument, owing to the homogeneity of equation \eqref{Eig problem}.

Now, for any integer $m \geq 1$, we define the function $w_m$ by  
\begin{equation*}
    w_m := (u - (1 - 2^{-m}))_{+}.
\end{equation*}
By construction, we have $w_m \in X_{p}(\Omega)$. Moreover, the following inequalities hold:
\begin{align}\label{boud of eigenfunction 2}
\begin{split}
    w_{m+1}(x) &\leq w_m(x) \quad \text{a.e. in } \mathbb{R}^N, \\
    u(x) &< (2^{m+1} - 1)\, w_m(x) \quad \text{for } x \in \{w_{m+1} > 0\},
\end{split}
\end{align}
together with the inclusions
\begin{equation*}
    \{w_{m+1} > 0\} \subseteq \{w_m > 2^{-(m+1)}\},
\end{equation*}
which are valid for all $m \in \mathbb{N}$.

We recall the  elementary fact for $v \in X_{p}(\Omega)$:
\begin{equation}\label{tex inequality 1}
    |v(x)-v(y)|^{p-2}\left(v_{+}(x)-v_{+}(y)\right)(v(x)-v(y)) \geq\left|v_{+}(x)-v_{+}(y)\right|^p
\end{equation}
for all $x, y \in \mathbb{R}^N$.

We will prove \eqref{boud of eigenfunction 1} by  estimating the decay of the quantity $a_m:=\left\|w_m\right\|_{L^p}^p$. On the one hand, in view of \eqref{tex inequality 1} with $v=u-\left(1-2^{-(m+1)}\right),$
\begin{align}\label{boud of eigenfunction 2.0}
\begin{split}
   & \left\|w_{m+1}\right\|_{X_p(\Omega)}^p= \int_{[0,1]} c_{N,s,p} \iint_{\mathbb{R}^{2 N}} \frac{|w_{m+1}(x)-w_{m+1}(y)|^p}{|x-y|^{N+ sp}} d x d y d \mu^{+}(s) \\
& \leq \int_{[0,1]} c_{N,s,p} \iint_{\mathbb{R}^{2 N}} \frac{|u(x)-u(y)|^{p-2}(u(x)-u(y))(w_{m+1}(x)-w_{m+1}(y))}{|x-y|^{N+sp}} d x d y d \mu^{+}(s).
\end{split}
\end{align}
Now, by plugging $w_{m+1}$ as a test function in \eqref{weak solution of eig problem}, we get the following
\begin{align}\label{boud of eigenfunction 3.0}
\begin{split}
& \lambda \int_{\Omega} |u(x)|^{p-2}u(x) w_{m+1}(x) d x\\
    & =\int_{[0,1]} c_{N, s, p} \iint_{\mathbb{R}^{2 N}} \frac{|u(x)-u(y)|^{p-2}(u(x)-u(y))(w_{m+1}(x)-w_{m+1}(y))}{|x-y|^{N+sp}} d x d y d \mu^{+}(s).
\end{split}
\end{align}

Thus, applying \eqref{boud of eigenfunction 3.0} and using \eqref{boud of eigenfunction 2}, one obtains from \eqref{boud of eigenfunction 2.0} that
\begin{align}\label{boud of eigenfunction 3}
    \left\|w_{m+1}\right\|_{X_p(\Omega)}^{p}&\leq \int_{[0,1]} c_{N, s, p} \iint_{\mathbb{R}^{2 N}} \frac{|u(x)-u(y)|^{p-2}(u(x)-u(y))}{|x-y|^{N+sp}} \times \nonumber\\
    &\times (w_{m+1}(x)-w_{m+1}(y))d x d y d \mu^{+}(s) \nonumber\\
& \leq\lambda\int_{\Omega} |u(x)|^{p-2}u(x) w_{m+1}(x) d x= \lambda \int_{\left\{w_{m+1}>0\right\}} |u(x)|^{p-2}u(x) w_{m+1}(x) d x \nonumber\\
& \leq\lambda \int_{\left\{w_{m}>2^{-(m+1)}\right\}} (2^{m+1}-1)^{p-1} w^p_{m}(x) \mathrm{d} x\nonumber \leq \lambda (2^{m+1}-1)^{p-1} a_{m}.\\
\end{align}

The left-hand side of the above inequality can, in turn, be bounded from below by $a_{m+1}$ through the application of the fractional Sobolev embedding. For this analysis, we apply the H\"older's inequality (with exponents $p_{\sharp}^{*} / p$ and $N /(p s_{\sharp})$ ) and continuous fractional Sobolev embedding given in  Theorem \ref{embedding lem}, to obtain
\begin{align}\label{boud of eigenfunction 4}
\begin{split}
    a_{m+1}=\|w_{m+1}\|_{L^p}^p&=\int_{\{w_{m+1}>0\}}|w_{m+1}|^pdx  \leq \left\|w_{m+1}\right\|_{L^{p_{\sharp}^*}}^p\left|\left\{w_{m+1}>0\right\}\right|^{\frac{ps_\sharp}{N}}\\ & \leq C \left\|w_{m+1}\right\|_{X_p(\Omega)}^p\left|\left\{w_{m+1}>0\right\}\right|^{\frac{ps_\sharp}{N}},
\end{split}
\end{align}
with a constant $C:=C(N, s, p)>0.$  

On the other hand, by \eqref{boud of eigenfunction 2} and Chebychev's inequality, one has
\begin{equation}\label{boud of eigenfunction 5}
    \left|\left\{w_{m+1}>0\right\}\right| \leq\left|\left\{w_m>2^{-(m+1)}\right\}\right| \leq 2^{p(m+1)} \|w_{m}\|_{L^p}^p = 2^{p(m+1)} a_m.
\end{equation}
Thus, combining \eqref{boud of eigenfunction 3}, \eqref{boud of eigenfunction 4} and \eqref{boud of eigenfunction 5} we obtain
\begin{align*}
    a_{m+1} & \leq C \left( 2^{p(m+1)} a_m\right)^{\frac{ps_\sharp}{N}} \left(\lambda(2^{m+1}-1)^{p-1} a_{m}\right) \\
    &\leq C \left( 2^{p(m+1)} a_m\right)^{\frac{ps_\sharp}{N}} \left(\lambda 2^{p(m+1)} a_{m}\right) \leq C \lambda \left( 2^{p(m+1)} a_m\right)^{1+\frac{ps_\sharp}{N}}.
\end{align*}

 Thus, we conclude that, for $ps_{\sharp}< N$, an estimate of the form
\begin{equation*}
    a_{m+1} \leq \tilde{C}^m a_m^{1+\alpha}, \quad \text { for all } \quad m \in \mathbb{N},
\end{equation*}
holds for some $\alpha>0$ and a suitable constant $\tilde{C}:= \max \{1, C \lambda \left( 2^{p(m+1)} \right)^{1+\frac{ps_\sharp}{N}} \}.$ By definition, we have $a_0:=\left\|u_{+}\right\|_{L^p}^{p}.$ Therefore, by choosing $\delta$ such that $ a_0  \leq \tilde{C}^{-\frac{1}{\alpha^2}}=: \delta^p,$  we conclude that 
\begin{equation}\label{boud of eigenfunction 6}
    \lim _{m \rightarrow \infty} a_m=0.
\end{equation}
 Since $a_m \rightarrow \|(u-1)_{+}\|_{L^p(\Omega)}^p$ as $m \rightarrow \infty,$ from \eqref{boud of eigenfunction 6} we infer that $(u-1)_+=0.$ This implies that $\|u_+\|_{L^\infty(\Omega)} \leq 1,$ which combined with the fact that $u =$ in $\mathbb{R}^N \backslash \Omega$ shows that $u_+ \in L^\infty(\mathbb{R}^N).$ This completes the proof of this result.
\end{proof}  

\begin{remark}
  Recall part $(ii)$ and $(iii)$ of Theorem \ref{embedding lem}. Observe that Theorem \ref{global boundedness result} remains true even for the range $ps_{\sharp}\geq N$. Indeed, for $ps_{\sharp}>N$, we have $u\in C^{0,\alpha}(\bar{\Omega})$ for $0<\alpha<s_{\sharp}-N/p$ and hence $u\in L^{\infty}(\Omega)$. On the other hand, when $ps_{\sharp}=N$, then $u\in L^q(\Omega)$ for all $q\in(1,\infty)$. Therefore, repeating the arguments of Theorem \ref{global boundedness result}, we conclude that $u\in L^{\infty}(\Omega)$.
\end{remark}

\begin{thm}\label{result sign changing}
      Let $\Omega$ be a bounded domain of $\mathbb{R}^{N}$ with Lipschitz boundary. Let $\mu^+$ satisfy \eqref{measure1}. Let $s_{\sharp}$ be as in \eqref{measure4} and $1<p<N/s_{\sharp}$. If $v$ is an eigenfunction of \eqref{Eig problem} associated to an eigenvalue $\lambda>\lambda_{1, \mu^+}(\Omega)$, then $v$ must be sign-changing.
\end{thm}
\begin{proof}
Suppose $v$ does not change sign in $\Omega$. Thus, without loss of generality, we assume that $v \geq 0.$ By Theorem \ref{strong maximum for eigenvalue}, we conclude that $v>0$ in $\Omega$. Let $u$ be the eigenfunction corresponding to $\lambda_{1}.$  Then, Theorem \ref{strict positive +} yields that $u>0$ in $\Omega$. Without loss of generality we can assume that $\|v\|_{L^{p}(\Omega)}=\|u\|_{L^{p}(\Omega)}=1$. Consider the function $\sigma_t$ defined by
\begin{equation*}
    \sigma_t=\left(t u^p+(1-t) v^p\right)^{\frac{1}{p}}, \quad \forall t \in (0,1).
\end{equation*} 
Obviously, 
\begin{equation*}
   \int_{\Omega} \left|\sigma_t\right|^p d x=t \int_{\Omega} |u|^p d x+(1-t) \int_{\Omega} |v|^p d x=1. 
\end{equation*}
Applying the convexity of the map $t \mapsto t^p$ for $p>1$, we get
\begin{align*}
\left|\nabla \sigma_t\right|^p & =\left|\left(t u^p+(1-t) v^p\right)^{\frac{1}{p}-1}\left(t u^{p-1} \nabla u+(1-t) v^{p-1} \nabla v\right)\right|^p \\
& =\sigma_t^p\left|t \frac{u^p \nabla u}{\sigma_t^p u}+(1-t) \frac{v^p \nabla v}{\sigma_t^p v}\right|^p \\
& =\sigma_t^p\left| w \frac{\nabla u}{u}+(1-w) \frac{\nabla v}{v}\right|^p  \\
& \leq \sigma_t^p\left(w\left|\frac{\nabla u}{u}\right|^p+(1-w)\left|\frac{\nabla v}{v}\right|^p\right)=t|\nabla u|^p+(1-t)|\nabla v|^p ,
\end{align*}
where $w=\frac{t u^p}{t u^p+(1-t) v^p}.$
Moreover, by \cite[Lemma 4.1]{FP14} we have 
\begin{align*}
    [\sigma_t]_{s,p}^p&= \iint_{\mathbb{R}^{2N}} \frac{\left|\sigma_t(x)-\sigma_t(y)\right|^p}{|x-y|^{N+sp}} d x d y \\
& \leq t \iint_{\mathbb{R}^{2N}} \frac{|u(x)-u(y)|^p}{|x-y|^{N+sp}} d x d y+(1-t) \iint_{\mathbb{R}^{2N}} \frac{|v(x)-v(y)|^p}{|x-y|^{N+sp}} d x d y\\
&=t[u]_{s,p}^p+(1-t)[v]_{s,p}^p.
\end{align*}
Combining the above facts, we obtain
\begin{align*}
    \left\langle \mathfrak{L}_{p}^{\prime}(\sigma_t), \sigma_t\right\rangle&\leq t\left\langle \mathfrak{L}_{p}^{\prime}(u), u\right\rangle+(1-t)\left\langle \mathfrak{L}_{p}^{\prime}(v), v\right\rangle.
\end{align*}
Consequently, owing to the fact that $u > 0$ is an eigenfunction associated with principal eigenvalue $\lambda_{1, \mu^+}$ and $v > 0$ is an eigenfunction associated with an eigenvalue $\lambda>\lambda_{1, \mu^+}$, we arrive at
\begin{align}\label{sign change est 1}
\begin{split}
    \left\langle \mathfrak{L}_{p}^{\prime}(\sigma_t), \sigma_t\right\rangle-\left\langle \mathfrak{L}_{p}^{\prime}(v), v\right\rangle&\leq t\left\langle \mathfrak{L}_{p}^{\prime}(u), u\right\rangle-t \left\langle \mathfrak{L}_{p}^{\prime}(v), v\right\rangle=t(\lambda_{1, \mu^+}-\lambda)<0.
\end{split}
\end{align}
On the other hand, by the convexity of the map $t \mapsto t^p$, we have
\begin{align*}
& \int_{\Omega} \left|\sigma_t\right|^p d x-\int_{\Omega} |v|^p d x \geq p \int_{\Omega} |v|^{p-2} v\left(\sigma_t-v\right) d x, \\
& \int_{\Omega}\left|\nabla \sigma_t\right|^p d x-\int_{\Omega}|\nabla v|^p d x \geq p \int_{\Omega}|\nabla v|^{p-2} \nabla v \cdot \nabla\left(\sigma_t-v\right) d x,
\end{align*}
and 
\begin{align*}
    &\int_{(0,1)}\iint_{\mathbb{R}^{2N}} \frac{\left|\sigma_t(x)-\sigma_t(y)\right|^p}{|x-y|^{N+ps}} d x d yd\mu^+(s)-\int_{(0,1)}\iint_{\mathbb{R}^{2N}}  \frac{|v(x)-v(y)|^p}{|x-y|^{N+p s}} d x d yd\mu^+(s)\\
    &\quad \quad \geq p \int_{(0,1)}\iint_{\mathbb{R}^{2N}} \frac{|v(x)-v(y)|^{p-2}}{|x-y|^{N+p s}}(v(x)-v(y))\left(\left(\sigma_t-v\right)(x)-\left(\sigma_t-v\right)(y)\right) d x d y d\mu^+(s).
\end{align*}
From the inequalities above, it follows that
\begin{equation}\label{sign change est 2}
    \left\langle \mathfrak{L}_{p}^{\prime}(\sigma_t), \sigma_t\right\rangle-\left\langle \mathfrak{L}_{p}^{\prime}(v), v\right\rangle \geq p\left\langle \mathfrak{L}_{p}^{\prime}(v), \sigma_t-v\right\rangle.
\end{equation}
Then, combining inequalities \eqref{sign change est 1} and \eqref{sign change est 2}, and using Definition \ref{weak solution def +}, we obtain
\begin{align*}
\begin{split}
 & p \lambda \int_{\Omega} |v|^{p-2} v\left(\sigma_t-v\right) d x=p\left\langle \mathfrak{L}_{p}^{\prime}(v), \sigma_t-v\right\rangle\leq t\left(\lambda_{1, \mu^+}-\lambda\right)<0.
\end{split}
\end{align*}
This implies that
\begin{align}\label{sign change est 3}
    \frac{p \lambda}{t} \int_{\Omega} |v|^{p-2} v\left(\sigma_t-v\right) d x \leq \lambda_{1, \mu^+}-\lambda<0, \quad \forall t\in (0,1).
\end{align}
From this, since $v>0$, we obtain that $\sigma_t -v \leq 0$ a.e. in $\Omega$. Again, by the convexity of the map $t \mapsto t^p$, it follows that
\begin{equation*}
    v-\sigma_t=v-\left(t u^p+(1-t) v^p\right)^{\frac{1}{p}} \leq v-(t u+(1-t) v)=t(v-u). 
\end{equation*}
Thus, for all $t \in(0,1),$ we have $\left| v^{p-1}\left(\frac{\sigma_t-v}{t}\right)\right| \leq  v^{p-1}(v-u),$ which is an integrable function. Moreover, we have
\begin{align*}
\lim _{t \rightarrow 0}  v^{p-1}\left(\frac{\sigma_t-v}{t}\right) & = v^{p-1} \lim _{t \rightarrow 0}\left(\frac{\sigma_t-\sigma_0}{t}\right) \\
& =\frac{1}{p}\left[ v^{p-1} v^{1-p}\left(u^p-v^p\right)\right]=\frac{1}{p}\left(u^p-v^p\right)  
\end{align*}
pointwise in $\Omega$. Consequently, the dominated convergence theorem yields $ v^{p-1}\left(\frac{\sigma_t-v}{t}\right)  \rightarrow \frac{1}{p}\left(u^p-v^p\right)$ in $L^1(\Omega)$. Thus, applying the limit as $t \rightarrow 0$ in \eqref{sign change est 3}, we get
\begin{equation*}
    p \lambda \int_{\Omega} \frac{1}{p}\left(u^p-v^p\right)  d x \leq \lambda_{1, \mu^+}-\lambda,
\end{equation*}
which leads to
\begin{equation*}
    0=\lambda\left(\int_{\Omega} |u|^p d x-\int_{\Omega} |v|^p d x\right) \leq \lambda_{1, \mu^+}-\lambda.
\end{equation*}
This contradicts our assumption that $\lambda>\lambda_{1, \mu^+}$. Hence, the proof is complete.
\end{proof}

\begin{lem}
 Let $\Omega$ be a bounded domain of $\mathbb{R}^{N}$ with Lipschitz boundary. Let $\mu^+$ satisfy \eqref{measure1}. Let $s_{\sharp}$ be as in \eqref{measure4} and $1<p<N/s_{\sharp}$. Let $v$ be an eigenfunction of problem \eqref{Eig problem} corresponding to $\lambda \neq \lambda_{1, \mu^+}(\Omega)$. Then there is a positive constant $C$ independent of $v$ such that  
\begin{equation}\label{need for sign changing}
    \lambda \geq C(N, s_{\sharp},p)\left|\Omega_{+}\right|^{-\frac{ps_{\sharp}}{N}} \text { and  } \lambda \geq C(N, s_{\sharp},p)\left|\Omega_{-}\right|^{-\frac{ps_{\sharp}}{N}} ,
\end{equation}
where $\Omega_{+}:=\{x \in \Omega:\, v>0\}$ and $\Omega_{-}:=\{x \in \Omega:\,v<0\}$.
\end{lem}

\begin{proof}
    Let $(\lambda, v)$ be an eigenpair for the problem \eqref{Eig problem} such that $\lambda>\lambda_{1, \mu^+}$ with $$\int_{\Omega} |v(x)|^p d x=1.$$ Then by Theorem \ref{result sign changing}, we have $v_{-} \neq 0$. Using $v_{-} \in X_{p}(\Omega)$ as a test function in the weak formulation \eqref{weak solution of eig problem}, we obtain
\begin{equation}\label{tech lem proof 1}
    \left\langle \mathfrak{L}_{p}^{\prime}(v), v_{-}\right\rangle =-\lambda \int_{\Omega} \left|v_{-}(x)\right|^p d x.
\end{equation}
Note that
\begin{align*}
 v(x)  v_{-}(x) & =-\left|v_{-}(x)\right|^2 \text { a.e. in } \Omega,\\
\nabla v(x) \cdot \nabla\left(v_{-}(x)\right) & =-\left(\nabla\left(v_{-}(x)\right)\right)^2 \text { a.e. in } \Omega, \text { and } \\
-(v(x)-v(y))\left(v_{-}(x)-v_{-}(y)\right) & =-\left(\left(v_{+}(x)-v_{+}(y)-\left(v_{-}(x)-v_{-}(y)\right)\left(v_{-}(x)-v_{-}(y)\right)\right.\right. \\
& \geq\left(v_{-}(x)-v_{-}(y)\right)^2 .
\end{align*}
By the above relations, we get
\begin{equation}\label{tech lem proof 2}
    -\left\langle \mathfrak{L}_{p}^{\prime}(v), v_{-}\right\rangle \geq \int_{[0,1]}c_{N,s,p} \iint_{\mathbb{R}^{2N}} \frac{\left|v_{-}(x)-v_{-}(y)\right|^p}{|x-y|^{N+sp}} d x d y d \mu^{+}(s).
\end{equation}
Moreover, applying the continuous embedding $X_{p}(\Omega) \hookrightarrow L^{p^{*}_{s_\sharp}}(\Omega)$ from Theorem \ref{embedding lem} and the Hölder's inequality with exponents $p_{s_\sharp}^{*}/p$ and $p_{s_\sharp}^{*}/(p_{s_\sharp}^{*}-p)$, it follows that
\begin{align}\label{tech lem proof 3}
    \int_{\Omega}\left|v_{-}(x)\right|^p d x=\int_{\Omega_{-}}\left|v_{-}(x)\right|^p d x &\leq \left|\Omega_{-}\right|^{(p_{s_\sharp}^{*}-p)/p_{s_\sharp}^{*}} \|v_{-}(x)\|_{L^{p^{*}_{s_\sharp}}(\Omega)}^{p}\nonumber\\
    &\leq \left|\Omega_{-}\right|^{(p_{s_\sharp}^{*}-p)/p_{s_\sharp}^{*}}(C_{p_{s_\sharp}^{*}})^{p} \|v_{-}(x)\|_{X_p(\Omega)}^{p},
\end{align}
where $C_{p_{s_\sharp}^{*}}$ denotes an embedding constant.

Combining \eqref{tech lem proof 1}, \eqref{tech lem proof 2} and \eqref{tech lem proof 3}, we obtain
\begin{equation*}
    \lambda \int_{\Omega} \left|v_{-}(x)\right|^p d x=-\left\langle \mathfrak{L}_{p}^{\prime}(v), v_{-}\right\rangle\geq \|v_{-}\|_{X_p(\Omega)}^{p}\geq \left|\Omega_{-}\right|^{(p-p_{s_\sharp}^{*})/p_{s_\sharp}^{*}}(C_{p_{s_\sharp}^{*}})^{-p} \int_{\Omega} \left|v_{-}(x)\right|^p d x.
\end{equation*}
Since $\left\|v_{-}\right\|_{L^{p}(\Omega)} \neq 0$, dividing both sides of the above inequality by $\left\|v_{-}\right\|_{L^{p}(\Omega)},$ we obtain
\begin{equation*}
    \left|\Omega_{-}\right| \geq\left(\frac{1}{C \lambda}\right)^{\frac{p_{s_\sharp}^{*}}{p_{s_\sharp}^{*}-p}},
\end{equation*}
where $C=(C_{p_{s_\sharp}^{*}})^{p}$.

Following the above arguments for $-v$ in place of the eigenfunction $v$, one can infer that
\begin{equation*}
    \left|\Omega_{+}\right| \geq\left(\frac{1}{C \lambda}\right)^{\frac{p_{s_\sharp}^{*}}{p_{s_\sharp}^{*}-p}},
\end{equation*}
where $C=(C_{p_{s_\sharp}^{*}})^{p}$. This completes the proof.
\end{proof}

\begin{thm}\label{simple new}
   Let $\Omega$ be a bounded domain of $\mathbb{R}^{N}$ with Lipschitz boundary. Let $\mu^+$ satisfy \eqref{measure1}. Let $s_{\sharp}$ be as in \eqref{measure4} and $1<p<N/s_{\sharp}$. Then, the first eigenvalue $\lambda_{1, \mu^+}(\Omega)$ of the problem \eqref{Eig problem} is simple.
\end{thm} 
\begin{proof}
Recall
\begin{equation*}
    \left\langle \mathfrak{L}_{p}^{\prime}(\sigma), \sigma\right\rangle:=  \int_{[0,1]} c_{N, s, p} \iint_{\mathbb{R}^{2 N}} \frac{|\sigma(x)-\sigma(y)|^{p}}{|x-y|^{N+sp}} d x d y d \mu^{+}(s)=\lambda_1>0.
\end{equation*} 
We aim to show that the principal eigenvalue $\lambda_{1}$ is simple. Let $u, v$ be two eigenfunctions associated with $\lambda_{1}$. Without loss of generality, assume that $u, v>0$ in $\Omega$ and
\begin{equation*}
    \int_{\Omega} |u|^{p} d x=\int_{\Omega} |v|^{p} d x=1.
\end{equation*}

Define, $\sigma:=\left(\frac{u^{p}+v^{p}}{2}\right)^{\frac{1}{p}}$. Then we have
\begin{equation*}
   \int_{\Omega} \left|\sigma\right|^p d x=\frac{1}{2} \int_{\Omega} |u|^p d x+\frac{1}{2} \int_{\Omega} |v|^p d x=1,
\end{equation*}
Applying the convexity of the map $t \mapsto t^p$ for $p>1$, we get
\begin{align}\label{extra pr del 1}
\left|\nabla \sigma\right|^p & =\left|\left(\frac{1}{2} u^p+\frac{1}{2} v^p\right)^{\frac{1}{p}-1}\left(\frac{1}{2} u^{p-1} \nabla u+\frac{1}{2} v^{p-1} \nabla v\right)\right|^p \nonumber\\
&=\frac{1}{2^p}\left(\frac{u^p+v^p}{2}\right)^{1-p}\left|u^{p-1} \nabla u+v^{p-1} \nabla v\right|^p\nonumber\\
& = \frac{\sigma^p}{2^p}\left|\frac{u^p \nabla u}{\sigma^p u}+ \frac{v^p \nabla v}{\sigma^p v}\right|^p \\
& =\sigma^p\left| w \frac{\nabla u}{u}+(1-w) \frac{\nabla v}{v}\right|^p \nonumber \\
& \leq \sigma^p\left(w\left|\frac{\nabla u}{u}\right|^p+(1-w)\left|\frac{\nabla v}{v}\right|^p\right)=\frac{1}{2}\left(|\nabla u|^p+|\nabla v|^p\right) ,\nonumber
\end{align}
where $w=\frac{ u^p}{ u^p+ v^p}.$ Hence, we have
\begin{equation}\label{extra pr del 2}
\int_{\Omega}\left|\nabla \sigma\right|^p d x \leq \frac{1}{2} \left(\int_{\Omega}|\nabla u|^p d x+ \int_{\Omega}|\nabla v|^p d x\right). 
\end{equation}
Moreover, by \cite[Lemma 4.1]{FP14} we have 
\begin{align}\label{extra pr del 3}
    [\sigma]_{s,p}^p&= \iint_{\mathbb{R}^{2N}} \frac{\left|\sigma(x)-\sigma(y)\right|^p}{|x-y|^{N+sp}} d x d y \\
& \leq \frac{1}{2}\iint_{\mathbb{R}^{2N}} \frac{|u(x)-u(y)|^p}{|x-y|^{N+sp}} d x d y+\frac{1}{2}\ \iint_{\mathbb{R}^{2N}} \frac{|v(x)-v(y)|^p}{|x-y|^{N+sp}} d x d y \nonumber\\
&=\frac{1}{2}\left([u]_{s,p}^p+[v]_{s,p}^p\right).\nonumber
\end{align}
Combining the above facts, we obtain
\begin{align}\label{extra pr del 4}
    \left\langle \mathfrak{L}_{p}^{\prime}(\sigma), \sigma\right\rangle&\leq \frac{1}{2}\left(\left\langle \mathfrak{L}_{p}^{\prime}(u), u\right\rangle+\left\langle \mathfrak{L}_{p}^{\prime}(v), v\right\rangle\right)=\frac{1}{2}\left(\lambda_1+\lambda_1\right)=\lambda_1.
\end{align}
Now, by the characterization of $\lambda_1$, we should have the equality in \eqref{extra pr del 4}, and it can be achieved when the equality holds in \eqref{extra pr del 3}, and in \eqref{extra pr del 2}. Moreover, the inequality \eqref{extra pr del 2} is obtained from \eqref{extra pr del 1}, and \eqref{extra pr del 1} is a consequence of the strict convexity property of the map $t \mapsto|t|^{p}$. Hence, in order to have equality in \eqref{extra pr del 1}, we must have $\frac{\nabla u}{u}=\frac{\nabla v}{v}$ a.e. in $\Omega$. This implies that $\nabla\left(\frac{u}{v}\right)=0$ a.e. in $\Omega$, and hence there exists a constant $c>0$ such that $u=c v$ a.e. in $\Omega$. This asserts that $\lambda_{1}$ is simple.
\end{proof}

\begin{thm}\label{properties of 1 eigenvalue}
   Let $\Omega$ be a bounded domain of $\mathbb{R}^{N}$ with Lipschitz boundary. Let $\mu^+$ satisfy \eqref{measure1}. Let $s_{\sharp}$ be as in \eqref{measure4} and $1<p<N/s_{\sharp}$. Then, the first eigenvalue $\lambda_{1, \mu^+}(\Omega)$ of the problem \eqref{Eig problem} is isolated.
\end{thm} 
\begin{proof}
     By definition, $\lambda_{1, \mu^+}(\Omega)$ is left-isolated. To prove that $\lambda_{1, \mu^+}(\Omega)$ is right-isolated, we argue by contradiction. We assume that there is a sequence of eigenvalues $\left\{\lambda_{m, \mu^+}\right\}$ such that $\lambda_{m, \mu^+} \searrow \lambda_{1, \mu^+}$ as $m \rightarrow \infty$ and $\lambda_{m, \mu^+} \neq \lambda_{1, \mu^+}$. Let $u_m$ be an eigenfunction associated to $\lambda_{m, \mu^+}$. Without loss of generality, we may assume that $\left\|u_m\right\|_{L^p(\Omega)}=1$. Then we have 
     \begin{equation*}
         \lambda_{m, \mu^+}=\int_{[0,1]}[u_m]_{s,p}^{p} d \mu^{+}(s).
     \end{equation*}
    We also define $\Omega_{m_{-}}=\left\{x \in \Omega: u_m(x) < 0\right\}$. Then by \eqref{need for sign changing} we have
\begin{equation}\label{need for sign changing 1}
\left|\Omega_{m_{-}}\right| \geq\left(\frac{C(N, s_{\sharp},p)}{\lambda_{m, \mu^+}}\right)^{\frac{N}{p s_\sharp}}=M \text { (say). }
\end{equation}
     
     Moreover, the sequence $\left\{u_m\right\}$ is bounded in $X_{p}(\Omega)$ and therefore we can extract a subsequence (still denoted by $\left\{u_m\right\}$) such that
\begin{equation*}
    u_m \rightharpoonup u \text { weakly in }X_{p}(\Omega), \quad u_m \rightarrow u \text { strongly in } L^p(\Omega).
\end{equation*}
Consequently, we have $\left\|u\right\|_{L^p(\Omega)}=1$, and applying Fatou’s lemma, we get
\begin{equation*}
    \frac{\int_{[0,1]}[u]_{s,p}^{p} d \mu^{+}(s)}{\int_{\Omega}|u(x)|^p d x} \leq \lim _{m \rightarrow \infty} \lambda_{m, \mu^+}=\lambda_{1, \mu^+}.
\end{equation*}
Hence, we conclude that $u$ is an eigenfunction associated to $\lambda_{1, \mu^+}(\Omega)$. Therefore, $u$ has a constant sign. Without loss of generality, let us assume that $u>0$ in $\Omega$. 

On the other hand, as $u_m \rightarrow u$ a.e. in $\Omega,$ by the Egorov’s theorem, for any $\delta>0$ there exists a subset $A_\delta$ of $\Omega$ such that $\left|A_\delta\right|<\delta$ and $u_m \rightarrow u>0$ uniformly in $\Omega \backslash A_\delta$. Since $u>0$ in $\Omega$ and $\Omega/A_\delta$ is compact, there exists $\epsilon>0$ such that $u \geq \epsilon$ in $\Omega/A_\delta$. Moreover, since $u_n \rightarrow u$ uniformly in $\Omega/A_\delta$, there exists $m_0 \in \mathbb{N}$ such that for all $x \in \Omega/A_\delta$,
$$
\left|u_m(x)-u(x)\right|<\frac{\epsilon}{2}, \quad \forall m \geq m_0 .
$$
This implies that
$$
u_{m_0}(x)>u(x)-\frac{\epsilon}{2} \geq \epsilon-\frac{\epsilon}{2}=\frac{\epsilon}{2}, \quad \forall x \in \Omega/A_\delta .
$$
As a consequence we have $\Omega/A_\delta \subset \Omega_{m_{0_{+}}}$. This again implies that $\Omega_{m_{0_{-}}} \subset  A_\delta$. Thus, we get $|\Omega_{m_{0_{-}}}| \leq|A_\delta|<\delta$. Since $\delta>0$ is arbitrary, choosing $\delta=M$, we get a contradiction to \eqref{need for sign changing 1}. This completes the proof.
\end{proof}
The following proposition is a technical result that will be used in the next result.
\begin{prop}\label{S+ operator} Let $\Omega \subset \mathbb{R}^{N}$ be a bounded domain with Lipschitz boundary, and let $\mu=\mu^+$ satisfy \eqref{measure1}. Let $s_{\sharp}$ be as in \eqref{measure4}, and assume $1 < p < N/s_{\sharp}$.
    Let $\{u_{k}\}$ be a sequence in $X_{p}(\Omega)$ such that $u_{k} \rightharpoonup u$ in $X_{p}(\Omega)$  and
$$
\limsup _{k\to \infty}\left\langle\mathcal{L}_{\mu,p}^+u_{k}, u_{k}-u\right\rangle = 0.
$$
Then, $u_{k} \rightarrow u$ in $X_{p}(\Omega).$ 
\end{prop}
\begin{proof}
  Let $\{u_{k}\}$ be a sequence in $X_{p}(\Omega)$ such that $u_k \rightharpoonup u$ in $X_{p}(\Omega)$  and 
  \begin{equation*}
      \limsup _{k\to \infty}\left\langle \mathcal{L}_{\mu,p}^+ u_k, u_k-u\right\rangle = 0.
  \end{equation*}
Since $u_k \rightharpoonup u$ in $X_{p}(\Omega),$ we have 
\begin{equation*}
    \lim_{k\to \infty}\left\langle \mathcal{L}_{\mu,p}^+ u, u_k-u\right\rangle = 0.
\end{equation*}
Then, applying H\"older's inequality and using Theorem \ref{embedding lem} , we get, for all $k \in \mathbb{N},$ that
\begin{align*}
 \mathbf{o}(1) &= \left\langle \mathcal{L}_{\mu,p}^+ u_k, u_k-u\right\rangle+\left\langle \mathcal{L}_{\mu,p}^+ u, u-u_k\right\rangle\\
&  = \left\langle \mathcal{L}_{\mu,p}^+ u_k, u_k\right\rangle-\left\langle \mathcal{L}_{\mu,p}^+ u_k, u\right\rangle-\left\langle \mathcal{L}_{\mu,p}^+ u, u_k\right\rangle+\left\langle \mathcal{L}_{\mu,p}^+ u, u\right\rangle\\
&\geq \left\|u_k\right\|^p_{X_p(\Omega)}-\left\|u_k\right\|^{p-1}_{X_p(\Omega)}\|u\|_{X_p(\Omega)}-\left\|u_k\right\|_{X_p(\Omega)}\|u\|^{p-1}_{X_p(\Omega)}+\|u\|^p_{X_p(\Omega)} \\
& = \left(\left\|u_k\right\|^{p-1}_{X_p(\Omega)}-\left\|u\right\|^{p-1}_{X_p(\Omega)}\right)\left(\|u_k\|_{X_p(\Omega)}-\|u\|_{X_p(\Omega)}\right).
\end{align*}

Thus, we have $\left\|u_k\right\|_{X_p(\Omega)} \rightarrow\|u\|_{X_p(\Omega)}$. Therefore,  using the fact that the space $X_{p}(\Omega)$ is  uniformly convex, we conclude that  $u_k \rightarrow u$ in $X_{p}(\Omega)$. 
\end{proof}

\begin{thm} \label{closep} Let $\Omega \subset \mathbb{R}^{N}$ be a bounded domain with Lipschitz boundary, and let $\mu=\mu^+$ satisfy \eqref{measure1}. Let $s_{\sharp}$ be as in \eqref{measure4}, and assume $1 < p < N/s_{\sharp}$.
    The set of all $(s,\mu^+)$-eigenvalues, that is, the spectrum $\sigma(s, \mu^+)$ to \eqref{Eig problem} is closed.
\end{thm}
\begin{proof} 
 Let $\lambda \in \overline{\sigma(s, \mu^+)}.$  Then, there exists a sequence of eigenvalues $\{\lambda_{k, \mu^+}\}$ of the problem \eqref{Eig problem} such that $\lambda_{k, \mu^+} \rightarrow \lambda$. Then, $\{\lambda_{k, \mu^+}\}$ is a bounded sequence. For each $k \in \mathbb{N}$, let $u_k$ be an eigenfunction corresponding to the eigenvalue $\lambda_{k, \mu^+}$ such that $\int_{\Omega} \left|u_k\right|^p d x=1$. Then we have 
    \begin{equation}\label{eq for closed}
        \left\langle \mathfrak{L}_{p}^{\prime}(u_k), v\right\rangle=\lambda_{k, \mu^+} \int_{\Omega} |u_k|^{p-2} u_k v d x ,
    \end{equation}
    for all $v \in X_{p}(\Omega)$.
  
 By taking $u_k$ as the test function for the eigenpair $\left(\lambda_{k, \mu^+}, u_k\right)$ in the weak formulation \eqref{eq for closed}, we have
\begin{equation*}
  \lambda_{k, \mu^+} =\left\|u_k\right\|^p_{X_p(\Omega)}. 
\end{equation*}
Hence, the sequence $\{u_k\}$ is bounded in $X_{p}(\Omega).$  Since $X_{p}(\Omega)$ is a reflexive Banach space, there exists a subsequence, still denoted by $\{u_k\}$, such that $u_k \rightharpoonup u$ weakly in $X_{p}(\Omega)$. Then, by Theorem \ref{embedding lem}, we conclude that $u_k \rightarrow u$, up to a subsequence, in $L^q(\Omega)$ for $1 \leq q<p_{s_\sharp}^*$. Therefore, testing \eqref{eq for closed} with $v=u_k-u$, using the Hölder's inequality, we get
 \begin{align}\label{eq for closed e1}
 \begin{split}
       \left\langle \mathfrak{L}_{p}^{\prime}(u_k), u_k-u\right\rangle&=\lambda_{k, \mu^+} \int_{\Omega} |u_k|^{p-2} u_k (u_k-u) d x \\
        &\leq \lambda_{k, \mu^+} \left\|u_k-u\right\|_{L^p(\Omega)}\left\|u_k\right\|_{L^p(\Omega)}^{p-1} \to 0, \text{ as } k\to \infty.
 \end{split}
    \end{align}
Then, by applying the weak convergence $u_k \rightharpoonup u$ in $X_{p}(\Omega)$ and Proposition \ref{S+ operator}, we obtain $u_k \rightarrow u$ in $X_p(\Omega).$
Therefore, passing to the limit under the integral sign in \eqref{eq for closed}, we obtain
\begin{equation*}
        \left\langle \mathfrak{L}_{p}^{\prime}(u), v\right\rangle=\lambda \int_{\Omega} |u|^{p-2} u v d x ,
    \end{equation*}
for all $v \in X_{p}(\Omega)$. Moreover, $\|u\|_{L^p(\Omega)}=\lim _{k \rightarrow \infty}\left\|u_{k}\right\|_{L^p(\Omega)}=1$. Hence, $(\lambda, u)$ is an eigenpair to \eqref{Eig problem P}. This concludes the proof. \end{proof}

Finally, for convenience, we combine the above results and state the main theorem of this section.
\begin{thm}\label{first eigenvalue exist +}
  Let $\Omega \subset \mathbb{R}^{N}$ be a bounded domain with Lipschitz boundary, and let $\mu=\mu^+$ satisfy \eqref{measure1}. Let $s_{\sharp}$ be as in \eqref{measure4}, and assume $1 < p < N/s_{\sharp}$. Then the statements below concerning the eigenvalues and eigenfunctions of problem \eqref{Eig problem} associated with $\mathcal{L}_{\mu, p}^+$ hold.
\begin{enumerate}[(i)]
    \item The first eigenvalue $\lambda_{1, \mu^+}(\Omega)$ is given by
\begin{equation}\label{first eigenvalue+}
\lambda_{1, \mu^+}(\Omega):=\inf_{u \in X_{p}(\Omega) \backslash\{0\}} \frac{\int_{[0,1]}[u]_{s,p}^{p} d \mu^{+}(s)}{\int_{\Omega}|u|^{p} d x} . 
\end{equation}
\item There exists a  function $e_{1, \mu^+} \in X_p(\Omega)$, an eigenfunction corresponding to the eigenvalue $\lambda_{1, \mu^+}(\Omega) $ which attains the minimum in \eqref{first eigenvalue+}.

\item The set of eigenvalues of the problem \eqref{Eig problem} consists of a sequence $(\lambda_{n, \mu^+})$ with
  \begin{align}
      0<\lambda_{1, \mu^+}<\lambda_{2, \mu^+} \leq \ldots \leq \lambda_{n, \mu^+} \leq \lambda_{n+1, \mu^+} \leq \ldots~\text{and }&
      \lambda_{n, \mu^+} \rightarrow \infty\,\,\,\text{as}\,\,n \rightarrow \infty.
  \end{align}
  
  \item  Every eigenfunction corresponding to the eigenvalue $\lambda_{1, \mu^+}(\Omega)$ in \eqref{first eigenvalue+} does not change sign and $\lambda_{1, \mu^+}(\Omega)$ is simple.
  \item The set of all $(s,\mu^+)$-eigenvalues, that is the spectrum $\sigma(s, \mu^+)$ to \eqref{Eig problem} is closed.
  \item Let $u \geq 0$ in $\Omega$ be an eigenfunction of \eqref{Eig problem} associated with an eigenvalue $\lambda > 0$. Then $u > 0$ in $\Omega$.

\item Let $v$ be an eigenfunction of \eqref{Eig problem} associated to an eigenvalue $\lambda>\lambda_{1, \mu^+}(\Omega)$. Then $v$ must be sign-changing.

\item Let $v$ be an eigenfunction of \eqref{Eig problem} associated to an eigenvalue $\lambda\neq \lambda_{1, \mu^+}(\Omega)$. Then there is a positive constant $C$ independent of $v$ such that  
\begin{equation*}
    \lambda \geq C(N, s_{\sharp},p)\left|\Omega_{+}\right|^{-\frac{ps_{\sharp}}{N}} \text { and  } \lambda \geq C(N, s_{\sharp},p)\left|\Omega_{-}\right|^{-\frac{ps_{\sharp}}{N}} ,
\end{equation*}
where $\Omega_{+}:=\{x \in \Omega:\, v>0\}$ and $\Omega_{-}:=\{x \in \Omega:\,v<0\}$.

\item   The first eigenvalue $\lambda_{1, \mu^+}$ of the problem \eqref{Eig problem} is isolated.

\item All eigenfunctions for positive eigenvalues $u \in X_p(\Omega)$ of \eqref{Eig problem} are globally bounded, that is, $u \in L^\infty(\mathbb{R}^N).$ 
\end{enumerate}
   \end{thm}

\section{Faber-Krahn inequality for nonlinear superposition operators} \label{sec7}

This section is devoted to the study of the shape optimization problem 
\begin{equation}\label{eq:shape-opt}
    \inf\{\lambda_{1, \mu^+}(\Omega) :\, |\Omega| = \rho\},
\end{equation}
where \( B \) denotes the Euclidean ball with volume \( \rho \), via the Faber--Krahn inequality for the operator \( \mathcal{L}_{\mu, p}^+ \). 
Since we are dealing with nonlinear superposition operators of mixed fractional order, it is necessary to employ a generalized form of the rearrangement inequality for the Sobolev spaces naturally associated with such operators. 
Accordingly, the proof of the Faber--Krahn inequality relies on an Almgren-Lieb type rearrangement result, whose proof follows the arguments of Theorem~A.1 in Frank and Seiringer~\cite{FS-2008}. 
We state the result below.

\begin{lem}\label{nonloc rear}
Let $\Omega$ be an open and bounded subset of $\mathbb{R}^{N}$. Assume that $\mu=\mu^{+}$ satisfies \eqref{measure1}. Let $s_{\sharp}$ be as in \eqref{measure4} and $1<p<N/s_{\sharp}$.
Then for any $u\in X_p(\Omega)$ we have
\begin{align*}
\begin{split}
 &\int_{(0,1)}c_{N,s,p}\iint_{\mathbb{R}^{2N}}  \frac{|u(x)-u(y)|^p}{|x-y|^{N+p s}} d x d yd\mu^{+}(s) \\
 & \quad \quad \quad \quad \quad \quad \quad \quad \quad \quad \quad \geq \int_{(0,1)}c_{N,s,p}\int_{\mathbb{R}^{2N} } \frac{\left|u^*(x)-u^*(y)\right|^p}{|x-y|^{N+p s}} d x d yd\mu^{+}(s),  
\end{split}
\end{align*}
where $u^*$ is a symmetric decreasing rearrangement of $u$.
If $p=1$, then equality holds iff $u$ is proportional to a non-negative function $v$ such that the level set $\{v>\tau\}$ is a ball for a.e. $\tau>0$. If $p>1$, then equality holds iff $u$ is proportional to a translate of a symmetric decreasing function.
\end{lem}
\begin{proof}
    First, we have the following representation
\begin{align*}
    &\int_{(0,1)}c_{N,s,p}\iint_{\mathbb{R}^{2N}} \frac{|u(x)-u(y)|^p}{|x-y|^{N+p s}} d x d yd\mu^{+}(s)\\
    &\quad \quad \quad \quad \quad \quad \quad \quad \quad \quad \quad \quad = \int_0^{\infty}  \left(\int_{(0,1)}\frac{c_{N,s,p}}{\Gamma((N+p s) / 2)} t^{(N+p s) / 2-1}d\mu^{+}(s)\right)K_t(u) d t,
\end{align*}
with
\begin{equation*}
    K_t(u)=\iint_{\mathbb{R}^{2N}} |u(x)-u(y)|^p e^{-t|x-y|^2}dxdy.
\end{equation*}
Then, applying Lemma A.2 from \cite{FS-2008} concludes the proof.
\end{proof}

The following theorem is the main result of this section.

\begin{thm}\label{thm Faber-Krahn inequality}
    (Faber-Krahn inequality for $\lambda_{\mu^{+}}(\Omega)$ ). Let $\Omega \subset \mathbb{R}^N$ be a bounded open set with boundary $\partial \Omega$ of class $C^1$. Assume that $\mu=\mu^{+}$ satisfies \eqref{measure1}. Let $s_{\sharp}$ be as in \eqref{measure4} and $1<p<N/s_{\sharp}$.  Let $\rho:=|\Omega| \in(0, \infty)$, and let $B$ be any Euclidean ball with volume $\rho$. Then,
\begin{equation}\label{Faber-Krahn inequality}
    \lambda_{1, \mu^+}(\Omega) \geq \lambda_{1, \mu^+}\left(B\right).
\end{equation}
Moreover, if the equality holds in \eqref{Faber-Krahn inequality}, then $\Omega$ is a ball.
\end{thm}

\begin{proof} We denote the Euclidean ball with centre at origin $0$ and volume $\rho$ by  $\widehat{B}$. Assume that $u_0 \in X_{p}(\Omega) \backslash\{0\}$ is the principal eigenfunction of $\mathcal{L}_{\mu, p}^{+}$ in $\Omega.$ Then, we denote the (decreasing) Schwarz symmetrization of $u_0$ by $u_0^*: \mathbb{R}^N \rightarrow \mathbb{R}.$
 Now, since $u_0 \in X_{p}(\Omega)$, it follows from the Polya-Szeg\"o theorem (see \cite{PS51}) that
\begin{equation}\label{eq7.2}
u_0^* \in X_p(\widehat{B}) \quad \text { and } \quad \int_{\widehat{B}}\left|\nabla u_0^*\right|^p d x \leq \int_{\Omega}|\nabla u|^p d x.
\end{equation}
Furthermore, by Lemma \ref{nonloc rear} we also have
\begin{align}\label{eq7.3}
\begin{split}
     &\int_{(0,1)}c_{N,s,p}\iint_{\mathbb{R}^{2 N}} \frac{\left|u_0^*(x)-u_0^*(y)\right|^p}{|x-y|^{N+sp}} d x d y d\mu^{+}(s)\\
     & \quad \quad \quad \quad \quad \quad \quad \quad \quad \quad \quad \leq \int_{(0,1)}c_{N,s,p}\iint_{\mathbb{R}^{2 N}} \frac{\left|u_0(x)-u_0(y)\right|^p}{|x-y|^{N+sp}} d x d yd\mu^{+}(s).
\end{split}
\end{align}

Combining  all these facts and the inequality above, we conclude that
\begin{align}\label{eq7.4}
\begin{split}
    \lambda_{1, \mu^+}(\Omega) & =\int_{[0,1]}[u_0]_{s,p}^p d \mu^{+}(s) \geq \int_{[0,1]}[u_0^*]_{s,p}^p d \mu^{+}(s) =\lambda_{1, \mu^+}(\widehat{B}).
\end{split}
\end{align}

From inequality \eqref{eq7.4} and the translation invariance of $\lambda_{1, \mu^+}(\Omega)$, we conclude that \eqref{Faber-Krahn inequality} holds for every Euclidean ball $B$ of volume $\rho$.

To finish the proof of this result, we assume  that
$$
\lambda_{1, \mu^+}(\Omega)=\lambda_{1, \mu^+}(B)
$$
for some  ball (and therefore, for all balls) $B$ with $\left|B\right|=\rho$. Thus, using  \eqref{eq7.4} we have
\begin{align}
\begin{split}
    & \int_{[0,1]}c_{N,s,p}\iint_{\mathbb{R}^{2 N}} \frac{\left|u_0(x)-u_0(y)\right|^p}{|x-y|^{N+sp}} d x d yd\mu^{+}(s)=\lambda_{1, \mu^+}(\Omega) \\
&=\lambda_{1, \mu^+}(\widehat{B})=\int_{[0,1]}c_{N,s,p}\iint_{\mathbb{R}^{2 N}} \frac{\left|u_0^*(x)-u_0^*(y)\right|^p}{|x-y|^{N+sp}} d x d yd\mu^{+}(s).
\end{split}
\end{align}
Particularly, from \eqref{eq7.2}-\eqref{eq7.3} and together with the fact that $\|u_0\|_{L^p(\Omega)}= \|u^*_0\|_{L^p(\widehat{B})},$ we get
\begin{equation*}
    \int_{(0,1)}c_{N,s,p}\iint_{\mathbb{R}^{2 N}} \frac{\left|u_0(x)-u_0(y)\right|^p}{|x-y|^{N+sp}} d x d y\mu^{+}(s)= \int_{(0,1)}c_{N,s,p}\iint_{\mathbb{R}^{2 N}} \frac{\left|u_0^*(x)-u_0^*(y)\right|^p}{|x-y|^{N+sp}} d x d y\mu^{+}(s).
\end{equation*}
Then, again by Lemma \ref{nonloc rear}, $u_{0}$ must be proportional to a translate of a symmetric decreasing function. This insures that the set
\begin{equation*}
    \Omega=\left\{x \in \mathbb{R}^N: u_0(x)>0\right\}
\end{equation*}
must be a ball (up to a set of zero Lebesgue measure). This completes the proof of the theorem.
\end{proof}

\section{Analysis of second eigenvalue of operator $\mathcal{L}_{\mu, p}^+$} \label{sec8}

The main aim of this section is to investigate the well-definedness  of the second eigenvalue of $\mathcal{L}_{\mu,p}^{+}.$ To begin this section, let us define
\begin{equation*}
    \Gamma_{1}(\Omega)=\left\{\phi: \mathbb{S}^{1} \rightarrow \mathcal{M}: \phi \text { is odd and continuous }\right\}
\end{equation*}
and
\begin{equation}\label{2 eigenvalue def}
\lambda_{2, \mu^+}(\Omega)=\inf _{\phi \in \Gamma_{1}(\Omega)} \max _{u \in \operatorname{Im}(\phi)}\|u\|_{X_p(\Omega)}^{p}, 
\end{equation}
where $\operatorname{Im}(\phi):=\phi(\mathbb{S}^1) \subset \mathcal{M},$ denotes the Image of $\phi,$ and $\mathcal{M}$ is defined by  \begin{equation}\label{the set M P11}
    \mathcal{M}:=\left\{u \in X_{p}(\Omega): \quad \|u\|_{L^{p}(\Omega)}=1\right\}. 
\end{equation}

Then, adopting the method of \cite{BP16}, we obtain the following result concerning the second eigenvalue of the operator $\mathcal{L}_{\mu,p}^{+}$.

\begin{thm}\label{properties of 2 eigenvalue}
    Let $\Omega$ be an open bounded subset of $\mathbb{R}^{N}$. Let $\mu^+$ satisfy \eqref{measure1}. Let $s_{\sharp}$ be as in \eqref{measure4} and $1<p<N/s_{\sharp}$. Let $\lambda_{2, \mu^+}(\Omega)$ be the positive number defined in \eqref{2 eigenvalue def}. Then the following statements hold.
     
(i) $\lambda_{2, \mu^+}(\Omega)$ is an eigenvalue of the operator $\mathcal{L}_{\mu,p}^+$.

(ii) $\lambda_{2, \mu^+}(\Omega)>\lambda_{1, \mu^+}(\Omega)$.

(iii) If $\lambda>\lambda_{1, \mu^+}(\Omega)$ is an eigenvalue of $\mathcal{L}_{\mu,p}^{+}$, then $\lambda \geq \lambda_{2, \mu^+}(\Omega)$.

(iv) Every eigenfunction $u \in \mathcal{M}$ associated to $\lambda_{2, \mu^+}(\Omega)$ has to change sign.
\end{thm}

\begin{proof}
\textit{(i)}
By Lemma \ref{PS cond gen}, the functional $\mathfrak{L}_{p}$ satisfies the Palais–Smale $(PS)$ condition on $\mathcal{M}$. Therefore, the application of \cite[Proposition 2.7]{Cuesta} proves the claim. 

\textit{(ii)} To show this  we use a contradiction argument. If possible, assume that 
$$
\lambda_{2, \mu^+}(\Omega)=\inf _{\phi \in \Gamma_{1}(\Omega)} \max _{u \in \operatorname{Im}(\phi)}\left(\int_{[0,1]}[u]_{s,p}^p d \mu^{+}(s)\right)=\lambda_{1}(\Omega)
$$
is true. Then, from the definition of $\lambda_{2, \mu^+}(\Omega)$, for each $m \in \mathbb{N}$, there exists $\phi_{m} \in \Gamma_{1}$ such that
\begin{equation}\label{lam 2 big lam 1}
\max _{u \in \phi_{m}\left(\mathbb{S}^{1}\right)}\int_{[0,1]}[u]_{s,p}^p d \mu^{+}(s) \leq \lambda_{1, \mu^+}(\Omega)+\frac{1}{m} .
\end{equation}

Let $e_{1, \mu^+}$ be the first eigenfunction to \eqref{Eig problem} corresponding to $\lambda_{1, \mu^+}$. By Theorem \ref{strict positive +}, we have either $e_{1, \mu^+}>0$ or $e_{1, \mu^+}<0$ in $\Omega$. Fix $\epsilon>0$, sufficiently small. Consider the following two disjoint neighborhoods of $e_{1, \mu^+}$
$$
B_{\epsilon}^{+}=\left\{u \in \mathcal{M}:\left\|u-e_{1, \mu^+}\right\|_{L^{p}(\Omega)}<\epsilon\right\} \quad \text { and } \quad B_{\epsilon}^{-}=\left\{u \in \mathcal{M}:\left\|u-\left(-e_{1, \mu^+}\right)\right\|_{L^{p}(\Omega)}<\epsilon\right\} .
$$

Note that $\phi_{m}\left(\mathbb{S}^{1}\right) \not \subset B_{\epsilon}^{+} \cup B_{\epsilon}^{-}$ due to the fact $\phi_{m} \in \Gamma_{1}(\Omega)$, implying that $\phi_{m}\left(\mathbb{S}^{1}\right)$ is symmetric and connected. Therefore, there exists $u_{m} \in \phi_{m}\left(\mathbb{S}^{1}\right) \backslash\left(B_{\epsilon}^{+} \cup B_{\epsilon}^{-}\right)$ for each $m \in \mathbb{N}$. Moreover, the sequence $\{u_{m}\}$ is bounded in $X_{p}(\Omega)$, thanks to \eqref{lam 2 big lam 1}. Therefore, there exists $v \in \mathcal{M}$ and a subsequence of $\{u_{m}\}$, (still denoted by $\{u_{m}\}$) such that $u_{m} \rightharpoonup v$ weakly in $X_{p}(\Omega)$ and $u_{m} \rightarrow v$ strongly in $L^{p}(\Omega)$. By the lower semicontinuity of the norm we have
\begin{align*}
    \int_{[0,1]}[v]_{s,p}^p d \mu^{+}(s)  &\leq \liminf _{k \rightarrow \infty}\left\|u_{k}\right\|_{X_p(\Omega)}^{p}=\liminf _{k \rightarrow \infty} \int_{[0,1]}[u_k]_{s,p}^p d\mu^{+}(s)=\lambda_{1, \mu^+}(\Omega),
\end{align*}
implying that $v \in \mathcal{M}$ is a global minimizer for $\mathfrak{L}_{p}$. Thus we get either $v=e_{1, \mu^+}$ or $v=-e_{1, \mu^+}$. Again, since $u_{m} \rightarrow v$ strongly in $L^{p}(\Omega)$, we have $v \in \mathcal{M} \backslash\left(B_{\epsilon}^{+} \cup B_{\epsilon}^{-}\right)$ giving us a contradiction. Hence, $\lambda_{2, \mu^+}(\Omega)>\lambda_{1, \mu^+}(\Omega)$.

\textit{(iii)} Let $(u, \lambda)$ be an eigenpair to the problem \eqref{Eig problem}, with $\lambda>\lambda_{1, \mu^+}(\Omega)$. Then, by Theorem \ref{result sign changing}, we conclude that $u$ needs  to change sign in $\Omega$, that is,  $u=u_{+}-u_{-}$ with $u_{+} \not \equiv 0$ and $u_{-} \not \equiv 0$, both being positive. Now, we test  the equation \eqref{weak solution of eig problem} for $(u, \lambda)$ with  $u_{+}$ and $u_{-}$ as test function to obtain
$$
\lambda \int_{\Omega} u_{+}^{p} d x=\int_{[0,1]} c_{N,s,p}\iint_{\mathbb{R}^{2N}} \frac{|u(x)-u(y)|^{p-2}(u(x)-u(y))}{|x-y|^{N+ps}}\left(u_{+}(x)-u_{+}(y)\right) d x d y d\mu^{+}(s),
$$
and
$$
-\lambda \int_{\Omega} u_{-}^{p} d x=\int_{[0,1]} c_{N,s,p}\iint_{\mathbb{R}^{2N}} \frac{|u(x)-u(y)|^{p-2}(u(x)-u(y))}{|x-y|^{N+ps}}\left(u_{-}(x)-u_{-}(y)\right) d x d y d\mu^{+}(s).
$$
Let us introduce the notations
$$
A:=A(x, y):=u_{+}(x)-u_{+}(y) \quad \text { and } \quad B:=B(x, y):=u_{-}(x)-u_{-}(y).
$$
Then we have
$$
A-B=u(x)-u(y).
$$
So, we can rewrite
$$
\lambda \int_{\Omega} u_{+}^{p} d x=\int_{[0,1]} c_{N,s,p}\iint_{\mathbb{R}^{2N}} \frac{(A-B)A}{|x-y|^{N+ps}}  d x d y d\mu^{+}(s),
$$
and
$$
-\lambda \int_{\Omega} u_{-}^{p} d x=\int_{[0,1]} c_{N,s,p}\iint_{\mathbb{R}^{2N}} \frac{(A-B)B}{|x-y|^{N+ps}}  d x d y d\mu^{+}(s).
$$

Let us take $\left(\omega_{1}, \omega_{2}\right) \in \mathbb{S}^{1}$. Multiplying the previous two identities by $\left|\omega_{1}\right|^{p}$ and $\left|\omega_{2}\right|^{p}$ respectively and subtracting them, we obtain
\begin{equation}\label{lam big lam 21}
\lambda=\frac{\int_{[0,1]} c_{N,s,p}\iint_{\mathbb{R}^{2N}} \frac{\left[\left|\omega_{1}\right|^{p}(A-B) A-\left|\omega_{2}\right|^{p}(A-B) B\right]}{|x-y|^{N+ps}} d x d yd\mu^{+}(s)}{\left|\omega_{1}\right|^{p} \int_{\Omega} u_{+}^{p}+\left|\omega_{2}\right|^{p} \int_{\Omega} u_{-}^{p} d x}. 
\end{equation}
Observe that we can write
\begin{align} \label{eq8.5}
    |\omega_{1}|^{p}(A-B) A-|\omega_{2}|^{p}(A-B) B&=|\omega_{1} A-\omega_{1} B|^{p-2}\left(\omega_{1} A-\omega_{1} B\right) \omega_{1} A \nonumber \\
    &-|\omega_{2} A-\omega_{2} B|^{p-2}\left(\omega_{2} A-\omega_{2} B\right) \omega_{2} B.
\end{align}
Now, let us recall the  following pointwise inequality (see \cite[Inequality $(4.7)$]{BP16})
\begin{align}\label{lam big lam 22}
\begin{split}
    |\omega_{1} A-\omega_{1} B|^{p-2}\left(\omega_{1} A-\omega_{1} B \right) \omega_{1} A & -|\omega_{2} A-\omega_{2} B|^{p-2}\left(\omega_{2} A-\omega_{2} B\right) \omega_{2} B \\
&\geq\left|\omega_{1} A-\omega_{2} B\right|^{p}.
\end{split}
\end{align}

In order to complete the proof, we define the following element of $\Gamma_{1}(\Omega)$
$$
f(\omega)=\frac{\omega_{1} u_{+}-\omega_{2} u_{-}}{\left(\left|\omega_{1}\right|^{p} \int_{\Omega} u_{+}^{p}+\left|\omega_{2}\right|^{p} \int_{\Omega} u_{-}^{p} d x\right)^{1 / p}}, \quad \omega=\left(\omega_{1}, \omega_{2}\right) \in \mathbb{S}^{1}.
$$
Then we have
$$
\int_{[0,1]} [f(\omega)]_{s,p}^pd\mu^{+}(s)=\frac{\int_{[0,1]} c_{N,s,p}\iint_{\mathbb{R}^{2N}} \frac{\left|\omega_{1} A-\omega_{2} B\right|^{p}}{|x-y|^{N+ps}} d x d y d\mu^{+}(s)}{\left|\omega_{1}\right|^{p} \int_{\Omega} u_{+}^{p}+\left|\omega_{2}\right|^{p} \int_{\Omega} u_{-}^{p} d x}.
$$
Next, by using the  inequalities \eqref{eq8.5} and \eqref{lam big lam 22}, and recalling the relation \eqref{lam big lam 21}, we get
$$
\int_{[0,1]} [f(\omega)]_{s,p}^pd\mu^{+}(s) \leq \lambda, \quad \text { for every } \omega \in \mathbb{S}^{1}.
$$
By appealing to the definition of $\lambda_{2, \mu^+}(\Omega)$ we get the desired conclusion that $\lambda \geq \lambda_{2, \mu^+}(\Omega).$ 

$(iv)$ The sign-changing property of eigenfunctions associated with $\lambda_{2, \mu^+}(\Omega)$ follows from Theorem \ref{result sign changing} as $\lambda_{1, \mu^+}(\Omega)< \lambda_{2, \mu^+}(\Omega)$ by part $(ii).$
Thus, the proof is complete.
\end{proof}

\section{Mountain pass characterization of the second eigenvalue of nonlinear superposition operators} \label{sec9}

This section is devoted to establishing a mountain pass characterization of the second eigenvalue introduced in the preceding section. We begin with a technical lemma, the proof of which is inspired by the combination of the arguments presented in \cite[Lemma 5.1]{BP16} and \cite[Lemma 5.2]{BF13ii}. 

\begin{lem}\label{M-P car tech lem}
   Let $\Omega \subset \mathbb{R}^{N}$ be a bounded domain and $1<p<\infty$.  Let $\mu^+$ satisfy \eqref{measure1}.  For every $u \in \mathcal{M},$ we set
\begin{equation*}
    A(x, y)=u_{+}(x)-u_{+}(y) \quad \text { and } \quad B(x, y)=u_{-}(x)-u_{-}(y).
\end{equation*}
Define the continuous curve on $\mathcal{M}$ as
\begin{equation*}
    \gamma_t=\frac{u_{+}-\cos (\pi t) u_{-}}{\left\|u_{+}-\cos (\pi t) u_{-}\right\|_{L^p(\Omega)}}, \quad t \in\left[0, \frac{1}{2}\right].
\end{equation*}
Assume the following conditions 

 \begin{align} \label{eq9.1}
    &\left\|u_{-}\right\|_{L^p(\Omega)}^p \int_{(0,1)}c_{N,s,p} \iint_{\mathbb{R}^{2N}}\frac{|A-B|^{p-2}(A-B) A}{|x-y|^{N+s p}} d x d yd\mu^+(s)  \nonumber \\
    &\quad \quad \quad \quad +\left\|u_{+}\right\|_{L^p(\Omega)}^p \int_{(0,1)}c_{N,s,p} \iint_{\mathbb{R}^{2N}}  \frac{|A-B|^{p-2}(A-B) B}{|x-y|^{N+s p}} d x d yd\mu^+(s) \leq 0,
\end{align}
and 
\begin{align} \label{eq9.2}
    &\left\|u_{-}\right\|_{L^p(\Omega)}^p\int_{\Omega}|\nabla u_{+}|^p d x- \left\|u_{+}\right\|_{L^p(\Omega)}^p\int_{\Omega}|\nabla u_{-}|^p d x\leq 0
\end{align}
hold. Then we have
\begin{equation*}
    \left\|\gamma_t\right\|_{X_p(\Omega)} \leq\|u\|_{X_p(\Omega)}, \quad t \in\left[0, \frac{1}{2}\right] .
\end{equation*}
\end{lem} 

\begin{proof}
Observe that, since $u_{+}$ and $u_{-}$ have disjoint supports, we have
\begin{align}\label{M-P tech 1}
    \left\|\gamma_t\right\|_{X_p(\Omega)}^p&=\frac{\|u_{+}-cos(\pi t) u_{-}\|_{X_p(\Omega)}^p}{\|u_{+}-cos(\pi t) u_{-}\|_{L^p(\Omega)}^p}=\frac{\int_{[0,1]}[u_{+}-cos(\pi t) u_{-}]_{s,p}^pd\mu^+(s)}{\int_{\Omega}|u_{+}-cos(\pi t) u_{-}|^pdx} \nonumber\\
    &=\frac{\mu^+(1)\int_{\Omega}|\nabla(u_{+}-cos(\pi t) u_{-})|^pdx}{\int_{\Omega}|u_{+}-cos(\pi t) u_{-}|^pdx}\\
    &+\frac{\int_{(0,1)}c_{N,s,p}\iint_{\mathbb{R}^{2N}}\frac{|(u_{+}-cos(\pi t) u_{-})(x)-(u_{+}-cos(\pi t) u_{-})(y)|^p}{|x-y|^{N+sp}}dxdyd\mu^+(s)}{\int_{\Omega}|u_{+}-cos(\pi t) u_{-}|^pdx} \nonumber\\
    &+\frac{\mu^+(0)\int_{\Omega}|u_{+}-cos(\pi t) u_{-}|^pdx}{\int_{\Omega}|u_{+}-cos(\pi t) u_{-}|^pdx} \nonumber\\
     &=\frac{\mu^+(1)\left(\int_{\Omega}|\nabla u_{+}|^pdx+|cos(\pi t)|^p \int_{\Omega}|\nabla u_{-}|^pdx\right)}{\int_{\Omega} u_{+}^pdx+|cos(\pi t)|^p \int_{\Omega}u_{-}^pdx} \nonumber\\
    &+\frac{\int_{(0,1)}c_{N,s,p}\iint_{\mathbb{R}^{2N}}\frac{|A-cos(\pi t) B|^p}{|x-y|^{N+sp}}dxdyd\mu^+(s)}{\int_{\Omega} u_{+}^pdx+|cos(\pi t)|^p \int_{\Omega}u_{-}^pdx}+\mu^+(0).\nonumber
\end{align}
Since, by definition, $A \cdot B \leq 0$, we get, using Lemma \ref{tech lem foe lam 2}, that 
\begin{align*}
    |A-cos(\pi t) B|^p\leq |A-B|^{p-2}(A-B)A-|A-B|^{p-2}(A-B)B|cos(\pi t)|^p. 
\end{align*}
Substituting this into \eqref{M-P tech 1}, we obtain
\begin{align}\label{M-P tech 2}
    \left\|\gamma_t\right\|_{X_p(\Omega)}^p&\leq\mu^+(1)\frac{\int_{\Omega}|\nabla u_{+}|^pdx+|cos(\pi t)|^p \int_{\Omega}|\nabla u_{-}|^pdx}{\int_{\Omega} u_{+}^pdx+|cos(\pi t)|^p \int_{\Omega}u_{-}^pdx} \nonumber \\
    &+\frac{\int_{(0,1)}c_{N,s,p}\iint_{\mathbb{R}^{2N}}\frac{|A-B|^{p-2}(A-B)A}{|x-y|^{N+sp}}dxdyd\mu^+(s)}{\int_{\Omega} u_{+}^pdx+|cos(\pi t)|^p \int_{\Omega}u_{-}^pdx}\\
    &-|cos(\pi t)|^p\frac{\int_{(0,1)}c_{N,s,p}\iint_{\mathbb{R}^{2N}}\frac{|A-B|^{p-2}(A-B)B}{|x-y|^{N+sp}}dxdyd\mu^+(s)}{\int_{\Omega} u_{+}^pdx+|cos(\pi t)|^p \int_{\Omega}u_{-}^pdx}+\mu^+(0)\nonumber
\end{align}
for every $t \in[0,1 / 2].$
Let us now denote
\begin{equation*}
    I_1:=\frac{\int_{\Omega}|\nabla u_{+}|^pdx+|cos(\pi t)|^p \int_{\Omega}|\nabla u_{-}|^pdx}{\int_{\Omega} u_{+}^pdx+|cos(\pi t)|^p \int_{\Omega}u_{-}^pdx},
\end{equation*}
and 
\begin{align*}
    I_2:=&\frac{\int_{(0,1)}c_{N,s,p}\iint_{\mathbb{R}^{2N}}\frac{|A-B|^{p-2}(A-B)A}{|x-y|^{N+sp}}dxdyd\mu^+(s)}{\int_{\Omega} u_{+}^pdx+|cos(\pi t)|^p \int_{\Omega}u_{-}^pdx}\\
    &-|cos(\pi t)|^p\frac{\int_{(0,1)}c_{N,s,p}\iint_{\mathbb{R}^{2N}}\frac{|A-B|^{p-2}(A-B)B}{|x-y|^{N+sp}}dxdyd\mu^+(s)}{\int_{\Omega} u_{+}^pdx+|cos(\pi t)|^p \int_{\Omega}u_{-}^pdx}.
\end{align*}
We define the following functions
\begin{equation*}
    g(\xi)=\frac{a-\xi b}{c+\xi d} \quad \text{ and } \quad h(\xi)=\frac{e^2+\xi f^2}{k^2+\xi m^2} \quad \text{ for } \quad \xi\in [0,1],
\end{equation*}
where $a,b,e,f \in \mathbb{R}$ and $c,d,k,m\geq 0$ such that $c+d>0$ and $k^2+m^2>0$. 
Observe that, if we set
$$\xi=|cos(\pi t)|^p, \quad \quad c=\int_{\Omega} u_{+}^pdx, \quad d=\int_{\Omega} u_{-}^pdx,$$
$$a=\int_{(0,1)}c_{N,s,p}\iint_{\mathbb{R}^{2N}}\frac{|A-B|^{p-2}(A-B)A}{|x-y|^{N+sp}}dxdyd\mu^+(s), $$
and 
$$b=\int_{(0,1)}c_{N,s,p}\iint_{\mathbb{R}^{2N}}\frac{|A-B|^{p-2}(A-B)B}{|x-y|^{N+sp}}dxdyd\mu^+(s),$$
then $I_2$ coincides with a function $g$. Similarly, taking 
$$\xi=|cos(\pi t)|^p, \ e^2=\int_{\Omega} |\nabla u_{+}|^pdx, \quad f^2=\int_{\Omega} |\nabla u_{-}|^pdx, \quad k^2=\int_{\Omega} u_{+}^pdx, \quad m^2=\int_{\Omega} u_{-}^pdx,$$
we can identify $I_1$ with a function $h$.
Then, in order to get the conclusion, it suffices to show that the functions $g$ and $h$ are monotone increasing. But the function $g$ is monotone increasing if and only if $c b+d a \leq 0$, and the function $h$ is monotone increasing if and only if $e^2m^2-k^2f^2\leq 0$, which are guaranteed by conditions \eqref{eq9.1} and \eqref{eq9.2}, respectively. Using this and the fact that $u \in \mathcal{M}$ has a unit $L^p$ norm along with the fact that $u_+$ and $u_-$ have disjoint supports, we get from \eqref{M-P tech 2} that
\begin{align*}
    \left\|\gamma_t\right\|_{X_p(\Omega)}^p&\leq\mu^+(1)\int_{\Omega}|\nabla u|^pdx+\int_{(0,1)}c_{N,s,p}\iint_{\mathbb{R}^{2N}}\frac{|A-B|^{p}}{|x-y|^{N+sp}}dxdyd\mu^+(s)+\mu^+(0)\\
    &=\|u\|_{X_p(\Omega)}^p
\end{align*}
for every $t \in[0,1 / 2].$ This concludes the proof. 
\end{proof}

\begin{remark}\label{M-P car tech rem}
    Let $u \in \mathcal{M}$ be a function that does not satisfy conditions \eqref{eq9.1} and \eqref{eq9.2}, that is, 
\begin{align*}
    &\left\|u_{-}\right\|_{L^p(\Omega)}^p \int_{(0,1)}c_{N,s,p} \iint_{\mathbb{R}^{2N}}\frac{|A-B|^{p-2}(A-B) A}{|x-y|^{N+s p}} d x d yd\mu^+(s)\\
    &\quad \quad \quad \quad +\left\|u_{+}\right\|_{L^p(\Omega)}^p \int_{(0,1)}c_{N,s,p} \iint_{\mathbb{R}^{2N}}  \frac{|A-B|^{p-2}(A-B) B}{|x-y|^{N+s p}} d x d yd\mu^+(s) > 0
\end{align*}
and 
\begin{align*} 
    &\left\|u_{-}\right\|_{L^p(\Omega)}^p\int_{\Omega}|\nabla u_{+}|^p d x- \left\|u_{+}\right\|_{L^p(\Omega)}^p\int_{\Omega}|\nabla u_{-}|^p d x> 0.
\end{align*}
Then the function $v=-u \in \mathcal{M}$ satisfies conditions \eqref{eq9.1} and \eqref{eq9.2}. 
\end{remark} 

Let us define
\begin{equation*}
    \Gamma\left(e_{1, \mu^+},-e_{1, \mu^+}\right)=\left\{\gamma \in C\left([0,1] ; \mathcal{M}\right): \gamma_0=e_{1, \mu^+}, \gamma_1=-e_{1, \mu^+}\right\},
\end{equation*}
the set of continuous curves on $\mathcal{M}$ connecting the two solutions $e_{1, \mu^+}$ and $-e_{1, \mu^+}$ of \eqref{first eigenvalue+}. We have the following characterization for $\lambda_2(\Omega)$.

\begin{thm}
    (Mountain pass characterization).  Let $\Omega \subset \mathbb{R}^N$ be an open and bounded set.  Let $\mu^+$ satisfy \eqref{measure1} and   let $1<p<\infty$.  Then we have
\begin{equation*}
    \lambda_{2, \mu^+}(\Omega)=\inf _{\gamma \in \Gamma\left(e_{1, \mu^+},-e_{1, \mu^+}\right)} \max _{u \in \operatorname{Im}(\gamma)}\|u\|_{X_p(\Omega)}^p.
\end{equation*}
\end{thm} 
\begin{proof}
    Observe that, for every $\gamma \in \Gamma(e_{1, \mu^+},-e_{1, \mu^+})$, the closed path on $\mathcal{M}$ obtained by gluing $\gamma$ and $-\gamma$ can be identified with the image of some odd continuous mapping $\phi$ from $\mathbb{S}^1$ to $\mathcal{M}$. Therefore, by definition of $\lambda_{2, \mu^+}(\Omega)$ we have
\begin{equation*}
    \lambda_{2, \mu^+}(\Omega)=\inf _{\phi \in \Gamma_{1}(\Omega)} \max _{u \in \operatorname{Im}(\phi)}\|u\|_{X_p(\Omega)}^{p} \leq \max _{u \in \operatorname{Im}(\phi)}\|u\|_{X_p(\Omega)}^p=\max _{u \in \operatorname{Im}(\gamma)}\|u\|_{X_p(\Omega)}^p.
\end{equation*}
By taking the infimum among all admissible paths $\gamma$, we obtain
\begin{equation*}
    \lambda_{2, \mu^+}(\Omega)\leq \inf _{\gamma \in \Gamma\left(e_{1, \mu^+},-e_{1, \mu^+}\right)}\max _{u \in \operatorname{Im}(\gamma)}\|u\|_{X_p(\Omega)}^p.
\end{equation*}
Let us now prove the reverse inequality. For every $n \in \mathbb{N}$, we take $\phi_n \in \Gamma_1(\Omega)$ such that
\begin{equation}\label{M-P car pr 1}
    \max _{u \in \operatorname{Im}\left(\phi_n\right)}\|u\|_{X_p(\Omega)}^p \leq \lambda_{2, \mu^+}(\Omega)+\frac{1}{n} .
\end{equation}
Let us pick up a function $u_n \in \operatorname{Im}(\phi_n)$ such that the hypotheses \eqref{eq9.1} and \eqref{eq9.2} of Lemma \ref{M-P car tech lem} are satisfied. This choice is always possible. Indeed, since $\phi_n$ is odd, the set $\operatorname{Im}\left(\phi_n\right)$ is symmetric with respect to the origin, i.e., if $v \in \operatorname{Im}\left(\phi_n\right)$, then $-v \in \operatorname{Im}\left(\phi_n\right)$ as well. Then the existence of such a $u_n$ follows from Remark \ref{M-P car tech rem}. Consequently, applying Lemma \ref{M-P car tech lem} and \eqref{M-P car pr 1}, we conclude that 
\begin{equation}\label{M-P car pr 2}
    \left\|\gamma_{n, t}\right\|_{X_p(\Omega)}^p \leq \lambda_{2, \mu^+}(\Omega)+\frac{1}{n}, \quad 0 \leq t \leq \frac{1}{2},
\end{equation}
where the curve $\gamma_{n, t}$ is given by
\begin{equation*}
    \gamma_{n, t}=\frac{\left(u_n\right)_{+}-\cos (\pi t)\left(u_n\right)_{-}}{\left\|\left(u_n\right)_{+}-\cos (\pi t)\left(u_n\right)_{-}\right\|_{L^p(\Omega)}}, \quad 0 \leq t \leq \frac{1}{2}.
\end{equation*}
Observe that the curve $\gamma_n$ connects $u_n$ to its $L^p$-renormalized positive part.

Now, we aim to connect the function $\frac{\left(u_n\right)_{+}}{\left\|\left(u_n\right)_{+}\right\|_{L^p(\Omega)}}$ to the first eigenfunction $e_{1, \mu^+}$. For this, we consider the curve
\begin{equation*}
  \sigma_{n, t}=\left((1-t) \frac{\left(u_n\right)_{+}^p}{\left\|\left(u_n\right)_{+}\right\|_{L^p(\Omega)}}+t e_{1, \mu^+}^p\right)^{\frac{1}{p}}, \quad t \in[0,1],  
\end{equation*}
along which our energy functional is convex (see the proof of Theorem \ref{result sign changing}), i.e.
\begin{equation*}
   \left\|\sigma_{n, t}\right\|_{X_p(\Omega)}^p \leq(1-t) \frac{\left\|\left(u_n\right)_{+}\right\|_{X_p(\Omega)}^p}{\left\|\left(u_n\right)_{+}\right\|_{L^p(\Omega)}^p}+t\left\|e_{1, \mu^+}\right\|_{X_p(\Omega)}^p . 
\end{equation*}
In particular, it follows from \eqref{M-P car pr 2} that
\begin{equation*}
    \left\|\sigma_{n, t}\right\|_{X_p(\Omega)}^p \leq \lambda_{2, \mu^+}(\Omega)+\frac{1}{n}, \quad t \in[0,1].
\end{equation*}
Now, gluing together $\gamma_n$ and $\sigma_n$ we obtain the new curve 
\begin{equation*}
    \widetilde{\gamma}_{n, t}= \begin{cases}\gamma_{n, t}, & t \in[0,1 / 2], \\ \sigma_{n,(2 t-1)}, & t \in[1 / 2,1],\end{cases}
\end{equation*}
which connects $u_n$ to $e_{1, \mu^+}$ and on which the energy is always less than $\lambda_{2, \mu^+}(\Omega)+1 / n$. 

Finally, gluing together the three paths $\widetilde{\gamma}_n$, $-\widetilde{\gamma}_n$ and $\phi_n$, using the fact that the energy functional is
even (therefore, the previous estimate still holds true on this path), we get a continuous curve $\eta_n \in \Gamma(e_{1, \mu^+},-e_{1, \mu^+})$ such that
\begin{equation*}
    \max _{t \in[0,1]}\left\|\eta_{n, t}\right\|_{X_p(\Omega)}^p \leq \lambda_{2, \mu^+}(\Omega)+\frac{1}{n}, \quad n \in \mathbb{N} .
\end{equation*}
By taking the infimum over $\Gamma(e_{1, \mu^+},-e_{1, \mu^+})$, we then get
\begin{equation*}
    \inf _{\gamma \in \Gamma\left(e_{1, \mu^+},-e_{1, \mu^+}\right)} \max _{u \in \operatorname{Im}(\gamma)}\|u\|_{X_p(\Omega)}^p \leq \lambda_{2, \mu^+}(\Omega)+\frac{1}{n} .
\end{equation*}
Passing to the limit as $n$ goes to $\infty$, we obtain the desired conclusion.
\end{proof}


\section*{Conflict of interest statement}
On behalf of all authors, the corresponding author states that there is no conflict of interest.

\section*{Data availability statement}
Data sharing is not applicable to this article as no datasets were generated or analysed during the current study.

\section*{Acknowledgement}
YA is supported by the Bolashak Government Scholarship of the Republic of Kazakhstan. This work was completed while SG was visiting the Ghent Analysis \& PDE Center, Ghent University, Belgium. SG gratefully acknowledges the financial support for this research work under ARG-MATRICS, grant No: ANRF/ARGM/2025/001570/MTR, Anusandhan National Research Foundation (ANRF), Government of India.  YA, VK, and MR are supported by the FWO Odysseus 1 grant G.0H94.18N: Analysis and Partial Differential Equations and the Methusalem program of the Ghent University Special Research Fund (BOF) (Grant number 01M01021). VK and MR are also supported by FWO Senior Research Grant G011522N.

\end{document}